\documentclass[11pt,a4paper]{amsart}
\usepackage[all]{xy}
\SelectTips{cm}{}
\usepackage{ifthen}
\usepackage{amssymb, graphics}
\usepackage[dvips]{graphicx}
\usepackage{german}

\calclayout
\makeatletter
\def\serieslogo@{}
\def\@setcopyright{}
\makeatother

\title{Realisability and Localisation}

\author{Birgit Huber}

\newtheorem{lem}{Lemma}[section]
\newtheorem{prop}[lem]{Proposition} \newtheorem{cor}[lem]{Corollary}
\newtheorem{thm}[lem]{Theorem}

\newtheorem{construction}[equation]{Construction}
\theoremstyle{remark}
\newtheorem*{Rem}{Remark}

\theoremstyle{definition}
\newtheorem{exm}[lem]{Example}
\newtheorem{defn}[lem]{Definition}

\newtheorem{rem}[lem]{Remark}
\numberwithin{equation}{section}

\newcommand{\smatrix}[1]{\left[\begin{smallmatrix}#1\end{smallmatrix}\right]}

\renewcommand{\mod}{\operatorname{mod}\nolimits}

\newcommand{\id}{\operatorname{id}\nolimits}
\newcommand{\Tr}{\operatorname{tr}\nolimits}

\newcommand{\Mod}{\operatorname{Mod}\nolimits}
\newcommand{\Modgr}{\operatorname{Mod_{gr}}\nolimits}
\newcommand{\Moddg}{\operatorname{Mod_{dg}}\nolimits}

\newcommand{\Hom}{\operatorname{Hom}\nolimits}

\newcommand{\Loc}{\operatorname{Loc}\nolimits}
\newcommand{\END}{\operatorname{\mathcal E\!\!\:\mathit n\mathit d}\nolimits}
\newcommand{\HOM}{\operatorname{\mathcal H\!\!\:\mathit o\mathit m}\nolimits}

\renewcommand{\Im}{\operatorname{Im}\nolimits}
\newcommand{\Ker}{\operatorname{Ker}\nolimits}
\newcommand{\Coker}{\operatorname{Coker}\nolimits}
\newcommand{\cone}{\operatorname{Cone}\nolimits}
\newcommand{\res}{\operatorname{res}\nolimits}
\newcommand{\per}{\operatorname{per}\nolimits}

\newcommand{\Ext}{\operatorname{Ext}\nolimits}

\newcommand{\umod}{\operatorname{\underline{mod}}\nolimits}
\newcommand{\uMod}{\operatorname{\underline{Mod}}\nolimits}

\newcommand{\uHom}{\operatorname{\underline{Hom}}\nolimits}

\newcommand{\Cone}{\operatorname{Cone}\nolimits}
\newcommand{\Char}{\operatorname{Char}\nolimits}

\newcommand{\Inj}{\operatorname{Inj}\nolimits}

\newcommand{\Ann}{\operatorname{Ann}\nolimits}

\newcommand{\Ab}{\mathrm{Ab}}

\newcommand{\op}{\mathrm{op}}
\newcommand{\inc}{\mathrm{inc}}

\newcommand{\can}{\mathrm{can}}

\newcommand{\comp}{\mathop{\raisebox{+.3ex}{\hbox{$\scriptstyle\circ$}}}}

\newcommand{\xto}{\xrightarrow}

\newcommand{\grspec}{\operatorname{Spec_{gr}}\nolimits}
\newcommand{\ra}{\longrightarrow}
\newcommand{\Ddg}{\D}

\newcommand{\Ainf}{A_{\infty}}
\newcommand{\qis}{quasi-isomorphism}
\newcommand{\Hgk}{H^{*}(G,k)}

\newcommand{\Kinjg}{\operatorname{\bf K}(\Inj kG)}
\newcommand{\HH}{\mathrm{HH}}

\newcommand{\Zgr}{Z_{\mathrm gr}}

\def\p{\mathfrak{p}}
\def\m{\mathfrak{m}}
\def\a{\mathfrak{a}}
\def\b{\mathfrak{b}}

\def\ep{\varepsilon}

\def\g{\gamma}

\def\r{\rho}

\def\Si{\Sigma}

\def\A{{\mathcal A}}
\def\B{{\mathcal B}}
\def\C{{\mathcal C}}
\def\D{{\mathcal D}}

\def\I{{\mathcal I}}

\def\K{{\mathcal K}}

\def\M{{\mathcal M}}
\def\N{{\mathcal N}}

\def\S{{\mathcal S}}

\def\T{{\mathcal T}}
\def\U{{\mathcal U}}

\def\bbZ{\mathbb Z}
\def\bbB{\mathbb B}

\def\bfa{\mathbf a}

\def\bfi{\mathbf i}

\def\bfp{\mathbf p}

\def\bfC{\mathbf C}
\def\bfD{\mathbf D}
\def\bfK{\mathbf K}
\def\bfL{\mathbf L}
\def\bfR{\mathbf R}
\def\bfS{\mathbf S}

\begin{document}
\originalTeX 
\begin{titlepage}
\thispagestyle{empty}
\begin{center}
\thispagestyle{empty}
\mbox{ } \\
\textbf{\Large 
Universit\"{a}t Paderborn}\\
{\tiny \ }\\
\textbf{\Large 
Fakult\"{a}t f\"{u}r Elektrotechnik, Informatik und Mathematik}\\
\mbox{ }\\
\mbox{ }\\
\mbox{ }\\
\mbox{ }\\
\mbox{ }\\
\mbox{ }\\
\mbox{ }\\
\mbox{ }\\
\mbox{ }\\
\mbox{ }\\
\mbox{ }\\
\mbox{ }\\
{\huge
\textbf{Dissertation}\\
\mbox{ }\\
\mbox{ }\\
 \mbox{ } \\
\textbf{
Realisability and Localisation}\\
{\tiny \ }\\
\textbf{
\Large
Realisierbarkeit und Lokalisierung}\\
\mbox{ }\\
\mbox{ }\\
\textbf{Birgit Huber}}\\
\mbox{ }\\
\large{2007}
\mbox{ }\\
\mbox{ }\\
\mbox{ }\\
\mbox{ }\\
\mbox{ }\\
\mbox{ }\\
\thispagestyle{empty}
\mbox{  } \\

\end{center}
\end{titlepage}
\pagebreak
\maketitle

\begin{abstract}
Let $A$ be a differential graded algebra with cohomology ring $H^*A$. A~gra\-ded module over $H^*A$ is called \emph{realisable} if it is (up to direct summands) of the form $H^*M$ for some differential graded $A$-module $M$. Benson, Krause and Schwede have stated  a local and a global obstruction for realisability. The global obstruction is given by the Hochschild class determined by the secondary multiplication of the $A_{\infty}$-algebra structure of $H^*A$.

In this thesis we mainly consider differential graded algebras $A$ with graded-com\-mu\-ta\-tive cohomology ring. We show that a finitely presented graded $H^*A$-module $X$  is realisable if and only if its $\p$-localisation $X_{\p}$ is realisable for all graded prime ideals $\p$ of $H^*A$.

In order to obtain such a local-global principle also for the global obstruction, we define the \emph{localisation of a differential graded algebra $A$ at a graded prime $\p$ of $H^*A$}, denoted by $A_{\p}$, and show the existence of a  morphism of differential graded algebras inducing the canonical map $H^*A \to (H^*A)_{\p}$ in cohomology. The latter result actually holds in a much more general setting: we prove that every smashing localisation on the derived category of a differential graded algebra is induced by a morphism of differential graded algebras. 

Finally we discuss the relation between realisability of modules over the group cohomology ring and the Tate cohomology ring.
\end{abstract}

\germanTeX
\begin{abstract}
Sei $A$ eine differenziell graduierte Algebra mit Kohomologiering $H^*A$. Ein graduierter $H^*A$-Modul hei\ss t \emph{realisierbar}, falls man ihn (bis auf direkte Summanden)  mit einem $H^*A$-Modul von der Form $H^*M$ identifizieren kann, wobei $M$ ein differenziell graduierter $A$-Modul ist.  Benson, Krause und Schwede haben ein lokales und ein globales Hindernis f"ur Realisierbarkeit angegeben. Das globale Hindernis ist durch eine Hochschild Klasse gegeben, welche durch die sekund"are Mulitplikation der $A_{\infty}$-Algebra-Struktur von $H^*A$ bestimmt ist. 

In dieser Doktorarbeit betrachten wir haupts"achlich differenziell graduierte Algebren $A$ mit graduiert-kommutativen Kohomologieringen. Wir zeigen, dass ein endlich pr"asentierter graduierter $H^*A$-Modul $X$ genau dann realisierbar ist, wenn dessen $\p$-Lokalisierung $X_{\p}$ f"ur alle graduierten Primideale $\p$ von $H^*A$ realisierbar ist. 

Um ein solches Lokal-Global Prinzip auch f"ur globale Realisierbarkeit zu formulieren, definieren wir die \emph{Lokalisierung einer differenziell graduierten Algebra $A$ an einem Primideal $\p$ von $H^*A$} und bezeichnen sie mit $A_{\p}$.  Wir zeigen die Existenz eines Morphismus von differenziell graduierten Algebren, der in der Kohomologie die kanonische Abbildung $H^*A \to (H^*A)_{\p}$
induziert. Letzteres Resultat beweisen wir in einem wesentlich allgemeineren Kontext: Wir zeigen, dass jede mit direkten Summen kommutierende Lokalisierung der derivierten Kategorie einer differenziell graduierten Algebra von einem Morphismus von differenziell graduierten Algebren induziert ist. 

Abschlie\ss end diskutieren wir den Zusammenhang von Realisierbarkeit von Moduln "uber dem Gruppen-Kohomologiering und dem Tate-Kohomologiering.
\end{abstract}

\originalTeX


\newpage
\tableofcontents
\pagebreak
\subsection*{Danksagung}
An erster Stelle m\"ochte ich mich ganz besonders herzlich bei meinem Dissertationsbetreuer  Henning Krause f\"ur die Vergabe dieses Themas, die  engagierte Betreuung dieser Doktorarbeit und viele wichtige Anregungen f\"ur meine Arbeit bedanken. Besonders danken m\"ochte ich ihm auch f\"ur die Unterst\"utzung au\ss erhalb der Mathematik, insbesondere f\"ur seinen gro\ss en Einsatz, die Finanzierung meiner Arbeit stets zu sichern. Erw\"ahnen m\"ochte ich auch meine  gro\ss artigen Dienstreisen zu Konferenzen oder Forschungsaufenthalten, die dank seines Einsatzes vom mathematischen Institut oder der Forschungskommision der Universit\"at Paderborn finanziert wurden.  

W\"ahrend eines dreimonatigen Forschungsaufenthalt an der NTNU Trondheim in Norwegen wurde meine Arbeit von \O{}yvind Solberg betreut, und auch ihm m\"ochte ich f\"ur sein gro\ss es Engagement und seine Gastfreundschaft sehr herzlich danken.

Besonders dankbar bin ich auch Dave Benson f\"ur eine Einladung zu einem einw\"ochigen Forschungsaufenthalt nach Aberdeen. Ich m\"ochte ihm sehr herzlich f\"ur die vielen und langen Diskussionen danken, in denen ich sehr viel gelernt habe und durch die ich die Kapitel \ref{sec group and Tate} and \ref{sec Comparing} \"uber Gruppenkohomologie wesentlich verbessern konnte.

Mein sehr herzlicher Dank gilt auch Bernhard Keller, der sich bei jeder Konferenz, auf der wir uns trafen,  Zeit genommen hat, meine Fragen mit gro\ss er Geduld und Sorgfalt zu beantworten. Dies hat entscheidend zur Entstehung des neunten Kapitels beigetragen. 

Auch bei Ragnar-Olaf Buchweitz m\"ochte ich mich herzlich f\"ur Diskussionen w\"ahrend seines Gastaufenthalts in Paderborn bedanken. Hierbei habe ich viel \"uber Kommutative Algebra gelernt und Anregungen erhalten, die mich bei dem Beweis von Ergebnissen in Abschnitt \ref{sec loc-glob global} einen gro\ss en Schritt weiter gebracht haben.

F\"ur anregende mathematische Diskussionen m\"ochte ich mich auch herzlich bei Helmut Lenzing und bei Dirk Kussin bedanken.

Herzlich danken m\"ochte ich Wolfgang Zimmermann f\"ur die Einf\"uhrung in das Gebiet der Algebra und Darstellungstheorie, die ich von ihm an der Universit\"at M\"unchen erhalten habe.

\medskip

Ich m\"ochte allen Mitgliedern und ehemaligen Mitgliedern der Paderborner Dar\-stel\-lungs\-theorie Arbeitsgruppe und allen anderen wissenschaftlichen und nicht-wissen\-schaft\-lichen Mitarbeitern  f\"ur die angenehme und freundliche Atmosph\"are am Mathematischen Institut der Universit\"at Paderborn danken. Besonders danke ich in diesem Zusammenhang auch Kristian Br\"uning und Karsten Schmidt f\"ur die nette Atmosph\"are in unserem B\"uro, viele mathematische Diskussionen und die \glqq \LaTeX-Unter\-st\"ut\-zung\grqq.

Herzlich danken m\"ochte ich auch Jes\'us Soto und Stefan Wolf f\"ur das Korrekturlesen meines Manuskripts.

\medskip

\enlargethispage*{1cm}

Ich danke der Konrad-Adenauer-Stiftung f\"ur ein Graduiertenstipendium \" uber 2,5 Jahre und besonders f\"ur  die Finanzierung des Flugs  zu einer Konferenz nach Mexiko.

W\"ahrend meines dreimonatigem Aufenthalts in Trondheim wurde ich durch ein Marie-Curie-Stipendium finanziert, und ich m\"ochte daf\"ur dem Liegrits-Netzwerk und Idun Reiten   sehr herzlich danken.

Besonders danken m\"ochte ich auch der Universit\"at Paderborn f\"ur ein sechsmonatiges  Abschlussstipendium, welches mir erm\"oglicht hat, meine Doktorarbeit ohne finanzielle Sorgen in Ruhe abzuschlie\ss en.

Sehr herzlich danke ich Claus Ringel f\"ur die Einladung zu einem zweimonatigen Gast\-aufenthalt am Sonderforschungsbereich \glqq Spektrale Strukturen und To\-pologische Me\-tho\-den in der Mathematik\grqq\  (SFB 701)  an der Universit\"at Bielefeld.

\medskip

Last but not least m\"ochte ich mich bei meinen Eltern und bei Mart\'in Jim\'enez f\"ur die gro\ss artige und liebevolle Unterst\"utzung au\ss erhalb der Mathematik bedanken.

\newpage

Ich danke den Gutachtern

\vspace*{1cm}

\noindent Prof. Dave Benson\\
University of Aberdeen\\

\vspace{0.7cm}
\noindent Prof. Dr. Henning Krause\\
Universit\"at Paderborn\\

\vspace{0.7cm}
\noindent Prof. Dr. Alexander Zimmermann\\
Universit\'e de Picardie, Amiens

\newpage

\section{Introduction}
In this thesis we connect two algebraic concepts which seem unrelated
at first sight: realisability and localisation. Using some advanced
methods from Homological Algebra, we establish a local-global
principle for realisability. Before we discuss our main results in
detail, let us first explain the concepts we deal with.

The starting point is a cohomology theory which assigns to a
mathematical object $X$ its cohomology group $H^*X$. Such cohomology
theories arise for example in Algebraic Topology, Algebraic Geometry,
or in Representation Theory. Usually there exists some commutative
cohomology ring $E$ such that $H^*X$ is naturally an $E$-module. Then an
$E$-module is realisable if it is up to isomorphism of the form $H^*X$ for
some object $X$.

The assignment described above can be expressed in the language of categories and functors: If $H^* \colon \C \to \D$ is a functor between categories $\C$ and $\D$, then realisability is concerned with deciding whether an object $D \in \D$ is isomorphic to an object in the image of the functor $H^*$. 

In some specific representation theoretic context, Benson, Krause, and
Schwede \cite{BKS} established a criterion for realisability. They investigated
for a finite group $G$ and a field~$k$ the Tate cohomology functor 
$$\hat H^*(G,-) \colon \uMod kG \to \Modgr \hat H^*(G,k)$$ 
from the stable module category $\uMod kG$ into the category of $\bbZ$-graded modules over the Tate cohomology ring $\hat H^*(G,k)$.
In this setting, realisability deals with deciding whether a graded $\hat H^*(G,k)$-module $X$ 
can be written as $\hat H^*(G,M)$ for some module $M$ over the group algebra $kG$.

The stable module category $\uMod kG$ has some additional structure: it is a triangulated category. The functor $\hat H^*(G,-)$ commutes both with arbitrary direct sums and with arbitrary products. 

More generally, Benson, Krause and Schwede \cite{BKS} consider a compactly generated trian\-gulated category $\T$ admitting arbitrary direct sums, and a cohomological functor
$$H^*\colon \T \to \Modgr E$$
into the category of $\bbZ$-graded modules over a cohomology ring $E$
which preserves arbitrary direct sums and products.

For an arbitrary $\bbZ$-graded $E$-module $X$, they have given a \emph{local obstruction}
$$\kappa(X) \in \Ext^{3,-1}_E(X,X)$$ which is trivial if and only if $X$ is a direct summand of $H^*M$ for some object $M \in \T$. 

Moreover, Benson, Krause and Schwede \cite{BKS} show that if there exists an infinite sequence of obstructions 
$$\kappa_n(X) \in \Ext^{n,2-n}_{E}(X,X),\; n \ge 3,$$
where the class $\kappa_n(X)$ is defined provided that the previous one $\kappa_{n-1}(X)$ vanishes,  then it even holds  $X \cong H^*M$. In this sequence of obstructions all but the first one depend on choices. 
Only $\kappa_3(X)$ is uniquely determined and actually, it equals the local obstruction $\kappa(X)$. It is remarkable that, despite of the necessity of an infinite sequence of obstructions to decide if $X \cong H^*M$, the first obstruction already tells whether $X$ is a direct summand of $H^*M$. 

Since we mainly deal with the latter question, we call a $\bbZ$-graded $E$-module $X$ \emph{reali\-sable} if it is a direct summand of $H^*M$ for some $M \in \T$. If $X \cong H^*M$, then we refer to a \emph{strictly realisable} module.
 

The triangulated categories for which realisability is particularly interesting arise as derived categories of differential graded algebras, or shortly, dg algebras. Such algebras are complexes  with an additional multiplicative structure. They have their origin in Algebraic Topology \cite{Cartan} and encode topological invariants. 
Derived categories of dg algebras were first studied systematically by Bernhard Keller \cite{KDdg}.

If $A$ is a dg algebra, then realisability is concerned with deciding whether a graded module over the cohomology ring $H^*A$ is (up to direct summands) isomorphic to a cohomology module $H^*M$, where $M$ is a dg $A$-module. The functor in question is
$$H^*\colon \D(A) \to \Modgr H^*A,$$
where $\D(A)$ denotes the derived category of the dg algebra $A$.
Benson, Krause and Schwede \cite{BKS} 
show that this setting even admits   a \emph{global  obstruction} for realisability. 
For this purpose, they use a result of Kadeishvili~\cite{Kade} saying that $H^*A$ admits an $A_{\infty}$-algebra structure.
The global obstruction arises as the Hochschild class $\mu_A \in \HH^{3,-1}(H^*A)$ determined by the secondary multiplication $m_3^{H^*A}$ of $H^*A$: if the canonical class $\mu_A$ is trivial, then all $\bbZ$-graded $H^*A$-modules are realisable~\cite{BKS}.
\medskip

The other main concept we consider is localisation. This is an algebraic concept which has its origin in Geometry. In Commutative Algebra, the localisation of a commutative ring $R$ by a multiplicatively closed subset $S$ of $R$ is a uniquely determined ring of fractions $S^{-1}R$ with the property that each $s \in S$ is made invertible in $S^{-1}R$. Similarly, one defines a module $S^{-1}M$ which is a module over $S^{-1}R$. One considers in particular multiplicatively closed subsets $S \subseteq R$ which are the complement of a prime ideal $\p$ of $R$. The ring of fractions is then denoted by $R_{\p}$ and the module $S^{-1}M$ by $M_{\p}$. 
The ring of fractions $R_{\p}$ is a local ring, and many results on commutative rings or modules over commutative rings can be proven more easily under the assumption that the ring in question is local. 
The \emph{local-global principle} says that an assertion holds if and only it holds in  localisation at every prime ideal. It is a classical principle of Commutative Algebra.

Many rings which arise in Representation Theory of Groups or Algebraic Topology are not commutative, but still, their elements commute up to a sign which depends on the degree of the elements. Therefore these rings are called  {graded-commu\-tative}. Examples are the group and Tate cohomology ring of a group or the singular cohomology ring $H^*X$ of a topological space $X$. These examples arise as cohomology of a dg algebra. Also many other dg algebras do not have a commutative, but still a graded-commu\-ta\-tive cohomology ring.

Although these rings are not strictly commutative, still many results from Commutative Algebra can also be proven in this more general setting. 
In particular, localisation at prime ideals can be done similarly as in classical Commutative Algebra. This is folklore knowledge for some experts, but there seems to be no published account. For this reason, we provide some material on rings of fractions in the graded-commutative setting in Section \ref{grcom}. 


We will also consider localisation of triangulated categories. 
In the 1960s, Gabriel and Zisman \cite{GZ} introduced the Calculus of Fractions for arbitrary categories, generalising the classical localisation of modules.  This was used by Verdier in his th\`ese~\cite{V} to study localisation of triangulated categories and in particular, to construct the Verdier quotient.

In this thesis we show that there is a strong relation between realisability and locali\-sa\-tion. We prove relations in several settings. 

In Chapter \ref{sec Realisability and localisation} we consider differential graded algebras $A$ having a graded-commutative cohomology ring $H^*A$. 
We show  that if a graded $H^*A$-module $X$ is realisable, then so is $X_{\p}$ for every graded prime ideal of $H^*A$. Our 
main result of Chapter~\ref{sec Realisability and localisation} is
\begin{thm}[Local-global principle]
Let $A$ be a dg algebra over a commutative ring such that $H^*A$ is graded-commutative and coherent. The following conditions are {equi\-valent} for a finitely presented, graded $H^*A$-module $X$:
\begin{itemize}
\item[(1)] $X$ is realisable.
\item[(2)] $X_{\p}$ is realisable for all graded prime ideals $\p$ of $H^*A$.
\item[(3)] $X_{\m}$ is realisable for all graded maximal ideals $\m$ of $H^*A$.
\end{itemize}
\end{thm}
We will prove such a local-global principle also for \emph{global realisability} in Chapter \ref{sec Localising mu}.
Before we give a precise formulation of this result, let us first explain the contents of Chapter~\ref{seclift}.

The results stated in Chapter~\ref{seclift} are joint work with K. Br\"uning \cite{BH} and also deal with a realisability problem, however in a different setting. We consider \emph{smashing localisation functors} on derived categories of dg algebras, that is,  localisation functors of triangulated categories which commute with arbitrary direct sums. We show that every smashing localisation on the derived category of a dg algebra can be realised by a morphism of dg algebras.
More precisely, we will prove
\begin{thm}[joint work with K. Br\"uning]\label{jointwK}
Let $A$ be a dg algebra over a commutative ring and $L\colon \D(A) \to \D(A)$ a smashing localisation. Then there 
exists a dg algebra $A_L$ with the property that $\D(A_L) \simeq \D(A)/\Ker L$, and the map
$$\D(A)(A,A)^* \to \D(A)(LA,LA)^*,\quad f \mapsto L(f),$$
is induced by a zigzag of dg algebra maps $$A \xleftarrow{\sim} A' \xto{\varphi} A_L.$$
That is, there exists a dg algebra $A'$ quasi-isomorphic to $A$ and 
a morphism of dg algebras  $\varphi\colon A' \to A_L$ such that in cohomology, we have the commutative diagram 
$$\xymatrix{\ar @{--} H^*A' \ar[d]_{\cong} \ar[rd]^-{H^*\varphi}& \\
\D(A)(A,A)^* \ar[r]^-{L}&  \D(A)(LA,LA)^*.}$$
Moreover, if $A$ is a cofibrant dg algebra, then there exists a morphism $A \to A_L$ which induces the algebra map $\D(A)(A,A)^* \to \D(A)(LA,LA)^*$ in cohomology.
\end{thm}
We focus in particular on the following special case: Let  $A$ be a dg algebra with graded-commutative cohomology ring, $\p$ a graded prime ideal of $H^*A$ and $L_{\p} \colon \D(A) \to \D(A)$ the smashing localisation with the $L_{\p}$-acyclic objects being those $X \in \D(A)$ such that $(H^*X)_{\p}=0$. Then we denote
$$A_{\p}=A_{L_{\p}}$$
and call the dg algebra $A_{\p}$ \emph{localisation of $A$ at a prime $\p$ in cohomology}. 
The cohomology of $A_{\p}$ satisfies $H^*(A_{\p}) \cong (H^*A)_{\p}$ as graded rings, and with this identification, the canonical map $H^*A \to (H^*A)_{\p}$ is the cohomology of a zigzag of dg algebras 
$$A \xleftarrow{\sim} A' \xto{\varphi} A_{\p}.$$

In Chapter \ref{sec Localising mu} we 
consider dg algebras $A$ with graded-commutative cohomology ring. Our main result in this chapter is a local-global principle for global realisability. 
For this purpose, we state a global obstruction 
for the $\p$-local modules, i.e.\ those graded $H^*A$-modules with the property that $X \cong X_{\p}$. Here we use the dg algebra $A_{\p}$ constructed in Chapter \ref{seclift}: 
this obstruction arises from the  $A_{\infty}$-struc\-ture of $(H^*A)_{\p} \cong H^*(A_{\p})$. Actually, we show that the canonical Hochschild class $\mu_{A_{\p}} \in \HH^{3,-1}(H^*(A_{\p}))$, which is a global obstruction for the graded $H^*(A_{\p})$-modules due to Benson, Krause and Schwede~\cite{BKS}, is also a global obstruction for the $\p$-local $H^*A$-modules.

In order to relate the global obstruction  $\mu_A \in \HH^{3,-1}(H^*A)$ for the $H^*A$-modules and the global obstruction  $\mu_{A_{\p}} \in \HH^{3,-1}(H^*A_{\p})$ for the $\p$-local $H^*A$-modules, we construct a map of Hochschild cohomology rings 
$$\Gamma\colon \HH^{*,*}(H^*A) \longrightarrow \HH^{*,*}(H^*A_{\p})$$
which has the property $\Gamma(\mu_A)=\mu_{A_{\p}}$. 
This is the key to prove that also the global obstruction behaves well under $\p$-localisation:
\begin{thm}[Local-global principle]
Let $A$ be a differential graded algebra over a field $k$ such that $H^*A$ is 
graded-commutative. Assume that the algebra $H^*A^{\op} \otimes_k (H^*A)$ is Noetherian. Then the following conditions are equivalent:
\begin{itemize}
\item[(1)] $\mu_A \in \HH^{3,-1}(H^*A)$ is trivial.
\item[(2)] $\mu_{A_{\p}} \in \HH^{3,-1}(H^*A_{\p})$ is trivial for all graded prime ideals $\p$ of $H^*A$.
\item[(3)] $\mu_{A_{\m}} \in \HH^{3,-1}(H^*A_{\m})$ is trivial for all graded maximal ideals $\m$ of $H^*A$.
\end{itemize}
\end{thm}

\bigskip

In the last chapter of this thesis, we focus on realisability in the context of group representation theory. We study the relation between realisability over the group cohomology ring $H^*(G,k)$ and the Tate cohomology ring $\hat H^*(G,k)$, where $k$ is a field and $G$ a finite group. Note that $H^*(G,k)$ can be viewed as a subalgebra of $\hat H^*(G,k)$. 

For the group cohomology ring, the appropriate realisability setting is given by the functor
$$\xymatrix@C+4pt{\Kinjg \ar[rr]^-{\Kinjg(ik,-)^*}& & \Modgr H^*(G,k),}$$
where $\Kinjg$ is the homotopy category of injective $kG$-modules. 

The group cohomology ring $H^*(G,k)$ has better properties than the Tate cohomology ring $\hat H^*(G,k)$ which, for instance, is not Noetherian in general. However, when it comes to the source categories of realisability, the stable module category $\uMod kG$ is more ``handsome'' than the homotopy category $\Kinjg$. This is the reason why we are interested in studying the relation of realisability over  group and Tate cohomology. 

The triangulated categories $\Kinjg$ and $\uMod kG$ are related by a smashing localisation
$$\xymatrix{\ar @{--}\bfK(\Inj kG) \ar@<-1ex>[r]_-Q & \uMod kG \ar@<-1ex>[l]_-R}$$
and we are now concerned with finding a relation between realisability and this localisation of triangulated categories. 

We study realisability of fixed modules as well as global realisability. Note for the latter that both $H^*(G,k)$ and $\hat H^*(G,k)$ are the cohomology of a dg algebra and thus, they admit an $A_{\infty}$-structure yielding global obstructions denoted by 
$\mu_G \in \HH^{3,-1}(H^*(G,k))$ and $\hat \mu_G \in \HH^{3,-1}(\hat H^*(G,k))$.

The canonical class $\hat \mu_G$ has been computed for some groups $G$ by Benson, Krause and Schwede \cite{BKS}, and by Langer \cite{L}. 
We consider the same groups and compute the global obstructions for the group cohomology rings. In many cases, the Hochschild classes $\mu_G\in \HH^{3,-1}(H^*(G,k))$ and $\hat \mu_G \in \HH^{3,-1}(\hat H^*(G,k))$ turn out to behave surprisingly similar. 
As a first explanation, we  show that the algebra morphism $H^*(G,k) \to \hat H^*(G,k)$ is the cohomology of a zigzag 
of dg algebra morphisms. 
Then we are ready to prove the main result of this chapter, which is, in some parts, also an application of our results on Hochschild cohomology from Chapter \ref{sec Localising mu}:
\begin{thm}
Let $G$ be finite group, $k$ a field of characteristic $p>0$ and assume that $p$ divides the order of $G$.
If the Hochschild class $\hat{\mu}_G \in \HH^{3,-1}(\hat H^*(G,k))$ is trivial, then so is the Hochschild class $\mu_G \in \HH^{3,-1}(H^*(G,k))$. 
If the $p$-rank of the group $G$ equals one, then $\hat{\mu}_G$ is trivial if and only if $\mu_G$ is trivial. 
\end{thm}
In general, the last statement is not true for groups with $p$-rank at least two, as we show by giving a counter-example.

\subsection*{Organisation}
Our main results, as stated above, can be found in the Chapters \ref{seclift}, \ref{sec Realisability and localisation}, \ref{sec Localising mu} and~\ref{sec Comparing}. At the end of each of these chapters, we point out related open questions. 

In the second chapter we recall facts about triangulated categories and introduce briefly those triangulated categories that we deal with in this thesis.

A short review on graded rings and modules can be found in the third chapter. We focus on localisation of graded-commutative rings in Section \ref{grcom}.

In Chapter \ref{sec group and Tate} we introduce group and Tate cohomology rings and state their basic properties. 

Hochschild cohomology of graded rings is discussed in Chapter \ref{sec hochschild}. In particular, we study the multiplicative structure of the Hochschild cohomology ring $\HH^{*,*}(\Lambda)$ of a graded algebra $\Lambda$. 
We show that for elements $\zeta \in \HH^{m,i}(\Lambda)$ and $\eta \in \HH^{n,j}(\Lambda)$, we have the commutativity relation $\zeta \cdot \eta = (-1)^{mn} (-1)^{ij} \eta \cdot  \zeta$, where the multiplication is given by the Yoneda or the cup product. Hochschild cohomology rings of non-graded rings are well-known to be graded-commutative, but we do not know of a published source of this more general result. We will apply it to prove some of our results in Chapter \ref{sec Localising mu}.

In Chapter \ref{secdgmod} we introduce differential graded algebras and discuss properties of their derived categories.

A short introduction to $A_{\infty}$-algebas is given in Chapter \ref{sectionAinf}. In particular, for a dg al\-gebra $A$, we present  Kadeishvili's construction \cite{Kade} of the secondary multiplication $m_3^{H^*A}\colon H^*A^{\otimes 3} \to H^*A$ of the $A_{\infty}$-algebra $H^*A$. 

Chapter \ref{sec Local tria} is about localisation in triangulated categories and contains important re\-qui\-sites for our main results. After a short discussion of the Calculus of Fractions for arbitrary categories due to Gabriel and Zisman \cite{GZ} we  focus on triangulated categories. In particular, we introduce the Verdier quotient \cite{V} and consider localisation sequences of triangulated categories. In Section \ref{HomLoc} we state a theorem of Krause \cite{Krcohom} on cohomological localisation  and prove some results together with K. Br\"uning \cite{BH} which apply in particular to cohomological $\p$-localisation.

In Chapter \ref{sec Realisability} we give a short review of the results of Benson, Krause and Schwede~\cite{BKS} and introduce their local and global obstruction for realisability, as discussed above. We focus especially on realisability in the setting of dg algebras.

\subsection*{Notations and conventions}
Unless otherwise stated, modules are always considered to be \emph{right} modules. In particular, we denote by $\Mod R$ the category of right $R$-modules and by $\mod R$ the category of finitely generated right $R$-modules over a ring $R$. If $R$ is self-injective, then $\uMod R$ resp.\ $\umod R$ denote the stable module categories of $\Mod R$ resp.\ $\mod R$.

When we talk about graded rings and modules, we always mean $\bbZ$-graded rings and modules.
 If $R$ is a graded ring and it is clear from the context that we mean graded $R$-modules, then we sometimes speak of $R$-modules.
 We denote by $\Modgr R$ the category of graded right $R$-modules, where the morphisms are the homogeneous graded $R$-linear maps of degree zero.
 By $\Hom^i_{R}(M,N)$ we denote the homogeneous graded $R$-linear maps $M\to N$ rising the degree by $i \in \bbZ$, and we write $\Hom_{R}(M,N)$ for $\Hom^0_{R}(M,N)$. Moreover, we set $\Hom^*_{R}(M,N) = \coprod_{n \in \bbZ} \Hom^n_{R}(M,N).$

 Cohomology of graded  modules over a graded ring $R$ is bigraded; the first index gives the cohomological degree and the second, internal degree arises from the grading of $R$. 
For example, for $i \ge 0$ and $j \in \bbZ$ we have
$$\Ext_{R}^{i,j}(M,N)=\Ext_{R}^{i}(M,N[j]),$$
where $[j]$ denotes the $j$-fold shift on $\Modgr R$. 
 
In particular, Hochschild cohomology of a graded algebra $\Lambda$, denoted by $\HH^{*,*}(\Lambda)$, is bigraded. Note that only the first grading is changed by the differential. 

For the homotopy category of complexes in an additive category $\A$ we write $\bfK(\A)$,
and the derived category of an abelian category $\A$
is denoted by $\bfD(\A)$. For the derived category $\bfD(\Mod R)$ of a ring $R$ we write shortly $\bfD(R)$.

All dg algebras considered in this thesis are supposed to have a differential of degree~$+1$. So the homology of these dg algebras is, in fact, cohomology, and throughout this thesis, we speak of cohomology. If $A$ is a dg algebra, then we denote its cohomology ring by $H^*A$, and the homotopy resp.\ derived category of $A$ will be denoted by $\K(A)$ resp.\ $\D(A)$.

The symbol ${\sim}$ indicates a quasi-isomorphism and the symbol $\cong$ is used for  isomorphisms of objects in categories. Equivalences of categories are indicated by $\simeq$. 

The set of morphisms $X \to Y$ in a category $\C$ is denoted by $\C(X,Y)$ or $\Hom_{\C}(X,Y)$. If $\T$ is a triangulated category, then we denote its suspension functor by $\Sigma$ or $[1]$. For $i \in \bbZ$, we write $\T(X,Y)^i$ for $\T(X,\Sigma^i Y)$ and we denote
$$\T(X,Y)^* = \coprod_{i \in \bbZ}\T(X,Y)^i.$$

The composition of maps $f\colon A \to B$ and $g\colon B \to C$ is denoted by $g \circ f$ or $gf$; similarly for functors.

\newpage

\section{Triangulated categories}\label{sec Triangulated categories}
Triangulated categories were introduced independently by Verdier in his th\`ese \cite{V}, and in Algebraic Topology by Puppe \cite{P}. 

The purpose of this chapter is to state results on triangulated categories which we will use later on. In particular, we recall examples of triangulated categories and fix notation. For the definition of a triangulated category we refer to the book of Neeman \cite{N2} or Krause's notes \cite{Kchicago}.

Let $\T$ be a triangulated category. We denote the suspension functor by $[1]\colon \T \to \T$ or $\Sigma\colon \T\to\T.$ 
A non-empty full subcategory $\S$ is a \emph{triangulated subcategory} if 
\begin{itemize}
\item[(i)] $\S$ is closed under shifts, i.e.\ $X \in \S$ if and only if $X[1] \in \S$.
\item[(ii)] $\S$ is closed under triangles, i.e.\ if in the exact triangle $X \to Y \to Z \to X[1]$ two objects from $\{X,Y,Z\}$ belong to $\S$, then also the third.
\end{itemize}
 A triangulated subcategory $\S$ is called \emph{thick} if it is closed under direct factors, that is, a decomposition $X = X' \amalg X''$ for $X \in \S$ implies $X' \in \S$.

 A triangulated subcategory $\S$ admitting arbitrary direct sums is called \emph{localising}. 
If $\S$ is localising, then it is already a thick subcategory (\cite[Rem.\ 3.2.7]{N2}).

Let $\N$ be a class of objects in $\T$. The triangulated subcategory \emph{generated by} $\N$ is the smallest full triangulated subcategory which contains $\N$. We refer to \cite[Ch.\ 2.8]{Kchicago} for an explicit construction.
 
If $\T$ admits arbitrary direct sums, then the triangulated subcategory \emph{generated by} $\N$ is the smallest full triangulated subcategory which contains $\N$ \emph{and} is closed under taking arbitrary direct sums. We denote this category  $\Loc(\N)$ since it is the smallest localising subcategory that contains $\N$.

An object $X \in \T$ is called \emph{compact} if the covariant Hom functor
$$\T(X,-)\colon \T \to \Ab$$
into the category $\Ab$ of abelian groups commutes with arbitrary direct sums.

We provide a useful criterion to prove that a category is generated by compact objects. 
\begin{lem}\label{compactgen}\cite[Lemma 2.2.1]{SS1}
Let $\T$ be a triangulated category with arbitrary direct sums and $\M$ a set of compact objects. 
The following conditions are equivalent:
\begin{itemize}
\item[(1)] $\T$ is generated by $\M$, i.e.\ $\T = \Loc(\M)$.
\item[(2)] An object $X \in \T$ is trivial if and only if there are no graded maps from $\M$ to $X$, i.e.\ $\T(M,X[n])=0$ for all $M \in \M$ and $n \in \bbZ$.
\end{itemize}
\end{lem}

An \emph{exact functor} $\T \to \S$ between triangulated categories is a functor preserving exact triangles and shifts. More precisely, it is a pair $(F,\eta)$, consisting of a functor $F\colon \T \to \S$ and a natural isomorphism $\eta \colon F \circ [1]_{\T} \xto{\cong} [1]_{\S} \circ F$ such that for every exact triangle $X \xto{f} Y \xto{g} Z \xto{h} X[1]$, the triangle
$$FX \xto{Ff} FY \xto{Fg} FZ \xto{\eta_X \circ Fh} (FX)[1]$$
is exact in $\S$.
The following proposition is useful to check whether an exact functor is an equivalence. It is a version of `Beilinson's Lemma' \cite{Bei}. 
\begin{prop}\cite[Prop.\ 3.10]{S}\label{schwedeprop}
Let $\T, \S$ be triangulated categories admitting arbitrary direct sums and  $F\colon \T \to \S$  an exact functor preserving arbitrary direct sums. Suppose that $\T$ has a compact generator $C$ such that
\begin{itemize}
\item[(1)] $FC$ is a compact generator of $\S$, and
\item[(2)] the map
$$F\colon \T(C,C[n]) \to \S(FC,FC[n])$$
is bijective for all $n \in \bbZ$.
\end{itemize}
Then $F$ is an equivalence of triangulated categories.
\end{prop}

Let $\T$ be a triangulated and $\A$ an abelian category. A functor $F\colon \T \to \A$ is called \emph{co\-ho\-mo\-logical} if it sends each exact triangle in $\T$ to an exact sequence in $\A$. 
In particular, if \linebreak $X \to Y \to Z \to X[1]$ is an exact triangle in $\T$, then $F$ gives rise to an infinite exact sequence 
$$\cdots \to F(Y[-1]) \to F(Z[-1]) \to FX \to FY \to FZ \to F(X[1]) \to F(Y[1]) 
\to \cdots$$

For a proof of the following well-known lemma, we refer to \cite[Ch.\  2.3]{Kchicago}.
\begin{lem}\label{T(X,-)cohomol}
Let $\T$ be triangulated and $X$ in $\T$. The representable functors
\begin{align*}
\T(X,-)\colon \T \to \Ab  \quad \text{and} \quad \T(-,X)\colon \T^{op} \to \Ab
\end{align*}
are cohomological.
\end{lem}

The following result is due to Neeman. It is a consequence of the Brown Representability Theorem (see for example \cite[Ch.\ 4.5]{Kchicago}).
\begin{prop}\cite[Prop.\ 3.3]{Kscheme}\label{Brownprop}
Let $F\colon \S \to \T$ be a an exact functor between triangulated categories, and suppose that $\S$ is compactly generated. 
\begin{itemize}
\item[(1)] There is a right adjoint $\T \to \S$ if and only if $F$ preserves arbitrary direct sums.
\item[(2)] There is a left adjoint $\T \to \S$ if and only if $F$ preserves arbitrary products.
\end{itemize}
\end{prop}

\begin{exm}\label{triaex}
We introduce briefly the triangulated categories which are considered in this thesis, in particular to fix notation. For  detailed definitions, we refer to  \cite{Kchicago}.

(1) Let $\A$ be an additive category and denote by $\bfC(\A)$ the category of complexes in $\A$.  The null-homotopic
maps form an ideal in $\bfC(\A)$ and the {\em homotopy category} 
$\bfK(\A)$ is the quotient of $\bfC(\A)$ with respect to this ideal. 
Denote by
$\Sigma\colon\bfK(\A)\to\bfK(\A)$ the equivalence which takes a complex $X$ to its shifted
complex  $\Sigma X$, defined by
$$(\Sigma X)^n=X^{n+1}\quad\textrm{and}\quad
d^n_{\Sigma X}=-d^{n+1}_X.$$ Given a map $\alpha \colon X\to Y$ of complexes, the {\em mapping
cone} $\Cone(\alpha)$ is the complex defined in degree $n$ by $X^{n+1}\amalg Y^n$, and endowed with the differential
$d^n_{\Cone(\alpha)}=\smatrix{-d^{n+1}_X&0\\ \alpha^{n+1}&d^n_Y}.$ It fits into
a {\em mapping cone sequence} $$X\xto{\alpha}Y\xto{\beta} \Cone(\alpha)\xto{\g}\Si X,$$
given in degree $n$ by 
$$X^n\xto{\alpha^n} Y^n\xto{\smatrix{0\\ \id}} X^{n+1}\amalg Y^{n}
\xto{\smatrix{-\id&0}} X^{n+1}.$$ 
$\bfK(\A)$ is triangulated, with the exact triangles being those isomorphic to a mapping cone sequence as defined above.

(2) The {\em derived category} $\bfD(\A)$ of an abelian category $\A$ is obtained
from $\bfK(\A)$ by formally inverting all quasi-isomorphisms. 
$\bfD(\A)$ is a triangulated category, where the triangulated structure is induced by the one of $\bfK(\A)$. More precisely, $\bfD(\A)$ carries a unique triangulated structure such that the canonical functor $\bfK(\A) \to \bfD(A)$ is exact. 


(3) In Chapter \ref{secdgmod} we introduce differential graded algebras. If $A$ is a differential graded algebra, then the  homotopy category $\K(A)$ and the derived category $\D(A)$ are triangulated categories. This generalises the homotopy resp.\ derived category of a non-graded algebra, viewed as differential graded algebra concentrated in degree zero.

(4) Let $\A$ be an exact category in the sense of Quillen \cite{Q}.  Thus
$\A$ is an additive category with a distinguished class of
sequences
\begin{equation*}\label{eq:ses}
0\to X\xto{\alpha} Y\xto{\beta} Z\to 0
\end{equation*}
which are called {\em exact}.  The exact sequences satisfy a number of
axioms. In particular, the maps $\alpha$ and $\beta$ in each exact sequence
as above form a {\em kernel-cokernel pair}. That is, $\alpha$ is a kernel
of $\beta$ and $\beta$ is a cokernel of $\alpha$. A map in $\A$ arising as
the kernel in some exact sequence is called {\em admissible mono}, and a
map arising as a cokernel is called {\em admissible epi}. A full
subcategory $\B$ of $\A$ is {\em extension-closed} if every exact
sequence in $\A$ belongs to $\B$, provided that its end terms belong to
$\B$.

Let $\A$ be an exact category. An object $P \in \A$ is called {\em projective} if
the induced map $\Hom_\A(P,Y)\to\Hom_\A(P,Z)$ is surjective for every
admissible epi $Y\to Z$. Dually, an object $I$ is {\em injective} if
the induced map $\Hom_\A(Y,I)\to\Hom_\A(X,I)$ is surjective for every
admissible mono $X\to Y$.  The category $\A$ has {\em enough
projectives} if every object $Z$ admits an admissible epi $Y\to Z$
with $Y$ projective, and it has {\em enough injectives} if every
object $X$ admits an admissible mono $X\to Y$ with $Y$
injective. Finally, $\A$ is called a {\em Frobenius category} if $\A$
has enough projectives and enough injectives and if both coincide.

The \emph{stable category} of a Frobenius category $\A$ is denoted by $\bfS(\A)$ and defined to be the quotient of $\A$ with respect to the ideal $\I$ of morphisms factoring through an injective object.  Thus
$$\Hom_{\bfS(\A)}(X,Y)=\Hom_\A(X,Y)/\I(X,Y)$$ for all $X,Y$ in
$\A$.

We choose for each $X \in \A$ an exact sequence
$$0 \to X \to I(X) \to \Sigma X \to 0$$
such that $I(X)$ is injective. The morphism $X \to I(X)$ is called \emph{injective hull}. One easily checks that the assignment $X \mapsto \Sigma X$ defines  an equivalence on $\bfS(\A)$.
Every exact sequence $0 \to X \to Y \to Z \to 0$ fits into a commutative diagram 
$$\xymatrix{ 0\ar[r]&X
\ar[r]^\alpha \ar@{=}[d]&Y\ar[r]^\beta \ar[d]&Z\ar[r]\ar[d]^\g&0\\
0\ar[r]&X\ar[r]&I(X)\ar[r]&\Si X\ar[r]&0}$$ 
such that $I(X)$ is injective.
The category $\bfS(\A)$ carries a triangulated structure, with the exact triangles being those isomorphic to a sequence of maps 
$$ X\xto{\alpha} Y\xto{\beta} Z\xto{\g}\Si X$$ 
as in the diagram above.

(4)(a) If $A$ is a finite dimensional self-injective algebra, then $\Mod A$ is a Frobenius category and obviously, $\Sigma$ equals the first cosyzygy $\Omega^{-1}$. We denote the stable category $\bfS(\Mod A)$ by $\uMod A$ and conclude that it has a triangulated structure.

(4)(b) The homotopy category $\bfK(\A)$ of an additive category $\A$ identifies with the stable category of $\bfC(\A)$, where the exact structure is induced by the degree-wise split short exact sequences of complexes, see \cite[Ch.\ 7.2]{Kchicago}.

(4)(c) Similarly, if $A$ is a differential graded algebra, then the homotopy category $\K(A)$ is the stable category of a Frobenius category. We provide more details in Section~\ref{K(A)stablecat}.
\end{exm}

\section{Graded rings and modules}\label{gr}
In this section we introduce graded rings and modules, and state some properties of the category of graded modules. In particular, we fix the sign convention we will use throughout this paper. Unless otherwise stated, we mean graded \emph{right} modules when we speak of graded modules.

A \emph{$\bbZ$-graded ring}  is a ring $R$ together with a decomposition of abelian groups 
$$R = \coprod_{i \in \bbZ}R_i$$ 
such that $R_iR_j \subseteq R_{i+j}$. 

A \emph{$\bbZ$-graded module}   over a $\bbZ$-graded ring $R$ is an $R$-module $M$ together with a decomposition of abelian groups 
$$M = \coprod_{i \in \bbZ}M_i$$
 satisfying $M_iR_j \subseteq M_{i+j}$. 
 
 A \emph{$\bbZ$-graded algebra} over some commutative ring $k$ is a graded ring $\Lambda$  which also has a graded $k$-module structure that makes $\Lambda$ into a $k$-algebra. Note that the operation of $k$ on $\Lambda$ has degree zero.
 
 
Throughout this thesis, we will talk about \emph{graded rings}, \emph{graded algebras} and \emph{graded modules} and always refer to $\bbZ$-graded rings, algebras and modules.

Let $M$ be a graded module over a graded ring $R$.
The elements $m \in M_i$ are called homogeneous elements of degree $i$, and we 
denote the degree of $m$ by $|m|$. For $n \in \bbZ$, the $n$-fold shifted graded module $M[n]$ is given by $M[n]^i = M^{n+i}$. We use the notation $\Sigma^n m$ when we view $m \in M^{n+i}$ as an element in $M[n]^i$.

If $M, N$ are graded $R$-modules and $n \in \bbZ$, then an $R$-linear map $f \colon M \to N$ is called  \emph{homogeneous graded map} or shortly, \emph{graded map}  of degree $n$ if $f(M_j) \subseteq N_{j+n}$ for all $j \in \bbZ$. Note that $f$ can also be considered as a graded map $M \to N[n]$ of degree zero. We denote by $\Hom^n_R(M,N)$  the set of all graded maps $M\to N$ of degree $n$, and we define 
$$\Hom^*_R(M,N) = \coprod_{n \in \bbZ} \Hom^n_R(M,N).$$

The graded $R$-modules form a category denoted by $\Modgr R$. The morphisms are the graded maps of degree zero and we denote
$$\Hom_R(M,N) = \Hom_R^0(M,N).$$

A graded $R$-module $B \subseteq M$ is a \emph{graded submodule} of $M$ if the inclusion map is a morphism in $\Modgr R$. In this case, the quotient $M/B$ also carries a natural grading.
If $f\colon M \to N$ is a morphism in $\Modgr R$, then $\Ker f$, $\Im f$ and $\Coker f$ are graded modules. Moreover, one can show that $\Modgr R$ is a Grothendieck category~\cite[Ch.\ 2.2]{NvO}. 

The graded submodules of $R$ are called \emph{graded right ideals}. An arbitrary right ideal $I$ of $R$ is graded if and only if it is generated by homogeneous elements. 

A graded $R$-module is \emph{graded free} if it is a direct sum of shifted copies of $R$, or equivalently, if it has an $R$-basis consisting of homogeneous elements. Note that it is not enough to assume that the module is graded and free as non-graded module: Considering $R = \bbZ \times \bbZ$ as graded ring concentrated in degree $0$, the module $F = \bbZ \times \bbZ$ endowed with the grading $F_0 = \bbZ \times 0$, $F_1 = 0 \times \bbZ$, and $F_i =0$ otherwise,  is not graded free since it cannot have an $R$-basis consisting of homogeneous elements.

A graded $R$-module is called \emph{graded projective} if it is a projective object in the category $\Modgr R$, or equivalently, if it is a direct summand of a graded free $R$-module. Since $\Modgr R$ is a Grothendieck category, it has enough injective objects. Those are the \emph{graded injective modules}.

Every graded $R$-module $M$ admits a graded free presentation $F_1 \to F_0 \to M \to 0$, i.e.\ $F_0$ and $F_1$ are graded free modules. If both $F_0$ and $F_1$ can be chosen to be finitely generated, then $M$ is called \emph{finitely presented}.

A graded ring is \emph{right Noetherian} if it is right Noetherian as a ring, i.e.\ every (not necessarily graded) ideal is finitely generated. In this case, every finitely generated graded $R$-module is already finitely presented.

If $N$ is a graded $R$-module and $i \ge 0, j \in \bbZ$, then one defines 
$$\Ext_R^{i,j}(M,N)=\Ext_R^{i}(M,N[j]).$$
An element of $\Ext_R^{i,j}(M,N)$ can be represented by an exact sequence of graded $R$-modules
$$0 \to N[j] \to X_i \to \cdots  \to X_1\to M \to 0.$$

Assume now that $R$ is a graded algebra over some commutative ring $k$. For graded modules $M$ and $N$ the tensor product $M \otimes_k N$ is a graded module, where the degree $i$ component is given by

$$(M \otimes_k N)_i = \coprod_{p+q=i}M_p \otimes_k N_q.$$
Let $f\colon M \to M'$ and $g\colon N \to N'$ be graded maps. Note that due to the \emph{Koszul sign rule}, in the tensor product  $f \otimes g$ there appears a sign:
\begin{equation}\label{Koszul}
(f \otimes g)(m \otimes n) = (-1)^{|g| \cdot |m|}f(m) \otimes g(n),
\end{equation}
where $m \in M$ is homogeneous and $n \in N$.

One also needs to involve signs to define the \emph{opposite algebra}:  $R^{\op}$ is again a graded algebra, with multiplication
\begin{equation}\label{op algebra}
r \cdot r' = (-1)^{\vert r \vert \cdot \vert r' \vert}r'r.
\end{equation}

A graded right $R$-module can be viewed as graded left $R^{\op}$-module by setting
$$r \cdot m = (-1)^{|r||m|}mr.$$

If $S$ is a graded $k$-algebra, then $R \otimes_k S$ is a graded $k$-algebra
with multiplication
\begin{equation}\label{tensor prod algebra}
(r \otimes s)(r' \otimes s') = (-1)^{|r'||s|}(rr' \otimes ss').
\end{equation}

A graded $(R,S)$-bimodule $M$ is simultaneously a graded left $R$-module and a graded right $S$-module such that $(rm)s=r(ms)$. The graded $(R,S)$-bimodules correspond to the graded right modules over $R^{\op} \otimes_k S$.

If  $M$ is a graded
$(R,S)$-bimodule, then $M[t]$ is a graded $(R,S)$-bimodule by setting
\begin{equation}\label{bimodule sign}
r\cdot (\Sigma^t m) \cdot s \ = \
(-1)^{t|r|}\, \Sigma^t(rms). \  \end{equation}

Note that the shift functor $M\mapsto M[1]$ for graded {\em right} modules
 does not involve any extra sign: we have $(\Sigma m) \cdot r = \Sigma(mr)$.
However, the sign in \eqref{bimodule sign} appears when translating
graded right $R$-modules into graded left modules over $R^{\op}$.


The graded rings we are particularly interested in arise as graded endomorphism rings of objects of triangulated categories.
\begin{exm}\label{examgradendring}
Let $\T$ be a triangulated category with arbitrary direct sums and suspension functor $\Sigma$. For objects $M,N \in \T$ we write $\T(N,M)^i=\T(N,\Sigma^i M)$.
Then $$\T(N,N)^* = \coprod_{i \in \bbZ}\T(N,N)^i$$ is a graded ring,  called the \emph{graded endomorphism ring of $N$}, and $\T(N,M)^*$ is a graded $\T(N,N)^*$-module by composition of graded maps.
\end{exm}

\subsection{Graded-commutative rings}\label{grcom}
A graded ring $R$ is called \emph{graded-commutative} if $rs = (-1)^{|r||s|}sr$ for all homogeneous elements $r, s \in R$. Although such a ring is not strictly commutative,  many results about graded \emph{and} commutative rings (which are studied for example in \cite{BrHe}) can still be carried over. However, ``Commutative Algebra over graded-commutative rings'' is rarely treated in literature. We provide definitions and results about localisation of graded-commutative rings that we will need later on.

We like to thank Dave Benson for pointing out Remark \ref{ident spec} and Lemma \ref{wlogeven}.

\subsubsection{Prime and maximal ideals}

A graded right ideal $\m$ of a graded ring $R$ is called \emph{graded maximal right ideal} if $\m \neq R$ and moreover, for any graded right ideal $\a$ such that $\m \subseteq \a \subseteq R$, it holds $\a = \m$ or $\a = R$. If $\b \neq R$ is a graded right ideal of $R$, then there exists a maximal right ideal $\m$ containing $\b$. 
If $R$ is graded-commutative, a  graded (maximal) right ideal is a  graded  (maximal) ideal.

For an arbitrary ring $R$, a \emph{prime ideal} $\p \varsubsetneq R$ is an ideal such that  for $a, b \in R$, it holds $a \in \p$ or $b \in \p$ whenever $aRb \subseteq \p$. If $R$ is graded-commutative, then this definition simplifies as in the case of strictly commutative rings:

\begin{defn}\label{grspec}
Let $R$ be a graded-commutative ring. An ideal $\p \varsubsetneq R$ is  a \emph{prime ideal} if $ab \in \p$ implies that $a \in \p$ or $b \in \p$. The set of graded prime ideals $\p \subseteq R$ is called the \emph{graded spectrum} of $R$ and denoted by $\grspec(R)$.
\end{defn}

Actually, the prime spectrum of a graded-commutative ring can be identified with the prime spectrum of a graded, strictly commutative ring:

\begin{rem}\label{ident spec}
Let $R$ be a graded-commutative ring, and denote by $\mathfrak{n}$ the ideal generated by the homogeneous nilpotent elements. If $x$ is a homogeneous element of odd degree, it holds $2x^2=0$, and thus, $2x$ is nilpotent. Hence $x \equiv -x \mod \mathfrak{n}$, and the factor ring $R/\mathfrak{n}$ is a graded, strictly commutative ring.
Since $\mathfrak{n}$ is contained in all graded prime ideals, it follows that $$\grspec(R)=\grspec(R/\mathfrak{n}).$$
\end{rem}

\begin{lem}
Let $R$ be a graded-commutative ring and $\a \subseteq R$ a graded ideal.
\begin{itemize}
\item[(1)] $\a$ is prime if and only if $R/\a$ is a domain.
\item[(2)] $\a$ is a graded maximal ideal if and only if every non-zero homogeneous element of $R/\a$ is invertible.
\end{itemize}
\end{lem}
\begin{proof}
(1) is trivial. For (2) note that all non-zero homogeneous elements of a graded-commutative ring $T$ are invertible if and only if the only graded ideals contained in $T$ are $(0)$ and $T$. 
\end{proof}

If $R$ is graded commutative, then all non-zero homogeneous elements being invertible implies that $R$ is a domain. Thus every graded maximal ideal is prime.

\subsubsection{Rings and modules of fractions}\label{grcom localisation}
Let $R$ be a (not necessarily graded) ring and $S$ a multiplicative subset with $1 \in S$. We define the \emph{right ring of fractions} of $R$ with respect to $S$ as a ring $R[S^{-1}]$ together with a ring homomorphism $\rho\colon R \to R[S^{-1}]$ satisfying
\begin{itemize}
\item[(F1)] $\rho(s)$ is invertible for each $s \in S$.
\item[(F2)] Every element in $R[S^{-1}]$ has the form $\rho(r)\rho(s)^{-1}$ with $s \in S$.
\item[(F3)] $\rho(r) = 0$ if and only if there exists an element $s \in S$ such that $rs=0$.
\end{itemize}
It is not immediately clear from these axioms that $R[S^{-1}]$ is uniquely determined, but it is, in fact, the case. We refer to \cite[Ch.\ II]{St}.

Let $S$ be a multiplicatively closed subset of $R$. The ring $R[S^{-1}]$ exists if and only if the following conditions, called \emph{right Ore conditions}, are satisfied:
\begin{itemize}
\item[(O1)] If $s \in S$ and $r \in R$, then there exist $s' \in S$ and $r' \in R$ such that $sr'=rs'$
\item[(O2)] If $r \in R$ and $s \in S$ with $sr=0$, then there exists $s' \in S$ such that $rs'=0$.
\end{itemize}
If (O1) and (O2) are satisfied, then
$$R[S^{-1}] = R \times S/\sim,$$
where the equivalence relation $\sim$ is given by
$$\frac{r}{s} \sim \frac{r'}{s'} $$
if and only if there exist $u,v \in R$ such that $ru=r'v$ and $su=s'v$.

If the analogous left Ore conditions are satisfied, then there exists the left ring of fractions $[S^{-1}]R$. Furthermore, if both $R[S^{-1}]$ and $[S^{-1}]R$ exist, then they are isomorphic~\cite[II, Cor.\ 1.3]{St}.

If we now assume that $R$ is graded and $S$ a multiplicative subset of homogeneous elements with $1 \in S$, then it suffices to check the Ore conditions on homogeneous elements. Moreover, if $R[S^{-1}]$ exists, then it is a graded ring, where
$$\deg\big{(}\frac{r}{s}\big{)} = \deg(r) - \deg(s)$$
for any homogeneous element $r \in R$ and $s \in S$. \cite[Ch.\ 8.1]{NvO}

If $R$ is graded-commutative, then the Ore conditions are trivially satisfied and the right ring of fractions $R[S^{-1}]$ exists. One might want to define the equivalence relation as for strictly commutative rings, but the transitivity fails for elements $s \in S$ of odd degree. However, one easily checks

\begin{lem}\label{wlogeven}
Let $R$ be a graded-commutative ring and $S$ is a multiplicative closed subset of homogeneous elements of $R$. Let $S_{ev} \subseteq S$ the subset of even-degree elements of $S$. Then $R[S^{-1}] \cong R[S_{ev}^{-1}]$ as graded rings and the equivalence relation simplifies into $\frac{r}{s} \sim \frac{r'}{s'}$ with $r,r' \in R, s,s' \in S_{ev} $
if and only if there exists $t \in S_{ev}$ such that $rs't=r'st$.\qed
\end{lem}

Consequently, we may define the localisation of a graded-commutative ring $R$ with respect to a multiplicative subset $S$ of homogeneous elements to be the ring
$$R[S_{ev}^{-1}]$$
with the equivalence relation used in the strictly commutative case. Then addition and multiplication are also defined as in that well-known case. We write $S^{-1}R$ for $R[S_{ev}^{-1}]$.

If $M$ is a graded $R$-module, we define $M_S$, the \emph{localisation of $M$ with respect to $S$} as 
$$S^{-1}M = R \times S_{ev}/\sim,$$
with $\frac{m}{s} \sim \frac{m'}{s'}$ if and only if $ms't=m'st$ for some $t \in S_{ev}$. Obviously, $S^{-1}M$ is a graded $S^{-1}R$-module with the canonical structure and grading
$$\deg\big{(}\frac{m}{s}\big{)} = \deg(m) - \deg(s),$$
where $m \in M$ is homogeneous and $s \in S_{ev}$.

Note that $S^{-1}R$ is flat as both left and right $R$-module, and that $M \otimes_R S^{-1}R \cong S^{-1}M$ as graded $S^{-1}R$-modules.

Let $\a$ be a graded ideal of $R$ and let $S$ be the subset of homogeneous elements of $R\setminus \a$. Then we set $M_{\a} = S^{-1}M$. Similarly as in classical Commutative Algebra, we have
\begin{prop}[Local-global principle]\label{local global principle}
Let $M$ be a graded $R$-module. The following conditions are equivalent:
\begin{itemize}
\item[(1)] $M = 0$.
\item[(2)] $M_{\p} = 0$ for all graded prime ideals $\p$.
\item[(3)] $M_{\m} = 0$ for all graded maximal ideals $\m$.
\end{itemize}
\end{prop}
\begin{proof}
We only need to show that (3) implies (1). Let $x \in M$ be any element. We consider $\Ann(x)^*$, the largest graded ideal contained in $\Ann(x) = \{r \in R \;|\; rx=0\}$.
Assuming that $\Ann(x)^*$ is a proper graded ideal of $R$, we obtain a graded maximal ideal $\m$ which contains $\Ann(x)^*$. Since $M_{\m} = 0$, there exists a homogeneous element $s$ in $R\setminus \m$ such that $sx =0$. So $s$ is contained in $\Ann(x)^*$ and thus in $\m$, which is a contradiction.
\end{proof}

\section{Group and Tate cohomology rings}\label{sec group and Tate}
Homology and Cohomology of groups has been considered since the 1940s. Inspired by a work of 
 Hopf \cite{Hopf} from 1941 in which he considers what today is called the second homology group $H_2(G,\bbZ)$ of a group $G$, Eilenberg and Mac Lane \cite{EilMac} started to study systematically homology and cohomology of groups. 

The Tate cohomology ring was invented by Tate, but these results were never published by himself; the first published account is contained in the book of Cartan and Eilenberg~\cite{CE}.

In the first section of this chapter we study group cohomology rings, and in the second Tate cohomology rings. 

We thank Dave Benson for many useful comments and pointing out references to the author.

\subsection{Group cohomology rings}\label{sec Group cohomology rings}
Let $k$ be a commutative ring and $G$ a finite group. The ring $k$ becomes a $kG$-module by trivial action of $G$. This module is called the \emph{trivial module} and also denoted by $k$.

If $M$ is a $kG$-module and $n \geq 0$, then the \emph{$n$-th cohomology of $G$ with coefficients in $M$} is defined to be
$$H^n(G,M)= \Ext_{kG}^n(k,M).$$
The Yoneda splice multiplication yields a $k$-bilinear, associative map (see \cite[Sect.\ 6]{C})
$$H^n(G,M) \times H^l(G,k) \to H^{n+l}(G,M),$$
defined as follows:
If $\zeta \in H^n(G,M)$ is represented by
$$E_{\zeta}\colon 0 \to M \to X_0 \to \cdots \to X_n \xto{\pi} k \to 0$$
and $\eta \in H^l(G,k)$ by
$$E_{\eta}\colon 0 \to  k \xto{\iota} Y_0 \to \cdots \to Y_l \to k \to 0,$$
then the Yoneda splice product of $\zeta$ and $\eta$ is the class which is represented by the exact sequence obtained by splicing together $E_{\zeta}$ and $E_{\eta}$:
$$\xymatrix@C-9pt{ 0 \ar[r] & M \ar[r]  & X_0 \ar[r]
& \cdots  \ar[r] & X_n  \ar[rd]^-{\pi} \ar[rr]^{\iota \circ \pi} & & Y_0 \ar[r] & \cdots \ar[r] & Y_l \ar[r]   & k \ar[r] & 0\\
&&&&& k \ar[ru]^{\iota} &&&&&}$$

We denote by $\bfK(\Inj kG)$ the homotopy category of $\Inj kG$, which is the full subcategory $\Mod kG$ formed of the injective $kG$-modules. The category $\Inj kG$ is additive, and it is closed under arbitrary direct sums, provided that $kG$ is noetherian.


Write $iM \in \bfK(\Inj kG)$ for an injective resolution of a $kG$-module $M$. 
With the well-known identification 
$$\Ext^m_{kG}(k,X) \cong \bfK(\Inj kG)(ik,\Sigma^m(iX))$$
for any $X \in \Mod kG$ and $m \geq 0,$
one can also form a $k$-bilinear, associative product
by composition of chain maps of injective resolutions:
$$H^n(G,M) \times H^l(G,k) \to H^{n+l}(G,M),\quad (f,g) \mapsto \Sigma^l(g) \circ f.$$
This product coincides with the Yoneda splice product (see \cite[Sect.\ 6]{C}).

With any of the two multiplications,
$$H^*(G,k)=\coprod_{n \ge 0} H^n(G,k)$$
is a graded ring (concentrated in non-negative degrees) and 
$$H^*(G,M)=\coprod_{n \ge 0} H^n(G,M)$$
is a graded module over $H^*(G,k)$.

Note also that $H^*(G,k)$ is a graded-commutative ring \cite[Cor.\ 6.9]{C}. Due to Evens and Venkov, it is Noetherian whenever $k$ is. 
\begin{thm}[Evens, Venkov]
If $k$ is Noetherian, then $H^*(G,k)$ is a finitely generated $k$-algebra. 
\end{thm}


\subsection{Tate cohomology rings}
Let $k$ be a field of characteristic $p>0$ and $G$ be a finite group such that $p$ divides the order of $G$. We denote by $\uMod kG$ the stable module category of $kG$. The objects are the same as in the module category $\Mod kG$, but the morphisms are given by
$$\uHom_{kG}(M,N) = \Hom_{kG}(M,N)/\I(M,N),$$
where $\I(M,N)$ denotes the morphisms factoring through an injective object. The category $\uMod kG$ is triangulated with shift functor $\Omega^{-1}$, the first cosyzygy, see Example~\ref{triaex}(4)(a).  An elementary proof for the fact that $\uMod kG$  is triangulated can be found in Carlson's book \cite[Thms.\ 5.6, 11.4]{C}.

A \emph{Tate resolution} or \emph{complete resolution} of a $kG$-module $X$ is an 
exact sequence of projectives
$$tX\colon \cdots \to P_2  \to P_1 \to P_0 \xto{\delta} P_{-1} \to P_{-2} \to \cdots$$
with $\Im \delta = X$.
It can be constructed by splicing together a projective and an injective resolution of $X$.

If $M$ denotes a $kG$-module and  $n \in \bbZ$, then the \emph{$n$-th Tate cohomology group of $G$ with coefficients in $M$} is defined to be the $n$-th cohomology  of the complex $\Hom_{kG}(tk,M)$ and denoted by
$$\hat{H}^n(G,M) = \widehat{\Ext}_{kG}^n(k,M).$$
The Tate cohomology groups identify with morphism groups in the stable module \linebreak category of $kG$: it holds
$$\hat{H}^n(G,M) \cong \uHom_{kG}(k,\Omega^{-n}M),$$
see \cite[Ch.\ 6]{C}. 
Thus $\hat{H}^*(G,k) = \coprod_{n \in \bbZ}\hat{H}^n(G,k)$ becomes
a graded ring with multiplication 
$$\hat{H}^n(G,k) \times \hat{H}^m(G,k) \to \hat{H}^{n+m}(G,k),\quad (f,g) \mapsto \Omega^{-n}(f) \circ g,$$
and $\hat{H}^*(G,M)$ is a graded module over $\hat{H}^*(G,k)$, also by composition of graded maps.

Since the non-negative part of the Tate resolution $tk$ is a projective resolution of $k$, we have $\hat{H}^n(G,M) = H^n(G,M)$ for $n > 0$, and we obtain an exact sequence
$$0 \to \I(k,M) \to H^*(G,M) \to \hat{H}^*(G,M) \to H^{-}(G,M) \to 0,$$
where $H^{-}(G,M)$ denotes the negatively graded part of Tate cohomology. In particular, since $\I(k,M)=0$ whenever $M$ has no projective direct summands, we can view the group cohomology ring as subring of the Tate cohomology ring. In fact, for the positively graded part of $\hat{H}^*(G,k)$, Yoneda splice multiplication and composition of graded maps in the stable module category coincide \cite[Sect.\ 6]{C}.

\begin{prop}[Tate duality]\cite[XII, Cor.\ 6.5]{CE}, \cite[Sect.\ 2]{BKpure}\label{Tateduality}
Let $D=\Hom_k(-,k)$. For any $kG$-module $M$, it holds
$$\hat{H}^{n-1}(G,DM) \cong D\hat{H}^{-n}(G,M).$$
\end{prop}
The Tate cohomology ring is graded-commutative \cite[XII, Prop.\ 5.2]{CE}. In general, it is not Noetherian. However, this is true in the so-called \emph{periodic case} that we discuss below. We omit the proof of the following well-known characterisation.
\begin{lem}\label{periodic resol}
The following are equivalent: 
\begin{itemize}
\item[(1)] The trivial module $k$ admits a periodic projective resolution. 
\item[(2)] There exists $n >0$ and an element $x \in \hat{H}^n(G,k)$ such that the map $$\hat{H}^m(G,k) \to \hat{H}^{m+n}(G,k), \quad \gamma \mapsto \gamma \cdot x,$$
is an isomorphism for all $m \in \bbZ$.
\end{itemize}
\end{lem}

We call an element $x \in \hat{H}^n(G,k)$ satisfying (2) a \emph{periodicity generator}, and from Lemma \ref{periodic resol}, we infer 
\begin{lem}\label{periodic-loc}
In the periodic case, the Tate cohomology ring is a localisation of the group cohomology ring:
$$\hat H^*(G,k) = S^{-1} H^*(G,k),$$
where $S$ is the multiplicative subset generated by the periodicity generator of lowest degree. In particular, $\hat H^*(G,k)$ is Noetherian whenever it is periodic.\qed
\end{lem}

Whether Tate cohomology is periodic or not depends on the \emph{$p$-rank} of $G$, which is the maximal rank of an elementary abelian $p$-subgroup of $G$ and denoted by $r_p(G)$.

\begin{thm}\cite[XII, Prop.\ 11.1]{CE}\label{periodic-rank}
Let $k$ a be field of characteristic $p$ and $G$ a group with order divisible by $p$. Then the trivial module $k$ admits a periodic resolution if and only if $r_p(G)=1$.
\end{thm}

The $p$-groups of $p$-rank one are characterised as follows:

\begin{thm}\cite[Chapter 5, Theorem 4.10]{G}\label{when rank is one}
Let $G$ be a $p$-group. Then $r_p(G)=1$ if and only if $G$ is cyclic or $p=2$ and $G$ is a generalised quaternion group.
\end{thm}
The \emph{generalised quaternion group} $Q_{2^n}$ is given by the defining relations 
$$Q_{2^n} = \langle h,g \,|\, h^{2^{n-1}}=g^2=b, b^2=1, g^{-1}hg=h^{-1} \rangle$$
and has order $2^{n+1}$, see \cite[Ch.\ 2.6]{G}. In the special case $n=2$, the group is called \emph{quaternion group}.

\begin{rem}\label{shapeTate}
The Tate cohomology ring comes in three different types: It is periodic if and only if the $p$-rank of the group $G$ equals one. Whenever the depth of the group cohomology ring $H^*(G,k)$ is at least two, then $\hat{H}^*(G,k)$ is a trivial extension of $H^*(G,k)$ by the negatively graded
part $\hat{H}^-(G,k)$ of $\hat{H}^*(G,k)$ \cite[Thm.\ 3.3]{BC1}. That is in particular the case if the $p$-rank of the centre of a $p$-Sylow subgroup of $G$ is at least two (\cite{D}, see also \cite[Thm.\ 3.2]{BC1}). But there are also cases where the Tate cohomology ring is neither periodic nor a trivial extension: The smallest example is the semidihedral group of order~16
$$SD_{16}= \langle g,h\, |\, g^8=1, h^2=1, h^{-1}gh=g^3 \rangle.$$
The $2$-rank of $SD_{16}$ is two and the depth of $H^*(SD_{16},k)$ equals one (see \cite{BC1}).
\end{rem}

\section{The Hochschild Cohomology of a graded ring}\label{sec hochschild}
Hochschild cohomology $\HH^*(R)$ of an algebra $R$ over some commutative ring $k$ was first defined by Hochschild in 1945 using the Bar resolution \cite{Hochpers}. Cartan and Eilenberg showed that $\HH^*(R)$ is isomorphic to $\Ext_{R^e}^*(R,R)$ whenever $R$ is projective over $k$ \cite[Ch.\ IX.6]{CE}.

Hochschild cohomology of a graded algebra $\Lambda$
is bigraded; it is denoted by
$\HH^{*,*}(\Lambda)$, where the first index is the cohomological degree which is changed by the differential and the second, internal one, arises from the grading of $\Lambda$. 
The difference between ``usual'' Hochschild cohomology and Hochschild cohomology of a graded algebras consists basically in the occurrence of additional signs. However, sometimes it can be tedious to figure them out. 

 In this chapter we introduce Hochschild cohomology of a graded algebra $\Lambda$ with the sign conventions as in \cite{BKS}.  In addition, we study in Section \ref{hochring} the ring structure of $\HH^{*,*}(\Lambda)$. Hochschild cohomology $\HH^*(R)$ of a non-graded algebra $R$ is well-known to be a graded-commutative ring (in the sense of Section \ref{grcom}) with multiplication given by cup or Yoneda product. We show that for bigraded Hochschild cohomology $\HH^{*,*}(\Lambda)$, one additionally needs to take into account the internal degree: for $\zeta \in \HH^{m,i}(\Lambda)$ and $\eta \in \HH^{n,j}(\Lambda)$, we prove that $\zeta \eta = (-1)^{mn} (-1)^{ij} \eta \zeta.$
 
\smallskip

Throughout this chapter let $\Lambda$ be a graded algebra over a field $k$. We write $\Lambda^{\otimes n}$ for the $n$-fold tensor product over $k$ and denote a tuple $\lambda_1 \otimes \cdots \otimes \lambda_n \in \Lambda^{\otimes n}$ by $(\lambda_1,\cdots ,\lambda_n)$.

For any graded $(\Lambda,\Lambda)$-bimodule $M$, the Hochschild cohomology
of $\Lambda$ with coefficients in $M$ is the cohomology of the
bigraded complex $C^{*,*}(\Lambda,M)$ given by
\begin{equation}\label{def-Hochschild complex}
 C^{n,m}(\Lambda,M) \ = \ \Hom^{m}_k(\Lambda^{\otimes n},M) , \end{equation}
 where $n \ge 0$ and $m \in \bbZ$. Note that the differential $\delta\colon
C^{n,m}(\Lambda,M)\to C^{n+1,m}(\Lambda,M)$ changes only the first grading.
It is given by  
\begin{multline*}
(\delta \varphi)(\lambda_1,\dots,\lambda_{n+1}) =
(-1)^{m|\lambda_1|}\lambda_1\varphi(\lambda_2,\dots,\lambda_{n+1}) \ + \\
\sum_{i=1}^n (-1)^i \varphi(\lambda_1,\dots,\lambda_i\lambda_{i+1},
\dots,\lambda_{n+1})
+ (-1)^{n+1} \varphi(\lambda_1,\dots,\lambda_n)\lambda_{n+1}.
\end{multline*}
This construction obviously generalises the Hochschild complex
in the non-graded case.

The {\em Hochschild cohomology groups} $\HH^{*,*}(\Lambda,M)$ are
the cohomology groups of the complex
$C^{*,*}(\Lambda,M)$,
\[ \HH^{s,t}(\Lambda,M) \ = \ H^{s}(C^{*,t}(\Lambda,M)) \ . \]
$\HH^{s,t}(\Lambda,\Lambda)$ is abbreviated by $\HH^{s,t}(\Lambda)$.


Using the bimodule structure of $M[t]$ given in \eqref{bimodule sign}, there is a natural isomorphism of chain complexes
$$C^{*,m}(\Lambda,M) \ \cong \ C^{*,0}(\Lambda,M[t]),$$
which induces a natural isomorphism of Hochschild cohomology groups
$$\HH^{s,t}(\Lambda,M) \ \cong \ \HH^{s,0}(\Lambda,M[t]).$$

\subsection{Functoriality}
The complex $C^{*,*}(\Lambda,M)$ and its cohomology groups
$\HH^{*,*}(\Lambda,M)$ are covariant functors in the
$(\Lambda,\Lambda)$-bimodule $M$: If $f\colon M \to N$ is a morphism of $(\Lambda,\Lambda)$-bimodules, then we have a cochain homomorphism
$$C^{s,t}(\Lambda,M) \to C^{s,t}(\Lambda,N), \quad \varphi \mapsto  f \circ \varphi, $$ 
 which induces the map
$$\HH^{s,t}(\Lambda,M) \to \HH^{s,t}(\Lambda,N), \quad [\varphi] \mapsto  [f \circ \varphi].$$

Furthermore, the Hochschild groups are contravariant functors in the graded algebra. Let
$\alpha:\Gamma\to\Lambda$ be a map of graded $k$-algebras. Then
 a $(\Lambda,\Lambda)$-bimodule $M$ carries a
$(\Gamma,\Gamma)$-bimodule structure through the morphism $\alpha$. 
We obtain 
a cochain homomorphism 
$$C^{s,t}(\Lambda,M) \to C^{s,t}(\Gamma,M), \quad \varphi \mapsto  \varphi \circ \alpha^{\otimes s},$$ inducing the map
$$\HH^{s,t}(\Lambda,M) \to \HH^{s,t}(\Gamma,M), \quad [\varphi] \mapsto  [\varphi \circ \alpha^{\otimes s}].$$

In general, we cannot expect $\alpha:\Gamma\to\Lambda$ to give rise to a map $\HH^{s,t}(\Lambda)
\to \HH^{s,t}(\Gamma)$. However, in the case that $\alpha$ is a \emph{flat epimorphism of rings}, we will construct  an algebra homomorphism $\HH^{s,t}(\Lambda) \to \HH^{s,t}(\Gamma)$ induced by $\alpha$ in Chapter~\ref{flatepi}. For this construction, we need the graded Bar resolution.


\subsection{The graded Bar resolution}\label{secBar}
Let $\Lambda^e=\Lambda^{\op} \otimes \Lambda$. With the necessary precaution on the signs (see \eqref{op algebra} and \eqref{tensor prod algebra}), this algebra is graded and we may identify graded $(\Lambda,\Lambda)$-bimodules with graded right $\Lambda^e$-modules.

The \emph{graded bar resolution} $\bbB = \bbB(\Lambda)$ 
is a $\Lambda^e$-projective resolution of $\Lambda$ defined as
$\bbB_n= B_n = \Lambda^{\otimes (n+2)}$,
with $\Lambda^e$-module structure given by
\[ (\lambda_0,\dots,\lambda_{n+1})(\mu,\mu') \ = \
(-1)^{|\mu||\lambda_0\cdots\lambda_{n+1}|} \,
(\mu\lambda_0,\lambda_1,\dots,\lambda_n,\lambda_{n+1}\mu'), \]
and with differential 
\[ d_n(\lambda_0,\dots,\lambda_{n+1}) =
\sum_{i=0}^n (-1)^i (\lambda_0,\dots,\lambda_i\lambda_{i+1},
\dots,\lambda_{n+1}). \]
For the fact that $\bbB(\Lambda)$ is indeed a graded projective
$\Lambda^e$-resolution of $\Lambda$, we refer to \cite[Sect.\ 4]{BKS} and \cite[Ch.\ IX.6]{CE}.

\begin{lem}\label{hochiso}
The map
\begin{equation}\label{tilde}
\Hom_k^{t}(\Lambda^{\otimes s},M)  \to  \Hom_{\Lambda^e}^{t}(\Lambda^{\otimes (s+2)},M), \quad
f  \mapsto \tilde{f},
\end{equation}
where 
$$\tilde{f}(\lambda_0,\cdots,\lambda_{s+1}) = (-1)^{t|\lambda_0|}\lambda_0 f(\lambda_1,\cdots,\lambda_s)\lambda_{s+1},$$
is an isomorphism and extends to an isomorphism of chain complexes
\begin{equation}\label{hochchain}
C^{*,t}(\Lambda,M) \to \Hom^{t}_{\Lambda^e}(B_*,M).
\end{equation}
\end{lem}
Consequently, we have 
$$\HH^{s,t}(\Lambda,M) = H^s (C^{*,t}(\Lambda,M))= H^s(\Hom_{\Lambda^e}^{t}(\bbB,M)),$$  
and thus, it holds
$$\HH^{s,t}(\Lambda,M) = \Ext^{s,t}_{\Lambda^e}(\Lambda,M).$$
Of course, the Hochschild cohomology groups can also be computed with an arbitrary
graded $\Lambda^e$-projective resolution of $\Lambda$.

\subsection{Ring structure}\label{hochring}
G. Hochschild \cite{Hochpers} proved that Hochschild cohomology  $\HH^{*}(R)$ of an algebra $R$ admits a ring structure, with multiplication given by the cup product. It is well-known that the cup product coincides with the Yoneda product (see \cite[Prop.\ 1.1]{BGSS})  and
makes $\HH^*(R)$ into a graded-commutative ring (see for example~\cite{SolICRA}).
In this section we study the ring structure of bigraded Hochschild cohomology  $\HH^{*,*}(\Lambda)$ of a graded algebra $\Lambda$.

There is a degree zero chain map $\Delta\colon \bbB \to \bbB \otimes_{\Lambda} \bbB$ 
lifting the identity map of $\Lambda$, given by
$$\Delta(\lambda_0,\cdots,\lambda_{n+1})= \sum_{i=0}^{n}(\lambda_0,\cdots \lambda_i,1)\otimes_{\Lambda} (1,\lambda_{i+1},\cdots,\lambda_{n+1}).$$
For the non-graded case, this can be found for example in \cite{SolICRA}. It carries over to our case without any additional sign adjustment.

Let $\zeta \in \HH^{m,i}(\Lambda)$ be represented by a cocycle $B^m \to \Lambda$ of degree $i$, and represent
 $\eta \in \HH^{n,j}(\Lambda)$ by a $(n,j)$-cocycle $B^n \to \Lambda$.  Then the \emph{cup product} of $\zeta$ and $\eta$ is given by the composition of graded maps
$$\bbB \xto{\Delta} \bbB \otimes_{\Lambda} \bbB \xto{\zeta \otimes \eta} \Lambda \otimes_{\Lambda} \Lambda \xto{\nu} \Lambda,$$
where $\nu\colon \Lambda \otimes_{\Lambda} \Lambda \to \Lambda$ is the multiplication map. In order to write the cup product $\zeta \cup \eta$ explicitly as a $(m+n,i+j)$-cocycle $B^{m+n} \to \Lambda$, we need to take into account the Koszul sign rule \eqref{Koszul} and obtain
\begin{equation}\label{multcupproduct}
(\zeta \cup \eta)(\lambda_0,\cdots,\lambda_{m+n+1})=(-1)^{(|\lambda_0,\cdots,\lambda_m|)\cdot j}
\zeta(\lambda_0,\cdots\lambda_m,1) \eta(1,\lambda_{m+1},\cdots \lambda_{m+n+1}).\end{equation}
The cup product makes $\HH^{*,*}(\Lambda)$ into a \emph{bigraded ring}, that is, a $\bbZ \times \bbZ$-graded ring.

The \emph{Yoneda product} $\zeta * \eta$ is given through a \emph{graded lifting} of $\eta$, i.e.\ a graded chain map $\bbB \to \bar \Sigma^n \bbB$ which lifts $\eta$. Here $ \bar \Sigma^n \bbB$  denotes the $n$-fold shift to the left of the complex $\bbB$ without changing the signs in the differential. Note that the `internal' degree of this chain map is the degree of $\eta$. 

Now we adapt the proof for the non-graded case (see \cite[Prop.\ 1.1]{BGSS}) to our setting and show that  cup product and Yoneda product coincide: 
\begin{prop}\label{cup=Yoneda}
Let $\zeta \in \HH^{m,i}(\Lambda)$ and $\eta \in \HH^{n,j}(\Lambda)$ be represented by the $(m,i)$-cocycle $\zeta\colon B^m \to \Lambda$  and the $(n,j)$-cocycle $\eta\colon B^n \to \Lambda$, respectively. Then it holds
$$\zeta \cup \eta = \zeta * \eta.$$

\end{prop}
\begin{proof}
In \cite[Prop.\ 1.1]{BGSS}, it is shown that
$\tilde{\eta}\colon \bbB \to \bar \Sigma^n\bbB$, defined by
$$\bbB \xto{\Delta} \bbB \otimes_{\Lambda} \bbB \xto{\bbB \otimes \eta} \bbB \otimes_{\Lambda} \bar \Sigma^n \Lambda \xto{\nu} \bar \Sigma^n \bbB,$$
is a lifting of $\eta$.
Actually, this is a graded lifting of internal degree $j$.
Hence we may set
$$\zeta * \eta = \zeta  \tilde{\eta} = \zeta \nu (\bbB \otimes \eta) \Delta.$$
The Koszul sign rule \eqref{Koszul} permits the equations 
$$\zeta \nu = \nu (\zeta \otimes {\bar \Sigma^n \Lambda}) \quad \text{ and } \quad (\zeta \otimes {\bar \Sigma^n \Lambda})({\bbB} \otimes \eta)=\zeta \otimes \eta.$$
But then 
\begin{align*}
 \zeta * \eta & =  \zeta \nu (\bbB \otimes \eta) \Delta \\
 & =  \nu (\zeta \otimes {\bar \Sigma^n \Lambda}) (\bbB \otimes \eta) \Delta \\
 & =  \nu (\zeta \otimes \eta)  \Delta\\
 & =  \zeta \cup \eta.   \qedhere
\end{align*}
\end{proof}

\enlargethispage*{3cm}
Now  we are ready to prove 
\begin{thm}\label{bigraded-commutative}
$\HH^{*,*}(\Lambda)$ is a \emph{bigraded-commutative} ring: For $\zeta \in \HH^{m,i}(\Lambda)$ and $\eta \in \HH^{n,j}(\Lambda)$, we have
$$\zeta \eta = (-1)^{mn} (-1)^{ij} \eta \zeta,$$
where the multiplication is given by the cup or Yoneda product. 
\footnote{With another sign convention 
one obtains  a different result: $\zeta \, \tilde \cup \, \eta = (-1)^{(m+i)(n+j)} \eta \, \tilde \cup \, \zeta$ for $\zeta \in \HH^{m,i}(\Lambda)$ and $\eta \in \HH^{n,j}(\Lambda)$. Here the multiplication $\tilde \cup$ arises by taking into account also the external degree of the bigraded elements in the Koszul sign rule \eqref{Koszul} and thus in the cup product \eqref{multcupproduct}.
However, starting with our sign convention, one can define a new multiplication  by the rule
$\eta * \zeta = (-1)^{in} \eta \zeta$. This new multiplication
is then still associative, as is easily checked,
and satisfies $\zeta * \eta = (-1)^{(m+i)(n+j)} \eta * \zeta.$}
\end{thm}
\begin{proof}
Let $\zeta \in \HH^{m,i}(\Lambda)$ and $\eta \in \HH^{n,j}(\Lambda)$ be represented by the $(m,i)$-cocycle \linebreak $\zeta\colon B^m \to \Lambda$ and the $(n,j)$-cocycle $\eta\colon B^n\to \Lambda$, respectively. 
We carry over the proof of \cite[Thm.\ 2.1]{SolICRA} to our bigraded case. There it is shown that the chain map $\zeta'\colon \bbB \to \bar \Sigma^n \bbB$ given by
\begin{eqnarray*}
\zeta'_p\colon  B^{m+p} & \longrightarrow & B^p\\
(\lambda_0,\cdots,\lambda_{p+m+1}) & \longmapsto & (-1)^{mp}\big{(} \zeta(\lambda_0,\cdots,\lambda_m,1),\lambda_{m+1},\cdots,\lambda_{p+m+1}\big{)}
\end{eqnarray*}
is a lifting of $\zeta$. Since all $\zeta'_p$ are homogeneous maps of degree $i$, we conclude that $\zeta'$ is actually a graded lifting.
Moreover, we have that
{\small
\begin{eqnarray*}
\eta * \zeta(\lambda_0,\cdots,\lambda_{m+n+1}) & = & \eta \zeta_{n}(\lambda_0,\cdots,\lambda_{m+n+1})\\
 & = & (-1)^{mn}\eta\big{(}\zeta(\lambda_0,\cdots,\lambda_m,1),\lambda_{m+1},\cdots,
 \lambda_{n+m+1}\big{)}\\
 & = & (-1)^{mn}(-1)^{|\zeta(\lambda_0,\cdots,\lambda_{m},1)|\cdot j}
 \zeta(\lambda_0,\cdots,\lambda_m,1)\eta(1,\lambda_{n+1},\cdots,\lambda_{n+m+1})\\
 & = & (-1)^{mn}(-1)^{\big{(}i+|(\lambda_0,\cdots,\lambda_{m})|\big{)}\cdot j}
 \zeta(\lambda_0,\cdots,\lambda_m,1)\eta(1,\lambda_{m+1},\cdots,\lambda_{n+m+1})
\end{eqnarray*}}
On the other hand, it holds
$$(\zeta \cup \eta)(\lambda_0,\cdots,\lambda_{m+n+1})=(-1)^{|(\lambda_0,\cdots,\lambda_m)|\cdot j}
\zeta(\lambda_0,\cdots\lambda_m,1) \eta(1,\lambda_{m+1},\cdots \lambda_{m+n+1}).$$
We infer
$$\eta * \zeta = (-1)^{mn}(-1)^{ij} \zeta \cup \eta,$$
and the claim follows since cup and Yoneda product coincide by Proposition \ref{cup=Yoneda}.
\end{proof}

\begin{rem}\label{graded centre}
We may identify the graded ring $\HH^{0,*}(\Lambda)$ with the \emph{graded centre} of $\Lambda$, $\Zgr(\Lambda)= \{x \in \Lambda\,|\,x\lambda=(-1)^{|x||\lambda|}\lambda x\text{\; for all \;} \lambda \in \Lambda \}$, by the evaluation map
$$\HH^{0,*}(\Lambda)\to \Zgr(\Lambda), \quad f \mapsto f(1),$$
which is easily checked to be well-defined and bijective.
\end{rem}


\subsection{The cup product pairing}\label{cupproduct}
Let $\varphi\in \Hom_{\Lambda^e}(B_s,\Lambda)$ be a $(s,t)$-Hochschild cocycle 

\[ \xymatrix@C-5pt{\bbB \colon & \cdots \ar[r] &
B_s \ar[r] \ar[d]^(.6){\varphi} & B_{s-1}
\ar[r] &
\cdots \ar[r] & B_0 \ar[r] & \Lambda \ar[r] & 0 \\
& & \Lambda  }\]
Tensoring $\varphi$ over $\Lambda$ with a homomorphism of graded right
$\Lambda$-modules $f\colon X\to Y$ yields the map
$$f \otimes_{\Lambda}\varphi \colon X \otimes_{\Lambda} B_s \to Y$$
which is homogeneous of degree $t$.
The complex $\bbB  \otimes_{\Lambda} X$ is a projective resolution of $X$ in $\Modgr \Lambda$, and since this construction commutes with the differential, 
$f \otimes_{\Lambda} \varphi$ is a $(s,t)$-cocycle:
\[ \xymatrix@C-9pt{\bbB \otimes_{\Lambda} X \colon & \cdots \ar[r] &
B_s \otimes_{\Lambda} X \ar[r] \ar[d]^(.6){f \otimes_{\Lambda} \varphi} & B_{s-1} \otimes_{\Lambda} X
\ar[r] &
\cdots \ar[r] & B_0 \otimes_{\Lambda} X \ar[r] & X \ar[r] & 0 \\
& & Y  }\]
The cohomology class $f \cup \varphi $ of $f \otimes_{\Lambda} \varphi$
 only depends on the cohomology class of the cocycle $\varphi$ and hence, we obtain a well-defined map
\begin{equation}\label{cup pairing}
\cup \ \colon \ \Hom_{\Lambda}(X,Y)\otimes_{\Lambda} \HH^{s,t}(\Lambda)  \ \ra \
\Ext^{s,t}_{\Lambda}(X,Y), \ 
\end{equation}
\enlargethispage*{2cm}
called the {\em cup product pairing}.

\section{Differential graded algebras and their derived categories}\label{secdgmod}
Differential graded algebras, or shortly, dg algebras were introduced by Cartan~\cite{Cartan} in 1956. They arise as complexes with an additional multiplicative structure. For example, the endomorphism complex $\END(C)$ of a complex $C$ carries a natural dg algebra structure. Derived categories of dg algebras or  more generally, of dg categories, were first studied systematically by Bernhard Keller in `Deriving DG Categories' \cite{KDdg}.

\subsection{Differential graded algebras and modules}
A graded algebra over a commutative ring
$$A=\coprod_{n\in\bbZ}A^n$$
is called a {\em differential graded algebra} or \emph{dg algebra} if it carries a \emph{differential} $d\colon A\to A$, i.e.\ a graded $k$-linear map of degree $+1$ with the property $d^2=0$, which is required to satisfy the \emph{Leibniz rule}
$$d(xy)=d(x)y+(-1)^nxd(y)\quad\textrm{for all}\quad x\in
A^n\quad\textrm{and}\quad y\in A.$$
The cohomology of $A$ is a graded associative algebra over $k$ and denoted by $H^*A$.

A {\em dg $A$-module} is a
graded (right) $A$-module $X$ endowed with a differential $d\colon
X\to X$ satisfying the Leibniz rule
$$d(xy)=d(x)y+(-1)^nxd(y)\quad\textrm{for all}\quad x\in
X^n\quad\textrm{and}\quad y\in A.$$ 

A \emph{morphism of dg $A$-modules} is an
$A$-linear map which is homogeneous of degree zero and commutes with
the differential. We denote the category of dg $A$-modules by $\Moddg A$.

A map $f\colon X\to Y$ of dg $A$-modules is {\em null-homotopic} if
there is a graded $A$-linear map $\r\colon X\to Y$ of degree $-1$ such that $f=d_Y\comp \r+\r\comp d_X$. The
null-homotopic maps form an ideal and the {\em homotopy category}
$\K(A)$ is the quotient of $\Moddg A$ with respect to this
ideal. The homotopy category carries a triangulated structure which is
defined in the same way as for the homotopy category $\bfK(\C)$ of an additive
category $\C$.

A map $X\to Y$ of dg $A$-modules is a {\em quasi-isomorphism} if it
induces an isomorphism $H^nX\to H^nY$ in each degree $n \in \bbZ$. The {\em derived
category} of  the dg algebra $A$ is the localisation of $\K(A)$ with respect to
the class $S$ of all quasi-isomorphisms,
$$\D(A)=\K(A)[S^{-1}].$$  Note that $S$ is a
multiplicative system and compatible with the
triangulation. Therefore $\D(A)$ is triangulated and the
localisation functor $\K(A)\to \D(A)$ is exact.
\begin{rem}\label{dgremark}
(1) Any graded algebra is a dg algebra $A$ with differential zero. A non-graded 
algebra $R$ can be viewed as a dg algebra $A$ 
with $A^0=R$ and $A^n=0$, otherwise. In this case,
$\Moddg A$  can be identified with the category of complexes of $R$-modules $\bfC(\Mod R)$. Furthermore, $\K(A)$ identifies with $\bfK(\Mod R)$, the homotopy category of $\bfC(\Mod R)$, and $\D(A)$ with $\bfD(\Mod R)$, the derived category of complexes of $R$-modules. In particular, the results we state in the following sections 
carry over to the derived category of a module category over a non-graded algebra.
    
(2) Let $X,Y$ be complexes in some additive category $\C$. Then the complex $\HOM_\C(X,Y)$ is given by
$$\coprod_{n \in \bbZ} \prod_{p\in\mathbb Z}\Hom_\C(X^p,Y^{p+n}),$$
with differential 
$$d^n(\phi)=d_Y\comp \phi -(-1)^n\phi \comp d_X$$ 
for $\phi = (\phi^p)_{p \in  \bbZ} \in \Hom_\C(X^p,Y^{p+n})$.
Note that
$$H^n\HOM_\C(X,Y)\cong\Hom_{\bfK(\C)}(X,\Si^nY)$$ 
because $\Ker d^n$
identifies with $\Hom_{\bfC(\C)}(X,\Si^nY)$ and $\Im d^{n-1}$ with the
ideal of null-homotopic maps $X\to \Si^n Y$.  The composition of
graded maps yields a dg algebra structure for
$$\END_\C(X)=\HOM_\C(X,X),$$ and $\HOM_\C(X,Y)$ is a dg module over
$\END_\C(X)$.

(3) If $A$ is a dg algebra and $X,Y$ dg $A$-modules, then the homomorphism complex
$\HOM_A(X,Y) $ is defined in an analogous way:
$$\HOM_A(X,Y) = \coprod_{n \in \bbZ} \Hom_A^n(X,Y),$$
where $\Hom_A^n(X,Y)$ denotes the homogeneous graded $A$-linear maps $X\to Y$ rising the degree by $n \in \bbZ$.
The differential $d^n\colon \Hom_A^n(X,Y) \to \Hom_A^{n+1}(X,Y)$ is defined to be
$$d^n(f) = d_Y \circ f - (-1)^n f \circ d_X.$$
Also in this case, we have an isomorphism
$$H^n\HOM_A(X,Y)\cong\Hom_{\K(A)}(X,\Si^n Y).$$
The endomorphism ring $\END_A(X) = \HOM_A(X,X)$ is a dg algebra and $\HOM_A(X,Y)$ a dg module over $\END_A(X)$ by composition of graded maps. 
\end{rem}

The following well-known lemma shows that the functor 
$$\D(A)(A,-)^*= \coprod_{i \in \bbZ} \D(A)(A,-)^i$$ is naturally isomorphic to the cohomology functor $H^*$. 
\begin{lem}\label{D(A,-)^*=H^*-}
Let $A$ be a dg algebra and $H^*A$ its cohomology algebra. For any $X \in \D(A)$,
the evaluation map
$$\Ddg(A)(A,X)^* \to H^*X,\quad f \mapsto f(1),$$
is a natural isomorphism of graded $H^*A$-modules, where $\Ddg(A)(A,X)^*$ becomes a graded $H^*A$-module
via the isomorphism $\Ddg(A)(A,A)^* \cong H^*A$.
\end{lem}

It follows in particular that $\D(A)$ is compactly generated by $A$, the free dg $A$-module of rank one.

 J. Rickard \cite{Rickardcompact} proved that  the compact objects of $\bfD(\Mod R)$, where $R$ is a ring, are the {perfect complexes}. A complex of $R$-modules is called \emph{perfect} if it is 
quasi-isomorphic to a bounded complex of finitely generated projective $R$-modules. The full subcategory of perfect complexes of $\bfD(\Mod R)$ is denoted by $\bfD^{\per}(\Mod R)$. 

This characterisation of the compact objects was extended to the derived category $\D(A)$ of a dg algebra $A$ by Neeman \cite{Ncompact}. Here $\D^{\per}(A)$ denotes the smallest thick subcategory of $\D(A)$ containing~$A$. 
\begin{prop}\cite{Ncompact}, \cite[Ch.\ 5.5]{Kchicago}.
A dg module is compact in $\D(A)$ if and only if it is contained in $\D^{\per}(A)$.
\end{prop}


\subsection{$\K(A)$ as stable category of a Frobenius category}\label{K(A)stablecat}
Let $A$ be a dg algebra over a commutative ring $k$. Then $\Moddg A$ is an exact category with respect to the exact sequences of dg $A$-modules
$$0 \to X \to Y \to Z \to 0$$
which are split considered as sequences of graded $A$-module maps.
Furthermore, $\Moddg A$ is a Frobenius category (see Example \ref{triaex}(4) for a definition). The projective-injective objects are the dg $A$-modules $M \oplus M[1]$ with differential $\smatrix{0&\id\\ 0&0}$, where $M \in \Moddg A$. Since the maps factoring through an injective object are precisely the nullhomotopic maps, the associated stable category coincides with the homotopy category $\K(A)$. We refer to \cite[Sect.\ 8.2.3]{KinK-Z} and \cite[Sect.\ 2.2]{KDdg} for more details. 

\begin{lem}\label{choosesplit}
Let $\varphi\colon X \to Y$ be any morphism in $\K(A)$. Then $\varphi$ can be represented by a morphism $X \to \hat{ Y}$ in $\Moddg A$ which is a split monomorphism in the category of graded $A$-modules.
\end{lem}

\begin{proof}
Since $\K(A) = \S(\Moddg A)$, we can choose a map of dg $A$-modules $f\colon X \to Y$ that induces $\varphi$ in the stable category. Let $i\colon X \to I(X)$ be the injective hull and $s\colon I(X) \to X$ the graded $A$-linear map satisfying $s\circ i = \id_X$. We set $\hat Y = Y \oplus I(X)$.
Then the map
$$\smatrix{f\\ i}\colon X \to Y \oplus I(X)$$
is clearly a monomorphism of dg $A$-modules that induces $\varphi$ in the homotopy category, 
and the graded $A$-module map
$$\smatrix{0 & s}\colon Y \oplus I(X) \to X$$
satisfies $\smatrix{0 & s} \circ \smatrix{f\\ i} = \id_{X}.$
\end{proof}

\subsection{Homotopically projective and homotopically injective dg modules}\label{sechomproj}
We use the terminology of Bernhard Keller, as presented in \cite{KinK-Z}.
Throughout this section let  $A$ be a dg algebra over some commutative ring $k$. 
We say that a dg $A$-module $X$ is \emph{homotopically projective} if
$$\K(X,Y) = 0 $$
for all acyclic dg $A$-modules $Y$. 

Dually, $X$ is called \emph{homotopically injective} if
$$\K(Y,X) = 0$$
for all acyclic dg $A$-modules $Y$. 

We denote by $\K_{\bfp}(A)$ (resp.\ $\K_{\bfi}(A)$) the full subcategory of homotopically projective (resp.\ homotopically injective) dg $A$-modules of $\K(A)$.

\begin{thm}\cite[Sect.\ 8.1.6]{KinK-Z}\begin{itemize}
\item[(1)]        
For any dg $A$ module $X$, there is a triangle 
\begin{equation*}
\bfp X \to X \to \bfa X \to \Sigma \bfp X
\end{equation*}
in $\K(A)$, where $\bfp X$ is homotopically projective and $\bfa X$ is acyclic. Any triangle $Z \to X \to Y \to \Sigma Z$ with homotopically projective $Z$ and acyclic $Y$ is isomorphic to  $(\bfp X,X,\bfa X)$, and there is a unique such isomorphism extending the identity of $X$.
\item[(2)]        
For any dg $A$ module $X$, there is a triangle 
\begin{equation*}
\bfa' X \to X \to \bfi X \to \Sigma \bfa' X
\end{equation*}
in $\K(A)$, where $\bfi X$ is homotopically injective and $\bfa' X$ is acyclic. Any triangle $Z \to X \to Y \to \Sigma Z$ with acyclic $Z$ and homotopically injective $Y$ is isomorphic to  $(\bfa' X,X,\bfi X)$, and there is a unique such isomorphism extending the identity of $X$.
\end{itemize}
\end{thm}
In particular, each dg $A$ module $X$ is quasi-isomorphic to a homotopically projective (resp.\ homotopically injective) dg $A$-module and we call
\begin{align}
\bfp X \to X \quad \text{resp.} \quad X \to \bfi X
\end{align}
an \emph{homotopically projective} resp.\ \emph{homotopically injective} \emph{resolution.}
In particular, the assignments $\bfp$ and $\bfi$ can be shown to be functorial. Since they vanish on acyclic complexes, they extend to functors $\D(A) \to \K(A)$. In fact, we have
\begin{thm}\cite[Sect.\ 8.2.6]{KinK-Z}
\begin{itemize}
\item[(1)] The composition 
$$\K_{\bfp}(A) \hookrightarrow \K(A) \xto{\can} \D(A)$$
is an equivalence of triangulated categories with quasi-inverse given by
$$\bfp \colon \D(A) \to  \K_{\bfp}(A).$$ More precisely, $\bfp$ induces a fully faithful left adjoint to the quotient functor $\can\colon \K(A) \to \D(A)$.
\item[(2)] The composition 
$$\K_{\bfi}(A) \hookrightarrow \K(A) \xto{\can} \D(A)$$
is an equivalence of triangulated categories with quasi-inverse given by
$$\bfi \colon \D(A) \to  \K_{\bfi}(A).$$ More precisely, $\bfi$ induces a fully faithful right adjoint to the quotient functor $\can\colon \K(A) \to \D(A)$.
\end{itemize}
\end{thm}

\begin{cor}\label{K=D}
For all dg $A$-modules $X$ and $Y$, we have
$$\K(A)(X,\bfi Y) \cong \D(A)(X,\bfi Y) \cong \D(A)(X,Y) \cong \D(A)(\bfp X, Y) \cong \K(A)(\bfp X,Y).$$
\end{cor}

\begin{rem}
A dg $A$-module is homotopically projective if and only if it is chain homotopy equivalent to a cofibrant dg module (see \cite[Rem.\ 3.17]{S}). A dg $A$-module $X$ is \emph{cofibrant} if there exists an exhaustive increasing filtration by dg $A$-submodules 
$$0 = X_0 \subseteq X_1 \subseteq \cdots \subseteq X_n \subseteq \cdots$$
such that each subquotient $X_{n+1}/X_n$ is a direct summand of shifted copies of $A$.
The derived category $\D(A)$ can also be defined in the following way: The objects are the cofibrant modules, and the morphisms are chain homotopy classes of dg $A$-module morphisms.
\end{rem}

\subsection{Derived functors}\label{derivedfun}
Let $A$ and $B$ be two dg algebras over a commutative ring $k$. A \emph{dg $(A,B)$-bimodule} is a graded $(A,B)$-bimodule which carries in addition a $k$-linear differential $d$ of degree $+1$ satisfying
$$d(axb) = (da)xb + (-1)^p a(dx)b + (-1)^{p+q}ax(db)$$
for all $a \in A^p, x \in X^q, b \in B $.

Let $M$ be any dg $A$-module. To define the tensor product $M \otimes_A X$ of dg modules, we first observe that the tensor product $M \otimes_k X$ is a dg $B$-module.
As for graded rings, the degree $n$ component is defined to be
$$(M \otimes_k X)^n = \coprod_{p+q=n}M^p \otimes X^q.$$
Additionally, we now have the differential
$$d(m \otimes x) = (dm)\otimes x + (-1)^{\vert m \vert}m \otimes dx.$$
Since the $k$-submodule generated by all differences $ma \otimes x - m \otimes ax$ is stable under both $d$ and multiplication with elements of $B$, the quotient modulo this submodule is a well-defined dg $B$-module which we denote by
$M \otimes_{A} X.$ Moreover, this construction is functorial in $M$ and $X$.

Let $N$ be a dg $B$-module. Then $\HOM_B(X,N)$, as defined in Remark \ref{dgremark}, is a right dg $A$-module by setting
$$(fa)(x) = f(ax).$$

Observe that $- \otimes_A X$ and $\HOM_B(X,-)$ induce functors between $\K(A)$ and $\K(B)$ which form an adjoint pair
$$\xymatrix{\ar @{--} \K(A)\ar@<-1ex>[rr]_{- \otimes_A X} & & \K(B) \ar@<-1ex>[ll]-_{\HOM_B(X,-)}}$$
We define the \emph{total left derived functor } $- \otimes_{A}^{\bfL} X$ as the composition
$$\D(A) \xto{\bfp} \K(A) \xto{- \otimes_{A} X} \K(B) \xto{\can} \D(B),$$
and
the \emph{total right derived functor } $\bfR \HOM_B(X,-)$ as
$$\D(B) \xto{\bfi} \K(B) \xto{\HOM_B(X,-)} \K(A) \xto{\can} \D(A).$$
Then the total derived functors also form an adjoint pair
$$\xymatrix{\ar @{--} \D(A)\ar@<-1ex>[rr]_{- \otimes_A^{\bfL} X} & & \D(B) \ar@<-1ex>[ll]-_{\bfR \HOM_B(X,-)}}.$$
In particular, we deduce from Proposition \ref{Brownprop} that 
$- \otimes_{A}^{\bfL} X$ preserves arbitrary direct sums and $\bfR \HOM_B(X,-)$ preserves arbitrary direct products.


\subsection{Cofibrant differential graded algebras}\label{cofibrantdga} 
The category of dg algebras $\mathrm{dga}/k$ over a commutative ring $k$ admits a model category structure \cite{SS2}. 
A model category is a category with three distinguished classes of morphisms, the \emph{fibrations}, \emph{cofibrations} and \emph{weak equivalences}. These are required to satisfy certain axioms. An object $C$ in a model category is called \emph{cofibrant} if the morphism $0 \to C$ is a cofibration. We refer to~\cite[Ch.\ 1.1]{Hov} for details. 

In the category $\mathrm{dga}/k$, the fibrations are the degree-wise surjective dg algebra morphisms and the weak equivalences equal the \qis s. A dg algebra is called cofibrant if it is a cofibrant object in the model category $\mathrm{dga}/k$, that is:
\begin{defn}
A dg algebra $A$ is \emph{cofibrant} if for any
morphism of dg algebras $f \colon A\to C$ and every surjective \qis\ of dg algebras $g \colon B \to C$, there exists a lift $h \colon A \to B$. That is, we have a commutative diagram
$$
\xymatrix{
  & B\ar[d]\ar[d]|>{\object@{>>}}^{g}_{\sim}  & \\
A\ar[r]^{f} \ar@{..>}[ur]^{h}& C & }
$$
\end{defn}

A direct consequence of the model category axioms for $\mathrm{dga}/k$ is
\begin{lem}\label{cofibrantreplacement}\cite[Ch.\ 1.1]{Hov}
If $A$ is any dg algebra over a commutative ring $k$, then there exists a cofibrant dg algebra $A^{\mathrm{cof}}$ and a  \qis\ 
$$\xymatrix@C-2pt{
A^{\mathrm{cof}} \ar[r]^-{\sim} & A.}$$
\end{lem}
We remark that from the model theory axioms, it follows moreover that the \qis\ $A^{\mathrm{cof}} \xto{\sim} A$ above is a surjective map \cite[Ch.\ 1.1]{Hov}.

Examples for cofibrant dg algebras are the Sullivan algebras \cite[Ch.\ 12]{FHT} which we define in the following. From now on we assume that $k$ is a field of characteristic zero. We recall the definition of the free graded-commutative algebra:

Let $V$ be a graded vector space over $k$. The elements $v \otimes w - (-1)^{|v||w|} w \otimes v$ generate an ideal $I$ in the tensor algebra $TV$. The  \emph{free graded-commutative algebra} $\Lambda V$ is quotient of the Tensor algebra $TV$ by the ideal~$I$,
$$\Lambda V = TV/I.$$
If $v_1,\cdots,v_n$ is a $k$-basis of $V$, one also writes $\Lambda(v_1,\cdots,v_n)$ for $\Lambda V$.

\begin{defn}
A \emph{Sullivan algebra} is a dg algebra of the form $(\Lambda V,d)$, where
\begin{itemize}
\item[(1)] $V=\coprod_{p \ge 1}V_p$ is a positively graded vector space
\item[(2)] $V= \bigcup_{l\ge 0} V(l)$, where $V(0) \subseteq V(1) \subseteq \cdots $ is an increasing sequence of graded subspaces such that
$$d=0 \text{\;\;on\;\;} V(0) \text{\quad  and \quad} d(V(l)) \subseteq \Lambda V(l-1) \text{\;\;for all\,\,} l \ge 1.$$
\end{itemize} 
\end{defn}

\begin{lem}\cite[Lemma 12.4]{FHT}
Every Sullivan algebra $(\Lambda V,d)$ is a cofibrant dg algebra.
\end{lem}


For an arbitrary dg algebra $A$, the \qis\ $A^{\mathrm{cof}} \xto{\sim} A$ which exists by Lemma~\ref{cofibrantreplacement} is not easy to compute. However, for a certain class of dg algebras, one can construct explicitly quasi-isomorphic Sullivan algebras:

\begin{prop}\cite[Prop.\ 12.1]{FHT}
Assume that $A$ is a graded-commutative dg algebra concentrated in non-negative degrees which satisfies $H^0(A)=k$. Then there exists a Sullivan algebra $(\Lambda V,d)$ and a \qis
$$\xymatrix@C-2pt{
(\Lambda V,d) \ar[r]^-{\sim} & A.}$$
\end{prop}

\begin{exm}\cite[Ch.\ 12, Exm.\ 4]{FHT}
Not every dg algebra of the form $(\Lambda V,d)$ is a Sullivan algebra: Consider  $(\Lambda(v_1,v_2,v_3),d)$, where $|v_i|=1$, and the differential is given by 
$dv_1=v_2v_3$, $dv_2=v_3v_1$, and $dv_3=v_1v_2$. This dg algebra is not a Sullivan algebra. 
However,  we can state a Sullivan algebra which is quasi-isomorphic to 
$(\Lambda(v_1,v_2,v_3),d)$:
there is a \qis\
$$\xymatrix@C-2pt{
\sigma\colon (\Lambda(w),0) \ar[r]^-{\sim} & (\Lambda(v_1,v_2,v_3),d),}$$
where $w$ is of degree $3;$ the map $\sigma$ is given by $\sigma(w)=v_1v_2v_3$.
\end{exm}

\section{$A_{\infty}$-algebras}\label{sectionAinf}
$A_{\infty}$-algebras are generalisations of dg algebras. They were invented by J. Stasheff at the beginning of the 1960s as a tool in the study of `group-like' topological spaces. In the 1990s, the relevance of $A_{\infty}$-algebras in algebra became more and more apparent.
We focus on a result of Kadeishvili  stating that the cohomology of a dg algebra is an $A_{\infty}$-algebra. Instead of Kadeishvili's Russian original paper \cite{Kade} we refer the reader to the articles \cite{K-Ainfrep} and \cite{K-IntroAinf} by Bernhard Keller.

Throughout this chapter let $k$ be a field and write shortly $\otimes$ for $\otimes_k$.
\begin{defn}
An $\Ainf$-algebra is a $\bbZ$-graded vector space 
$$A = \coprod_{p \in  \bbZ}A_p$$ together with a family of homogeneous $k$-linear maps
$$m_n\colon A^{\otimes n}\to A,\;n\geq1,$$
of degree $2-n$ satisfying the relations
\begin{itemize}
\item[(i)] $m_1 m_1 = 0$.
\item[(ii)] $m_1 m_2=m_2(m_1\otimes1 + 1\otimes m_1)$.
\item[(iii)] More generally, for all $n\geq 1$,
$$\sum (-1)^{r+st} m_u(\id^{\otimes r}\otimes m_s \otimes \id^{\otimes t} )=0,$$
where the sum runs over all decompositions $n = r+s+t$, and we set $u=r+1+t$.
\end{itemize}
\end{defn}
Note that $(A,m_1)$ is a differential complex due to (i). Condition (ii) means that $m_1$ is a graded derivation with respect to the multiplication $m_2$, and equation (iii) with $n=3$ shows that the multiplication $m_2$ is associative only up to homotopy. The map $m_3$ is called the \emph{secondary multiplication}.

\begin{rem}(1) In general, an $\Ainf$-algebra is not associative. However, its cohomology $H^*A$ with respect to the differential $m_1$ is an associative $\bbZ$-graded algebra with the multiplication induced by $m_2$.

(2) If $A$ is concentrated in degree zero, then $A=A_0$ is just an associative algebra. That is because $m_n$ is of degree $2-n$ and consequently, all $m_n$ other than $m_2$ have to vanish.

(3) If $m_n$ is trivial for all $n \geq 3$, then $A$ is a dg algebra. Conversely, each dg algebra carries an $\Ainf$-structure with $m_1$ the differential, $m_2$ the multiplication, and all other $m_n$ trivial.
\end{rem}
A morphism between two $\Ainf$-algebras $A$ and $B$ is in general not just a map $A\to B$, but something quite more complicated:
\begin{defn}\label{Ainfmorph}
A morphism of $\Ainf$-algebras $f\colon A \to B$ is a family of graded maps 
$$f_n\colon A^{\otimes n} \to B$$
of degree $1-n$ such that 
\begin{itemize}
\item[(i)] $f_1m_1 = m_1 f_1$, i.e. $f_1\colon A \to B$ is a chain map.
\item[(ii)] $f_1 m_2 = m_2 (f_1 \otimes f_1) + m_1 f_2 + f_2(m_1 \otimes \id + \id \otimes m_1)$.
\item[(iii)] More generally, for $n \geq 1$, we have
$$\sum(-1)^{r+st}f_u(\id^{\otimes r} \otimes m_s \otimes \id^{\otimes t} = 
\sum (-1)^s m_r(f_{i_1} \otimes f_{i_2} \otimes \cdots f_{i_r}),$$
where the first sum runs over all decompositions $n = r+s+t$, and we set $u=r+1+t$. The second sum runs over
$1 \leq r \leq n$ and all decompositions $n=i_1 +\cdots +i_r$. Furthermore, the sign on the right hand side is given by
$$s = (r-1)(i_1 -1) + (r-2)(i_2 -1)+ \cdots + 2(i_{r-2} -1) + (i_{r-1} -1).$$
\end{itemize}
\end{defn}
Note that equation (ii) means that $f_1$ commutes with the multiplication $m_2$ up to a homotopy given by $f_2$.  

An $\Ainf$-morphism $f\colon A \to B$ is 
\begin{itemize}
\item a \emph{quasi-isomorphism} if the chain map $f_1$ is a \qis,
\item \emph{strict} if $f_i=0$ for all $i\neq 1$, 
\item the \emph{identity morphism} if $f\colon A \to A$ is strict with $f_1 = \id_A$.
\end{itemize}
The composition of two $\Ainf$-morphisms $g \colon A \to B$ and $h \colon B \to C$ is defined as
$$(h \circ g)_n = \sum (-1)^s h_r \circ (g_{i_1} \otimes \cdots \otimes g_{i_r}),$$
where the sum and the sign are as in Definition \ref{Ainfmorph} (iii). 

\begin{thm}[Kadeishvili \cite{Kade}, see also \cite{K-IntroAinf}]
Let $A$ be an $\Ainf$-algebra. Then the cohomology $H^*A$ has an $\Ainf$-algebra structure such that 
\begin{itemize}
\item[1)] $m_1^{H^*A}=0$ and $m_2^{H^*A}$ is induced by $m_2^A$, and
\item[2)] there is a \qis\ of $\Ainf$-algebras $f\colon H^*A \to A$ lifting the identity in cohomology, i.e. $H^*f_1 = \id_{H^*A}$.
\end{itemize}
Moreover, this structure is unique up to (non unique) isomorphism of $\Ainf$-algebras.
\end{thm}

In particular, the cohomology $H^*A$ of a dg algebra $A$ is an $\Ainf$-algebra.  We now show how to construct the secondary multiplication $m_3^{H^*A}$ of $H^*A$ and at the same time, the first three terms of the \qis\ $f\colon H^*A \to A$ lifting the identity of $H^*A$: 
\begin{construction}\label{m3construction}
Let $A$ be a dg algebra with differential $m_1^A$ and multiplication $m_2^A$. We view $H^*A$ as a complex with zero differential. Since we are working over a field, we can choose a \qis\ $f_1\colon H^*A \to A$ inducing the identity in cohomology. This amounts to choosing a representative cocycle for each cohomology class, in a linear way. Note that $f_1$ cannot be chosen to be multiplicative, but it does commute with multiplication \emph{up to coboundaries}. So we can choose a $k$-linear map of degree $-1$,
$$f_2\colon H^*A \otimes H^*A \to A,$$
satisfying
\begin{equation}\label{f_2}
m_1^A f_2(x,y) = f_1(xy) - f_1(x)f_1(y).
\end{equation}
So it holds indeed
\begin{equation}
f_1 m_2^{H^*A} = m_2^A (f_1\otimes f_1) + m_1^A f_2.
\end{equation}
Now we look for $f_3$ and $m_3$ such that
\begin{multline*}
f_1 \circ m_3^{H^*A} + f_2 \circ (m_2^{H^*A} \otimes \id - \id \otimes m_2^{H^*A})\\
+ f_3 \circ(m_1^{H^*A}\otimes \id^{\otimes 2} + \id \otimes m_1^{H^*A} \otimes \id + \id^{\otimes 2} \otimes m_1^{H^*A})\\
= m_3^A \circ ( f_1 \otimes f_1 \otimes f_1) + m_2^A \circ (f_1 \otimes f_2 - f_2 \otimes f_1) + m_1^A f_3.\\
\end{multline*}
Since $m_1^{H^*A}=0$, this simplifies into
\begin{equation}
f_1 \circ m_3^{H^*A} = m_2^A \circ (f_1 \otimes f_2 - f_2 \otimes f_1) - f_2 \circ (m_2^{H^*A} \otimes \id - \id \otimes m_2^{H^*A}) + m_1^A \circ f_3.
\end{equation}
Now one checks that the map
\begin{equation}
\Phi_3 = m_2^A \circ (f_1 \otimes f_2 - f_2 \otimes f_1) - f_2 \circ (m_2^{H^*A} \otimes \id - \id \otimes m_2^{H^*A})
\end{equation}
has its image in the \emph{cycles} $Z^*A$ of $A$. So we define 
\begin{equation}\label{m_3}
m_3^{H^*A} = \pi \circ \Phi_3,
\end{equation}
where $\pi$ denotes the quotient map $Z^*A\to H^*A$. 
Then
$$f_1 \circ m_3^{H^*A} - \Phi_3 = (f_1 \circ \pi - \id)\Phi_3$$
has its image in the coboundaries and thus we can indeed choose a $k$-linear map 
$$f_3\colon H^*A^{\otimes 3} \to A$$ of degree $-2$ such that
$$f_1 \circ m_3^{H^*A} - \Phi_3 = m_1^A \circ f_3$$
as desired. 
\end{construction}


This construction depends on some choices and the secondary multiplication is not uniquely determined. However, it determines a Hochschild class which is independent of all choices: 

\begin{prop}\label{m3Hochschild}\cite[Prop.\ 5.4]{BKS}
Let $A$ be a dg algebra over a field $k$. 
Then the secondary multiplication $m_3^{H^*A}$ of the $\Ainf$-algebra $H^*A$ 
is a $(3,-1)$-Hochschild cocycle. 
Moreover, its Hochschild class is independent of all choices in defining the maps $f_1$ and $f_2$. 
\end{prop}
The Hochschild class of any choice of $m_3^{H^*A}$ is denoted by $\mu_A \in \HH^{3,-1}(H^*A)$. We are particularly interested in this Hochschild class since it determines a global obstruction for realisability (\cite{BKS}, see Section \ref{sec global obstruction}). Because of this special property it is also referred to as \emph{canonical class}. In Chapter \ref{secexamples} we will compute the secondary multiplication and its Hochschild class in some examples. 

The following proposition has applications in the Chapters \ref{sec Localising mu} 
and~\ref{sec Comparing}.

\begin{prop}\cite[Cor.\ 5.7]{BKS}\label{aboutm3} 
Let $\alpha \colon A \to B$ be a morphism of dg algebras and $H^*\alpha
\colon H^*A \to H^*B$ the induced morphism in cohomology. 
\begin{itemize}
\item[(1)] In the Hochschild group $\HH^{3,-1}(H^*A,H^*B)$, it holds
$$H^*\alpha \circ \mu_A =\mu_B \circ (H^*\alpha)^{\otimes 3},$$
where $H^*B$ is a $(H^*A,H^*A)$-bimodule through $H^*\alpha$. 
\item[(2)] If $\alpha$ is a quasi-isomorphism, then the class $\mu_A$ is mapped to 
$\mu_B$ under the induced isomorphism between the Hochschild cohomology of $H^*A$ and
$H^*B$.
\end{itemize}
\end{prop}
Hence for any choice of secondary multiplications $m_3^{H^*A}$ and $m_3^{H^*B}$, the difference
$$H^*\alpha \circ m_3^{H^*A}-m_3^{H^*B}\circ (H^*\alpha)^{\otimes 3}$$
 is a $(3,-1)$-Hochschild coboundary.
If we assume in addition that the algebra map  $H^*\alpha
\colon H^*A \to H^*B$ is a \emph{monomorphism}, we can obtain equality of $H^*\alpha \circ m_3^{H^*A}$ and $m_3^{H^*B}\circ (H^*\alpha)^{\otimes 3}$ not only on the level of Hochschild classes, but even on the level of $k$-li\-near maps:

\begin{prop}\label{H*mono}
Let $\alpha \colon A \to B$ be a morphism of dg algebras and assume that $H^*\alpha
\colon H^*A \to H^*B$ is a monomorphism. Given the choices in defining $m_3^{H^*A}$, we can define $m_3^{H^*B}$ such that
$$H^*\alpha \circ m_3^{H^*A}\;= \;m_3^{H^*B}\circ (H^*\alpha)^{\otimes 3} \text{\quad in \;} \Hom_k^{-1}(H^*A^{\otimes 3},H^*B).$$
\end{prop}
\begin{proof}
Let $f_1^A\colon H^*A\to A$ and $f_2\colon H^*A \otimes H^*A \to A$ be any choices of the first two components of a \qis\ 
$f\colon H^*A \to A$ lifting the identity.
We define a graded degree zero map $f_1^B\colon H^*B \to B$  by
$$f_1^B = \alpha \circ f_1^A \circ (H^*\alpha)^{-1}$$
on the image of the monomorphism $H^*\alpha$ and extend this map $k$-linearly to a graded map inducing the identity in cohomology.
This is indeed possible  because we are working over a field
and since
$$H^*f_1^B(H^*\alpha(x))= H^*\alpha \circ f_1^A(x)=H^*\alpha(x).$$
Now we define $f_2^B$. On the image of $H^*\alpha \otimes  H^*\alpha$ we set
$$f_2^B = \alpha \circ f_2 \circ (H^*\alpha \otimes H^*\alpha)^{-1}.$$
We then extend $f_2^B$ $k$-linearly to a degree $-1$ map satisfying
$$d f_2^B(x,y) = f_2^B(x)f_2^B(y) - f_2^B(xy)$$
for all $x,y \in H^*B$.
Our choices for $f_1^B$ and $f_2^B$ then automatically yield
\begin{equation*}
H^*\alpha \circ m_3^{H^*A}=m_3^{H^*B}\circ (H^*\alpha)^{\otimes 3}.\qedhere
\end{equation*}
\end{proof}

\begin{rem}\cite[Exm.\ 7.7]{BKS}\label{m3Kunneth}
Let $A$ and $B$ dg algebras over a field $k$ and $m_3^{H^*A}$ resp.\ $m_3^{H^*B}$ secondary multiplications of their cohomology. Then under the K\"unneth isomorphism $H^*(A \otimes_k B) \cong H^*A \otimes_k H^*B$, the canonical class $\mu_{A \otimes B}$ is represented by the cocycle
\begin{multline*}
m(x_1 \otimes y_1,x_2 \otimes y_2,x_3 \otimes y_3) \ = \\
(-1)^{|x_3||y_1|+|x_3||y_2|+|x_2||y_1|}  m_3^{H^*A}(x_1,x_2,x_3) \otimes y_1y_2y_3 + x_1x_2x_3 \otimes m_3^{H^*B}(y_1,y_2,y_3).
\end{multline*}
\end{rem}

\section{Localisation in triangulated categories}\label{sec Local tria}
The classical localisation $S^{-1}R$ of a commutative ring $R$ with respect to a multiplicatively closed subset $S \subseteq R$ gives rise to the functor 
 $- \otimes_R S^{-1}R\colon  \Mod R \to \Mod S^{-1}R$. It assigns to an $R$-module $M$ the $S^{-1}R$-module $S^{-1}M$, whose elements are fractions $\frac{m}{s}, m \in M, s \in S$. 
The tensor functor  $- \otimes_R S^{-1}R$ is right adjoint to  $\Hom_{S^{-1}R}(S^{-1}R,-)$, and it is well-known that the latter functor is fully faithful.
 
This \emph{calculus of fractions} has been generalised by Gabriel and Zisman \cite{GZ} to arbitrary categories. In his th\`ese \cite{V}, Verdier applied this to introduce
localisation of triangulated categories. In particular, he invented the \emph{Verdier quotient} which is a quotient category $\T/\B$ of a triangulated category $\T$ by a triangulated subcategory $\B$.

In the first two sections we recall categories of fractions and localisation functors for arbitrary categories. Localisation of triangulated categories, in particular the Verdier quotient, will be introduced in Section \ref{sec Quotient categories}. Localisation functors of triangulated cate\-gories give rise to localisation sequences, which we define in Section \ref{sec Localisation sequences}, and those localisation sequences which are at the same time co-localisation sequences, the \emph{re\-colle\-ments}, are considered in Section \ref{recoll}. In the last section of this chapter we study cohomological localisations. These localisations are a key tool to prove our results stated in the Chapters \ref{seclift} and \ref{sec Realisability and localisation}.

 
\subsection{Categories of fractions}
A functor $F\colon \C \to \D$ is said to \emph{invert} a morphism $\sigma$ of $\C$ if $F\sigma$ is invertible. For a category $\C$ and any class of morphisms $\Sigma$ of $\C$ there exists (after taking the necessary set-theoretic precautions) the \emph{category of fractions} $\C[\Sigma^{-1}],$ together with a canonical functor 
$$Q_{\Sigma}\colon \C \to \C[\Sigma^{-1}]$$
having the following properties:
\begin{itemize}
\item[(Q1)] $Q_{\Sigma}$ makes the morphisms in $\Sigma$ invertible.
\item[(Q2)] If a functor $F\colon \C \to \D$ makes all morphisms in $\Sigma$ invertible, then there is a unique functor $G\colon \C[\Sigma^{-1}] \to \D$ such that $F = G \circ Q_{\Sigma}$.
\end{itemize}

An explicit construction of the category $\C[\Sigma^{-1}]$ can be found in the book of Gabriel and Zisman \cite{GZ}.

Let $\C, \D$ be categories and
$$\xymatrix{\ar @{--}\C \ar@<-1ex>[r]_-F & \D \ar@<-1ex>[l]_-G}$$
an adjoint pair of functors, that is,
$F\colon \C \to \D$ and $G\colon \D \to \C$ are a pair of functors such that $G$ is right adjoint to $F$. By $\eta\colon \id_{\C} \to G \circ F$  we denote the unit, and by  $\varepsilon\colon F \circ G \to \id_{\D}$ the counit of the adjunction. 
\begin{lem}\cite[Ch.\ I, Prop.\ 1.3]{GZ}\label{GabrielZisman}
Let $\Sigma$ be the set of morphisms $\sigma$ of $\C$ such that $F\sigma$ is invertible. The following are equivalent:
\begin{itemize}
\item [(1)] The functor $G$ is fully faithful.
\item [(2)] The counit of the adjunction $\varepsilon \colon F \circ G \to \id_{\D}$ is invertible.
\item [(3)] The functor $\bar{F}\colon \C[\Sigma^{-1}] \to \D$ satisfying $F = \bar{F} \circ Q_{\Sigma}$ is an equivalence.
\end{itemize}
\end{lem}

\subsection{Localisation functors}
Let $F\colon \C \to \D$ and $G\colon \D \to \C$ be an adjoint pair of functors satisfying the equivalent conditions of Lemma \ref{GabrielZisman}, and set $L=G\circ F$. The following well-known lemma shows that the pair $(F,G)$ can be reconstructed from $L$ and the adjunction unit $\Psi\colon \id_{\C} \to L$.
\begin{lem}\label{LRQ}
Let $L\colon \C\to\C$ be a functor and $\Psi\colon\id_{\C}\to L$  a natural transformation. The following are equivalent:
\begin{itemize}
\item[(1)] The map $L\Psi\colon L \to L^2$ is invertible and $L\Psi=\Psi L$.
\item[(2)] There exists a pair of functors $F\colon \C\to\D$ and $G\colon \D\to\C$ such that $F$ is left adjoint to $G$, $G$ is fully faithful, $L=G\circ F$, and $\Psi\colon \id_{\C}\to G\circ F$
is the unit of the adjunction. 
\end{itemize}
\end{lem} 
For a proof we refer to \cite{Krcohom}. However, we sketch how the pair $(F,G)$ can be constructed from the functor $L$: Given $L \colon \C \to \C$, we define $\D$ to be the full subcategory of $\C$ formed by 
those objects $X$ such that $\Psi X\colon X \to LX$ is invertible. The functor $F$ is given by $F \colon \C \to \D, \: FX = LX$, and $G \colon \D \to \C$ is defined to be the inclusion. Note also that $\Psi$ equals the adjunction unit $\eta\colon \id_{\C} \to G \circ F$.


\begin{defn} Let $\C$ be an additive category.
\begin{itemize}
\item[(1)] We call a pair $(L\colon \C \to \C,\Psi\colon \id_{\C}\to L)$  a \emph{localisation functor} if it satisfies the conditions of Lemma \ref{LRQ}. 
\item[(2)] An object $X \in \C$ is called \emph{$L$-acyclic} if $L(X)=0$, and the full subcategory of $L$-acyclic objects is denoted by $\Ker L$. 
\item[(3)] An object $X \in \C$ is called \emph{$L$-local} if $X \cong LX'$ for some $X' \in \C$. The full subcategory of $L$-local objects is denoted by $\C_L$.
\end{itemize}
\end{defn}
Remark that an object $X \in \C$ is $L$-local if and only if $\Psi X$ is invertible, see \cite[Lemma 1.5]{Krcohom}. Thus the category $\D$ constructed in Lemma \ref{LRQ} equals $\C_{L}$.

Justified by Lemma \ref{LRQ}, a functor $F\colon \C \to \D$ admitting a fully faithful right adjoint $G$ is also called \emph{localisation functor}. 

\subsection{Quotient categories}\label{sec Quotient categories}
If the category $\C$ admits an abelian (resp.\ triangulated) structure and $\B$ is a Serre  (resp.\ triangulated) subcategory, then we can form a quotient category by inverting a special class of morphisms. We present this construction for triangulated categories.

Let $\C$ be triangulated  and $\B$ a triangulated subcategory. We define $\Sigma_{\B}$ to be the class of all morphisms $\sigma \in \C$ such that there exists an exact triangle $$X \to Y \xrightarrow{\sigma} Z \to X[1],$$
 with $X \in \B$.

 The category of fractions $\C[\Sigma_{\B}^{-1}]$ is called the \emph{Verdier Quotient} and denoted by $\C/\B$. The following properties of the Verdier Quotient are well-known.

 \begin{lem} Let $Q_{\B}$ denote the canonical functor $Q_{\Sigma_{\B}} \colon \C \to \C/\B$.
 \begin{itemize}
 \item[(1)] $\C/\B$ carries a unique triangulated structure such that the  functor $Q_{\B} \colon \C \to \C/\B$ is exact. 
 \item[(2)]The kernel of $Q_{{\B}}$ consists of the \emph{thick closure} of $\B$, i.e. those objects  $X \in \C$ such that there exists $Y \in \C$ with $X \amalg Y \in \B$.
 \item[(3)] For any exact functor $F\colon \C \to \D$ satisfying $F(\B)=0$, there exists a unique functor $\bar{F}\colon \C/\B \to \D$ such that $F=\bar{F}\circ Q_{\B}$.
 \end{itemize}
 \end{lem}

A pair $(L\colon \T \to \T, \Psi\colon \id_{\C} \to L)$ is a \emph{localisation functor} of triangulated categories if $L$ is an exact localisation functor and $\Psi$ commutes with the suspension functor in the sense that $\Psi \circ [1]_{\C} \cong  [1]_{\C} \circ \Psi$.

Then $\Ker L$ and $\T_L$ are triangulated subcategories of $\T$. Moreover, the functor \linebreak $\T \to \T_L,\, X \mapsto LX$, is exact and induces an equivalence $$\T/\Ker L \xto{\simeq} \T_L.$$

It follows from Lemma \ref{LRQ} that we have a bijection between localisation functors $(L\colon \T \to \T, \Psi\colon \id_{\T} \to L)$ and quotient functors $Q\colon \T \to \T/\B$ having a fully faithful right adjoint $R$ (which are also called localisation functors). Observe that the adjunction unit $\eta\colon \id_{\T} \to RQ$ satisfies  $\eta \circ [1]_{\T} \cong [1]_{\T} \circ \eta$ because $R$ is fully faithful.
\begin{defn}
A localisation functor $(L\colon \T \to \T, \Psi\colon \id_{\C} \to L)$ is called \emph{smashing} if $L$ commutes with arbitrary direct sums. 

A localising subcategory $\B$ of $\T$ is called \emph{smashing} if  $Q\colon \T \to \T/\B$ admits a fully faithful right adjoint $R$ which commutes with arbitrary direct sums. Then the quotient functor $Q\colon \T \to \T/\B$ is also called \emph{smashing localisation}.
\end{defn}
Note that the composition $R\circ Q$ commutes with arbitrary direct sums if and only if $R$ does. Hence we have a bijection between the smashing localisation functors in the two different senses.
\subsection{Localisation sequences}\label{sec Localisation sequences}
A sequence of exact functors 
$$\A \xto{F} \B \xto{G} \C $$
between triangulated categories is called \emph{localisation sequence} if the following conditions hold:
\begin{itemize}
\item[(L1)] The functor $F$ has a right adjoint $F_{\rho}\colon \B \to \A$ satisfying $F_{\rho}\circ F \cong \id_{A}$.
\item[(L2)] The functor $G$  has a right adjoint $G_{\rho}\colon \C\to \B$ satisfying $G \circ G_{\rho} \cong \id_{\C}$, i.e.\ $G$ is a localisation functor.
\item[(L3)] Let $X$ be an object in $\B$. Then $GX=0$ if and only if $X \cong FX'$ for some $X' \in \A'$.
\end{itemize}
The sequence $(F,G)$ is called \emph{colocalisation sequence} if the sequence $(F^{\op},G^{\op})$ of opposite functors is a localisation sequence.


\enlargethispage*{1cm}
We recall the basic properties of a localisation sequence 
\begin{lem}[Verdier \cite{V}, see also \cite{Kscheme}]\label{Verdier}
Let $\A \xto{F} \B \xto{G} \C $ be a localisation sequence. Identify $\A = \Im F$ and $\C = \Im G_{\rho}$.
\begin{itemize}
\item[(1)] The functors  $F$ and $G_{\rho}$ are fully faithful. 
\item[(2)] For given objects $X,Y \in \B$, we have
\begin{align*}
X \in \A & \Longleftrightarrow \Hom_{\B}(X,\C) = 0,\\
Y \in \C & \Longleftrightarrow \Hom_{\B}(\A,Y) = 0.\\
\end{align*}
\item[(3)] The functor $G$ induces an equivalence $\B/\A \simeq \C$.
Hence every triangulated functor $G'\colon \B \to \C'$ satisfying $G'\circ F=0$ factors over $G$.
\item[(4)] For each $X \in \T$ there is an exact triangle
$$(F \circ F_{\rho})(X) \to X \to (G_{\rho} \circ G)X \to \Sigma\big{(}(F \circ F_{\rho})X\big{)}$$ which is functorial in $X$.
\item[(5)] The sequence 
$$\C \xto{G_{\rho}} \B \xto{F_{\rho}} \A$$
is a colocalisation sequence.
\end{itemize}
\end{lem}

\subsection{Recollements}\label{recoll}
We say that a sequence 
$$\A \xto{F} \B \xto{G} \C$$
of exact functors between triangulated categories induces a \emph{recollement} 
$$\xymatrix{\A \ar[r]   &   \ar@<1ex>[l]^-{F_{\lambda}} \ar@<-1ex>[l]_-{F_{\rho}} \B \ar[r]   &  \ar@<1ex>[l]^-{G_{\lambda}} \ar@<-1ex>[l]_-{G_{\rho}}  \C}$$
if it is at the same time a localisation and a colocalisation sequence.

This means that the functors $F$ and $G$ admit left adjoints $F_{\lambda}$ and $G_{\lambda}$ as well as right adjoints $F_{\rho}$ and $G_{\rho}$ such that the adjunction morphisms
$$F_{\lambda} \circ F \xto{\cong} \id_{\A} \xto{\cong} F_{\rho} \circ F \quad \text{and} \quad 
G \circ G_{\rho} \xto{\cong} \id_{\B} \xto{\cong} G \circ G_{\lambda}$$
are isomorphisms.

Let $\Lambda$ be a Noetherian ring. 
We denote by
$$I \colon \bfK_{ac}(\Inj \Lambda) \to \bfK(\Inj \Lambda)$$ the inclusion functor, and by $Q$ the canonical functor given by the composition
$$\bfK(\Inj \Lambda) \xto{\mathrm{inc}} \bfK(\Mod \Lambda) \xto{\can} \bfD(\Mod \Lambda).$$

\begin{thm}\cite[Cor.\ 4.3]{Kscheme}\label{recollement general}
The sequence
$$\bfK_{ac}(\Inj \Lambda) \xto{I} \bfK(\Inj \Lambda) \xto{Q} \bfD(\Mod \Lambda)$$
induces a recollement
\begin{equation}\label{recollementgen}
\xymatrix{\bfK_{ac}(\Inj \Lambda)  \ar[r] & \ar@<-1ex>[l]_-{I_{\rho}}  \ar@<1ex>[l]^-{I_{\lambda}}  \bfK(\Inj \Lambda)    \ar[r] & \ar@<1ex>[l]^-{Q_{\lambda}} \ar@<-1ex>[l]_-{Q_{\rho}}  \bfD(\Mod \Lambda). }
\end{equation}
\end{thm}

\begin{rem}\cite[Sect.\ 5]{BKcompl}
The recollement \eqref{recollementgen} provides two embeddings of $\bfD(\Mod \Lambda)$ into 
$\bfK(\Inj kG)$: The fully faithful functor $Q_{\rho}$ assigns to $X \in \bfD(\Mod \Lambda)$ its homotopically injective resolution $iX$ introduced in Chapter \ref{sechomproj}. The other embedding is given by the functor $Q_{\lambda}$, which identifies $\bfD(\Mod \Lambda)$ with the localising subcategory of $\bfK(\Inj kG)$ generated by $i\Lambda$. 

Assume now that $\Lambda$ is self-injective. Then $Q_{\lambda}$ maps a complex $X$ of $\Lambda$-modules to its homotopically projective resolution $pX$. Furthermore, for every $\Lambda$-module $M$, the canonical triangle
\begin{equation}\label{natural triangle}
pM \to iM \to tM \to \Sigma(pM)
\end{equation}
is isomorphic to the triangle
\begin{equation}\label{recollement triangle}
(Q_{\lambda} \circ Q)(\bar{M}) \to \bar{M} \to (I \circ I_{\lambda})\bar{M} \to \Sigma (Q_{\lambda} \circ Q)\bar{M}, 
\end{equation}
where $\bar{M} = Q_{\rho}M$.
\end{rem}
Krause proved Theorem \ref{recollement general} more generally for a locally Noetherian Grothendieck category $\A$ such that $\bfD(\A)$ is compactly generated. 
In order to prove that $(I,Q)$ induces a localisation sequence, the essential point is the following proposition which is also stated more generally in \cite{Kscheme}.

\begin{prop}\cite[Prop.\ 2.3]{Kscheme}\label{Kinjcompactlygen}
The triangulated category $\bfK(\Inj \Lambda)$ is compactly generated by the injective resolutions $iM$ of the Noetherian modules $M$. Moreover, the canonical functor $\bfK(\Mod \Lambda) \to \bfD(\Mod \Lambda)$ induces an equivalence
$$\bfK^c(\Inj \Lambda) \xto{\simeq} \bfD^b(\mod \Lambda),$$
where $\bfK^c(\Inj \Lambda)$ denotes the full subcategory of compact objects in $\bfK(\Inj \Lambda)$. 
\end{prop}

The following lemma is well-known and easy to check.
\begin{lem}\label{kac=stmod}
Let $\Lambda$ be a Noetherian self-injective ring.
The functor $$Z^0\colon \bfK_{ac}(\Inj \Lambda) \to \uMod \Lambda$$
is an equivalence with quasi-inverse $M  \mapsto tM$, where $tM$ denotes a Tate resolution of any representative of $M \in \uMod \Lambda$.
\end{lem}

Over a finite dimensional cocommutative Hopf algebra, the adjoints in the recollement can be written down explicitly:
\begin{prop}\label{recollhopf}\cite{BKcompl}
Let $H$ be a finite dimensional cocommutative Hopf algebra. Then 
the adjoints in the recollement \eqref{recollementgen} take the form
\begin{equation*}
\xymatrix{\uMod H \simeq \bfK_{ac}(\Inj H) \ar[rr]   & & \ar@<1ex>[ll]^-{-\otimes_k tk} \ar@<-1ex>[ll]_-{\HOM_k(tk,-)}   \bfK(\Inj H)   \ar[rr] & & \ar@<-1ex>[ll]_-{\HOM_k(pk,-)} \ar@<1ex>[ll]^-{-\otimes_k pk} \bfD(\Mod H). }
\end{equation*}
\end{prop}
In this case, the triangles \eqref{natural triangle} and \eqref{recollement triangle} are isomorphic to 
\begin{equation}\label{recollhopf triangle}
M \otimes_k pk \to M \otimes_k ik \to M \otimes_k tk \to \Sigma(M \otimes_k pk).
\end{equation}

\subsection{Cohomological localisation}\label{HomLoc}
Cohomological localisations were first studied by Bousfield \cite{Bous}. 
Hovey, Palmieri and Strickland \cite[Thm.\ 3.3.7]{HPS} applied them in the context of axiomatic stable homotopy theory. We refer to a paper of Krause \cite{Krcohom} for a more algebraic and detailed approach.

Throughout this section let $\T$ be a triangulated category which admits arbitrary direct sums and is generated by a set of compact elements. We fix a compact object $A\in \T$ and write $H^*$ for the functor $\T(A,-)^*$.

\begin{thm}\cite{Krcohom}\label{Henning}
Every exact localisation functor $(L,\Psi)$ on $\Modgr \Gamma$ extends to an exact localisation functor $(\hat{L},\hat{\Psi})$ on $\T$ such that the diagram
$$\xymatrix{\ar @{--} \T \ar[d]_{\hat{L}} \ar[rr]^-{H^*}& & \Modgr \Gamma \ar[d]^{L}\\
\T \ar[rr]^-{H^*}& & \Modgr \Gamma}$$
commutes up to a natural isomorphism.
More precisely, it holds
\begin{itemize}
\item[(1)] The morphisms $L H^* \hat{\Psi}, \Psi H^* \hat{L}$ and
$$\xymatrix@C-3pt{LH^* \ar[rr]^{LH^*\hat{\Psi}} & &  LH^*\hat{L} \ar[rr]^{(\Psi H^* \hat{L})^{-1}} & & H^*\hat{L}}$$
are invertible.
\item[(2)] An object $X$ in $\T$ is $\hat L$-acyclic if and only if $H^*X$ is $L$-acyclic.
\item[(3)] If $X \in \T$ is $\hat{L}$-local, then $H^*X$ is $L$-local. 
\end{itemize}
If $A$ is a generator of $\T$ and   $L$ preserves arbitrary direct sums, then $\hat{L}$ preserves arbitrary direct sums.
\end{thm}

\begin{rem}\cite[Rem.\ 2.4]{Krcohom}\label{alsoadjointscommute}
Let $L\colon \Modgr \Gamma \to \Modgr \Gamma$ be an exact localisation functor and denote by
$\hat L\colon \T \to \T$ the exact localisation functor which exists by Theorem~\ref{Henning}.
Write 
$\C$ for the $\hat L$-acyclic objects. By Lemma~\ref{LRQ}, $\hat L$ and $L$ give rise to adjoint pairs of functors
$$\xymatrix{\ar @{--} \T \ar@<-1ex>[r]_-{Q} & \T/{\C} \ar@<-1ex>[l]_-{ R} & \textrm{and}& \Modgr \Gamma \ar@<-1ex>[r]_-F & (\Modgr \Gamma)_L \ar@<-1ex>[l]_-G} $$
satisfying $\hat L= R \circ  Q$ and $L=G \circ F$. 
The diagram below commutes up to natural isomorphism. 
$$\xymatrix{\ar @{--} \T \ar[d]_{Q} \ar[rr]^-{\T(A,-)^*}  & & \Modgr \Gamma \ar[d]^{F}\\
\T /  \C \ar[d]_{R} \ar[rr]^-{\T / \C(QA,-)^*}& &  (\Modgr \Gamma)_L \ar[d]^{G} \\
\T \ar[rr]^-{\T(A,-)^*} & & \Modgr \Gamma
}$$
\end{rem}

Now assume that $\Gamma$ is graded-commutative and consider the localisation functor  $L\colon \Modgr \Gamma \to \Modgr \Gamma$ given by localisation with respect to a multiplicatively closed subset $S$ of~$\Gamma$. See Chapter \ref{grcom localisation} for localisation of graded-commutative rings. The following results in this section are joint work with K. Br\"uning \cite{BH}.

\begin{prop}\label{isomasrings}
Suppose that the ring $\T(A,A)^*$ is graded-commutative and let \linebreak
 $L\colon  \Modgr \T(A,A)^* \to \Modgr \T(A,A)^*$ be  localisation with respect to a multiplicatively closed subset of homogeneous elements  $S\subseteq \T(A,A)^*$.
Denoting $\C = \Ker \hat{L}$, the diagram
$$\xymatrix{\ar @{--} \T \ar[d]_{Q} \ar[rr]^-{\T(A,-)^*}  & & \Modgr \T(A,A)^* \ar[d]^{\can}\\
\T / \C \ar[rr]^-{\T /\C(QA,-)^*}& &  \Modgr S^{-1} \T(A,A)^*}$$
commutes up to natural isomorphism. Furthermore, $\T/\C(QA,QA)^*$ and $S^{-1}\T(A,A)^*$ are isomorphic not only as graded $\T(A,A)^*$-modules, but also as \emph{graded rings}. 
\end{prop}

\begin{proof}
The diagram commutes by Remark \ref{alsoadjointscommute}.
Writing again $H^*$ for $\T(A,-)^*$, the naturality of $H^*\Psi$ yields a commutative square
$$
\xymatrix{
H^*A\ar[rr]^-{\Psi H^*A} \ar[d]_-{H^*\hat{\Psi} A}& & LH^*A \ar[d]^-{L H^*\hat{\Psi} A}_-{\cong}\\
H^*\hat{L}A \ar[rr]^-{\Psi H^*\hat{L}A}_-{\cong}& & L H^*\hat{L}A
}
$$
in which the lower and the right hand side morphism are bijective by Theorem \ref{Henning}.
Now note that $H^*\hat{\Psi} A$ is up to isomorphism given by the canonical map $$Q\colon \T(A,A)^* \to \T/\C(QA,QA)^*, \quad f \mapsto Qf,$$
and that $\Psi H^*A$ equals up to isomorphism the canonical ring homomorphism
$$\can \colon \T(A,A)^*\to S^{-1}\T(A,A)^*.$$ 
Since $Q\colon \T/(A,A)^* \to \T/\C(QA,QA)^*$ is a multiplicative map inverting all elements in $S$, we obtain a ring homomorphism
$r\colon S^{-1}\T(A,A)^* \to \T/\C(QA,QA)^*$ which makes the upper triangle in the modified diagram
$$
\xymatrix{
\T(A,A)^*\ar[r]^-{\can} \ar[d]_-{Q} & S^{-1}\T(A,A)^*\ar[d]^-{S^{-1}Q}_-{\cong}\ar@{..>}[dl]_r\\
\T/\C(QA,QA)^* \ar[r]^-{\nu}_-{\cong} & S^{-1}\T/\C(QA,QA)^*
}
$$
commute. We now show that $r$ is bijective by proving the commutativity of the lower triangle.

Since both the maps $\nu\circ r$ and $S^{-1}Q$ make the following diagram of $\T(A,A)^*$-modules 
$$
\xymatrix{
  \T(A,A)^*\ar[r]^{\mu} \ar[d]_-{\nu \circ Q} & S^{-1}\T(A,A)^* \ar@<-3pt>[dl]_-{\nu\circ r}  \ar@<3pt>[dl]^-{S^{-1}Q}\\
S^{-1}\T/\C(QA,QA)^* &
}
$$
commute, the universal property of localisation of modules yields  $\nu\circ r = S^{-1}Q$ and hence the claim. 
\end{proof}

 \begin{prop}\label{C_pcompgen}
 Suppose that the ring $\T(A,A)^*$ is graded-commutative and let \linebreak
 $L\colon  \Modgr \T(A,A)^* \to \Modgr \T^*(A,A)^*$ be  localisation with respect to a multiplicatively closed subset of homogeneous elements  $S\subseteq \T(A,A)^*$.
If the compact object $A \in \T$ is a generator, then the category $\C = \Ker \hat{L}$ is generated by compact objects of~$\T$.
 \end{prop}
 \begin{proof}
We show that $\C$ is generated by $\{\cone(\sigma)\,\vert \, \sigma\colon A \to A[n] \in S\}$. Let $M$ be any object of $\C$. Using Lemma \ref{compactgen}, it is enough to show that $\T(\cone(\sigma),M)^* = 0$ for all $\sigma \in S$ implies $M=0$. 
By Lemma \ref{T(X,-)cohomol}, every triangle $$A \xto{\sigma} A[n] \to \cone(\sigma) \to A[1]$$ gives rise to an exact sequence 
$$\T(\cone(\sigma),M)^* \to \T(A[n],M)^* \xto{\T(\sigma,M)^*} \T(A,M)^* \to \T(\cone(\sigma)[-1],M)^*.$$
By assumption, we have  $\T(\cone(\sigma)[-1],M)^* = 0 = \T(\cone(\sigma),M)^*$. Hence 
the map $$\T(A[n],M)^* \xto{\T(\sigma,M)^*} \T(A,M)^*$$ is an isomorphism for all $\sigma \in S$ and thus, $\T(A,M)^*$ is $S$-local. On the other hand, $\T(A,M)^*$ is $S$-acylic and so we conclude that $\T(A,M)^*=0$. It follows that $M=0$ because $A$ is a compact generator.
 \end{proof}

\subsubsection{Cohomological $\p$-Localisation} \label{homlocp}
 Let $A$ be a dg algebra such that $H^*A$ is graded-commutative
and  let $\p$ be a graded prime ideal of $H^*A$.
Denote by $\C_{\p}$ the full subcategory of objects $X$ in $\D(A)$ such that $(H^*X)_{\p} = 0$. In other words,
$\C_{\p}$ is the kernel of the cohomological functor
$$(- \otimes_{H^*A} (H^*A)_{\p}) \circ \Ddg(A)(A,-)^*.$$

From the previous discussion we obtain
 \begin{cor}\label{HPSdia}
 The localisation $$\xymatrix{\ar @{--} \D(A)\ar@<-1ex>[r]_-Q & \D(A)/\C_{\p} \ar@<-1ex>[l]_-R}$$
 is smashing, and there is an isomorphism $r\colon \D(A)(A,A)^*_{\p} \xto{\cong} \D(A)/\C_{\p}(QA,QA)^*$
of graded rings making the diagram
{\small $$\xymatrix{
\D(A)(A,A)^*\ar[r]^-{\can} \ar[d]_-{Q} & \D(A)(A,A)_{\p}^* \ar[dl]_-{\cong}^-r\\
\D(A)/\C_{\p}(QA,QA)^*  & }
$$}
commutative. Furthermore, the squares
$$\xymatrix@C+4.2pt{\ar @{--} \Ddg(A) \ar[d]_Q \ar[rr]^-{\Ddg(A)(A,-)^*}& & \Modgr H^*A
\ar[d]^{- \otimes_{H^*A} (H^*A)_{\p}} && \Ddg(A)  \ar[rr]^-{\Ddg(A)(A,-)^*}& & \Modgr H^*A\\
\Ddg(A)/\C_{\p} \ar[rr]^-{\Ddg(A)/\C_{\p}(QA,-)^*}& & \Modgr H^*A_{\p} && \Ddg(A)/\C_{\p} \ar[u]_R \ar[rr]^-{\Ddg(A)/\C_{\p}(QA,-)^*}& & \ar[u]^{\inc}
\Modgr H^*A_{\p}}$$
commute up to natural isomorphism.\qed
\end{cor}

\section{Realising smashing localisations by morphisms of dg algebras}\label{seclift}
The results in this chapter are joint work with K. Br\"uning \cite{BH}, with substantial contributions of Bernhard Keller. We show that every smashing localisation on a derived category of a dg algebra can be realised by a morphism of dg algebras. More precisely, if $A$ is a dg algebra and $L\colon \D(A) \to \D(A)$ a smashing localisation, we prove the exis\-tence of a dg algebra $A_L$ with the property $\D(A)/\Ker L \simeq \D(A_L)$, a dg algebra $A'$ quasi-isomorphic to $A$ and a zigzag of dg algebra morphisms 
$$A \xleftarrow{\sim} A' \to A_L$$
which identifies in cohomology with the algebra map $L\colon \D(A)(A,A)^*\to \D(A)(LA,LA)^*$. 
If the dg algebra $A$ is cofibrant, then the algebra map $L$ identifies with the cohomology of a morphism $A \to A_L$, and the quotient functor is naturally isomorphic to the left derived functor $-\otimes_A^{\bfL} A_L$.

As an application, we consider in Section \ref{Apdef} dg algebras with graded-commutative cohomology ring. For such a dg algebra $A$, we introduce the \emph{localisation of $A$ at a prime $\p$ in cohomology} and denote this dg algebra by $A_{\p}$. It has the property $H^*(A_{\p}) \cong (H^*A)_{\p}$.  Moreover, we show that with this identification of graded algebras, the canonical morphism $H^*A\to (H^*A)_{\p}$ is induced by a zigzag of dg algebra morphisms.

\subsection{Construction of a dg algebra morphism}\label{sec construction}
Let $A$ be a differential graded algebra over some commutative ring
$k$ and let $$L\colon \D(A)\to \D(A)$$ be a smashing localisation. Denoting by $\C$ the category of $L$-acyclic objects, we have an adjoint pair of functors
$$\xymatrix{\ar @{--} \D(A)\ar@<-1ex>[r]_-Q & \D(A)/\C \ar@<-1ex>[l]_-R}$$
satisfying $R \circ Q=L$. The right adjoint $R$ is fully faithful and commutes with arbitrary direct sums.

Our first aim is to write the quotient category $\D (A)/\C$ as derived
category of a differential graded algebra $A_L$. 
Then we construct a zigzag of dg algebra morphisms $A \xleftarrow{\sim} A' \to A_L$ 
which induces 
the algebra morphism $$\D(A)(A,A)^* \to \D(A)(LA,LA)^*, \quad f \mapsto Lf,$$
in cohomology. 
For this purpose, we identify throughout this chapter  the functors $H^*\colon \D(A) \to \Modgr H^*A$ and $\D(A)(A,-)^*\colon \D(A) \to \Modgr H^*A$ (see Lemma \ref{D(A,-)^*=H^*-}).

The following lemma which we learned from Dave Benson is the key to our construction.
\begin{lem}\label{pullback}
Let $A$, $B$ dg algebras and $M$ a dg $(B,A)$-bimodule. Let $\alpha\colon A \to M$ and $\beta\colon B \to M$ be maps of dg modules which satisfy $\alpha(1) = \beta(1)$. Then 
$$X = \{(a,b) \in A \times B \; | \; \alpha(a) = \beta(b)\}$$
 is a dg algebra with differential 
$d_X = (d_A,d_B)$ and the projections $p_1, p_2$ in the pullback diagram
$$\xymatrix{X \ar[r]^{p_2} \ar[d]_{p_1} &  B \ar[d]^{\beta}  \\
A \ar[r]_{\alpha} & M\\}$$
are dg algebra morphisms.
If $\beta$ is a surjective \qis, then the diagram induces a pullback diagram in cohomology.
\end{lem}
\begin{proof}
The first assertions are immediately checked. For the last one we show 
that $H^*X = \{(\overline{a},\overline{b}) \in H^*A \times H^*B \; | \; H^*\alpha(\overline{a}) = H^*\beta(\overline{b})\}$.

A pair $(\overline{a},\overline{b}) \in H^*X$ trivially satisfies the property $H^*\alpha(\overline{a}) = H^*\beta(\overline{b})$ and consequently, the inclusion $\subseteq$ is always fulfilled. 

For the other inclusion we need to assume that $\beta$ is a surjective \qis. Let $(\overline{a},\overline{b}) \in H^*A \times H^*B$ such that $H^*\alpha(\overline{a}) = H^*\beta(\overline{b})$. We choose representing cocycles $a$  of $\overline{a}$ and $b$ of $\overline{b}$. Then $\alpha(a) - \beta(b) = m$ for some coboundary $m \in M$. Since $\beta$ is a surjective \qis, there is a coboundary $b' \in B$ such that $\beta(b') = m$. Hence the tuple $(a,b+b')$ satisfies $\alpha(a) = \beta(b+b')$ and thus $(\overline{a},\overline{b}) = (\overline{a},\overline{b+b'}) \in H^*X.$
\end{proof}

The following lemma ensures that the cohomology of the dg algebra $A_L$ which we construct below is independent of all choices that we will make. 

\begin{lem}\label{zigzag}
Let $X, Y$ be dg $A$-modules and let $\nu\colon X \to Y$ be an isomorphism in $\K(A)$. 
\begin{itemize}
\item[(1)] 
Denote by $X \to I(X)$ the injective hull of $X$ in the Frobenius category $\Moddg A$. There exists a dg algebra $S$ and a zigzag of quasi-isomorphisms of dg algebras
$$\END_A(X) \xleftarrow{\sim} S \xrightarrow{\sim} \END_A(Y\oplus I(X)).$$
\item[(2)] Let $I$ be any injective module in the Frobenius category $\Moddg A$.
There is a dg algebra $T$ and a zigzag of quasi-isomorphisms of dg algebras
$$\END_A(Y) \xleftarrow{\sim} T \xrightarrow{\sim} \END_A(Y\oplus I).$$
\item[(3)] There exists a zigzag of \qis s of dg algebras from $\END_A(X)$ to $\END_A(Y)$.
\end{itemize}
\end{lem}
\begin{proof}
(1) By Lemma \ref{choosesplit}, we can choose a representing dg $A$-module map $$\bar{\nu}\colon X \to Y \oplus I(X)$$
of $\nu \in \K(A)(X,Y)$ which is split as map of graded $A$-modules. Hence the map 
$$\bar{\nu}^*\colon \END_A(Y  \oplus I(X))\to\HOM_A(X,Y  \oplus I(X)), \quad f \mapsto f \circ \bar{\nu},$$ is surjective.
Applying Lemma \ref{pullback}, the pullback
$$\xymatrix{
S \ar[r]^-{p_2} \ar[d]_{p_1} &  \END_A(Y \oplus I(X)) \ar[d]|>{\object@{>>}}^-{\sim}_-{\bar{\nu}^*}\\
\END_A(X) \ar[r]^-{\bar{\nu}_*}_-{\sim} & \HOM_A(X,Y \oplus I(X)) & }$$
gives rise to a pullback diagram in cohomology, and the object $S$ is actually a dg algebra. In particular, we obtain quasi-isomorphisms of dg algebras 
$$\xymatrix{\END_A(X) &  \ar[l]^-{p_1}_-{\sim}  S  \ar[r]^-{\sim}_-{p_2} & \END_A(Y  \oplus I(X)).}$$

(2) The dg $A$-module map $\iota\colon Y \xto{\smatrix{\id & 0}} Y \oplus I$ is obviously a split monomorphism inducing $\id_Y$ in the homotopy category. Hence we obtain a pullback diagram
$$\xymatrix{
T \ar[r]^-{p_2}_-{\sim} \ar[d]_{p_1}^-{\sim} &  \END_A(Y \oplus I) \ar[d]|>{\object@{>>}}^-{\sim}_-{\iota^*}\\
\END_A(Y) \ar[r]^-{\iota_*}_-{\sim} & \HOM_A(Y,Y \oplus I) & }$$
yielding the claim.

(3) is a trivial consequence of (1) and (2).
\end{proof}




The proof of the following lemma is immediate.
\begin{lem}\label{QAcompact}
The object $QA$ is a compact generator of $\D(A)/\C$.\qed
\end{lem}

Fix a homotopically projective replacement of $RQA \in \D(A)$. By abuse of notation we denote the replacement also by $RQA$.
\begin{prop}\label{RHOM(RQA,R-)}
The functor $\bfR \HOM(RQA,R-)\colon \D(A)/\C \to \D(\END_A(RQA))$ is an equivalence
of triangulated categories.
\end{prop}
\begin{proof}
We use Proposition \ref{schwedeprop} and first note that $\bfR \HOM(RQA,R-)$ preserves arbitrary direct sums because for any family $(X_i)_{i\in I}$ in $\D(A)/\C$, the map
$$\coprod_{i \in I} \bfR \HOM(RQA,RX_i) \to \bfR \HOM(RQA,R\coprod_{i \in I}X_i)$$
identifies in cohomology with the isomorphism
{\small
$$\coprod_{i \in I} \D(A)(RQA,RX_i)  \cong 
\coprod_{i \in I} \D(A)/\C(QA,X_i)
 \cong \D(A)/\C(QA,\coprod_{i \in I}X_i)
 \cong \D(A)(RQA,R\coprod_{i \in I}X_i). $$}

Moreover, the functor $\bfR \HOM(RQA,R-)$ maps the compact
generator $QA$ of $\D(A)/\C$ to $\END_A(RQA)$ which compactly generates $\D(\END_A(RQA))$. Finally, the map
$$\D(A)/\C(QA,QA[n]) \xto{\bfR \HOM(RQA,R-)} \D(\END_A(RQA))(\END_A(RQA),\END_A(RQA)[n])$$
is an isomorphism for all $n \in \mathbb Z$ since $RQA$ being homotopically projective implies that the diagram
$$\xymatrix{\ar @{--} \D(A)/\C(QA,QA[n]) \ar[d]_{R}^{\cong}  \ar[rrr]^-{\bfR \HOM(RQA,R-)} & & &  \D(\END_A(RQA))(\END_A(RQA),\END_A(RQA)[n])\\
\D(A)(RQA,RQA[n]) \ar[rrr]_{\cong} & & & H^n(\END_A(RQA)) \ar[u]^{\cong}  }$$ is commutative.
\end{proof}

Hence we have shown that the quotient category $\D(A)/\C$ is equivalent to the derived category of the dg algebra $\END_A(LA)$, where $LA$ was chosen to be homotopically projective. Note that Lemma \ref{zigzag} provides a zigzag of \qis s between the endomorphism dg algebras of two different homotopically projective replacements of an object in $\D(A)$. 

In order to construct a zigzag $A \xleftarrow{\sim} A' \to \END_A(LA)$ of dg algebra morphisms  inducing 
$$\D(A)(A,A)^* \to \D(A)(LA,LA)^*$$ in cohomology, we need to make another choice on the dg $A$-module representing $LA$.


Let $\eta\colon \id \to RQ$ be the unit and $\varepsilon \colon QR
\to \id$ the counit of the adjunction
$$\xymatrix{\ar @{--} \D(A)\ar@<-1ex>[r]_Q & \D(A)/\C \ar@<-1ex>[l]_R.}$$

Since $A$ is homotopically projective, we have
$$\eta_A \in \D(A)(A,RQA) \cong \K(A)(A,RQA).$$

\begin{lem}\label{eta*qis}
For any map $\bar{\eta}_A$ in $\Moddg A$ that represents $\eta_A \in \K(A)(A,RQA)$ and any dg $A$-module $M$, the map 
$$\bar{\eta}_A^{\,*} \colon \HOM_A(RQA,RQM) \to \HOM_A(A,RQM),\quad f \mapsto f \circ \bar{\eta}_A,$$
is a quasi-isomorphism.
\end{lem}
\begin{proof}
Since $R$ is fully faithful, the usual adjunction isomorphism  (see \cite[Ch.\ IV.1]{Mac})
gives rise to the mutually inverse maps 
$$H^n (\bar{\eta}_A^{\,*}) \colon \D(A)(RQA,RQM[n]) \to \D(A)(A,RQM[n]),\quad f
\mapsto f \circ \eta_A,$$ and
\begin{equation*}
\D(A)(A,RQM[n]) \to \D(A)(RQA,RQM[n]), \quad g \mapsto
R(\varepsilon_{QA}) \circ RQ(g).\qedhere
\end{equation*}
\end{proof}

\begin{rem}\label{END(RQA)welldefined}
By Lemma \ref{choosesplit}, we may represent $\eta_A\colon A \to LA$ by a monomorphism of dg $A$-modules
$$\bar{\eta}_A\colon A \to \widehat{LA},$$
 which is a split as map of graded $A$-modules. Remember that $\widehat{LA}=LA \oplus I(A)$, where $A \to I(A)$ is the injective hull of $A$ in the Frobenius category $\Moddg A$, and $LA$ was already chosen to be homotopically projective. By Lemma \ref{zigzag}, we have a zigzag of \qis s   $$\END_A(LA) \xleftarrow{\sim} T \xrightarrow{\sim}\END_A(\widehat{LA}).$$ 
\end{rem}

We define the dg algebra $A_L$ to be $\END_A(\widehat{LA})$. By abuse of notation we write  $A_L=\END_A(LA)$.
Note that from Lemma \ref{zigzag} and Proposition \ref{RHOM(RQA,R-)}, it follows that 
$$\D(A_L) \simeq \D(A)/\C.$$

\begin{thm}\label{A'}
The algebra map
$$\D(A)(A,A)^* \to \D(A)(LA,LA)^*,\quad f \mapsto L(f),$$
is induced by a zigzag of dg algebra maps $$A \xleftarrow{\sim} A' \xto{\varphi} A_L.$$
That is, there exists a dg algebra $A'$ quasi-isomorphic to $A$, a morphism of dg algebras  $\varphi\colon A' \to A_L$ and in cohomology, we have the commutative diagram 
{$$\xymatrix{\ar @{--} H^*A' \ar[d]_{\cong} \ar[rd]^-{H^*\varphi}& \\
\D(A)(A,A)^* \ar[r]^-{L}&  \D(A)(LA,LA)^*}$$}
\end{thm}

\begin{proof}
We identify the dg algebras $\END_A(A)$ and $A$ through the isomorphism given by evaluation at $1$. 
Let 
$$\xymatrix{
A' \ar[r]^-{p_2} \ar[d]_-{p_1}^-{\sim} &  A_L \ar[d]|>{\object@{>>}}_-{\sim}^-{\bar{\eta}_A^{ \, *}}\\
\END_A(A) \ar[r]^-{\bar{\eta}_{A *}} & \HOM_A(A,LA) & }$$
be a pullback diagram.

The map $\bar{\eta}_{A *}$ is a quasi-isomorphism (Lemma \ref{eta*qis}), and surjective since $\bar{\eta}_A$ is a split monomorphism of graded $A$-modules (Remark \ref{END(RQA)welldefined}). We infer from Lemma \ref{pullback} that $A'$ is a dg algebra quasi-isomorphic to $A$, and we set $\varphi=p_2$.

In cohomology, we obtain the commutative diagram
$$
\xymatrix@C+7pt{
H^* A' \ar[r]^-{H^*(p_2)} \ar[d]_{H^*(p_1)}^{\cong} & H^* \END_A(LA) \ar[d]_{\cong}^{H^* (\bar{\eta}_A^{ \, *})}\\
H^* \END_A(A) \ar[r]^-{H^*(\bar{\eta}_{A *})}
& H^*(\HOM_A(A,LA)) & }
$$
and thus it remains to show that the composition
$$H^* (\bar{\eta}_A^{ \, *})^{-1} \circ  H^* (\bar{\eta}_{A *})$$
identifies with the map
$$\D(A)(A,A)^* \to \D(A)(LA,LA)^*,\quad f \mapsto L(f).$$

In fact, for $f \in \D(A)(A,A)^*$ we have
$$\big{(}H^* (\bar{\eta}_A^{ \, *})^{-1} \circ  H^* (\bar{\eta}_{A *})\big{)}(f) = R(\varepsilon_{QA}) \circ RQ(\eta_A) \circ RQ(f) \, \in \D(A)(RQA,RQA).$$
But it is well-known that $\varepsilon_{QA} \circ Q(\eta_A) \cong \id_{QA}$ (see 
\cite[Ch.\ IV.1]{Mac}) and hence  the claim follows.
\end{proof}
Observe that the map $\varphi\colon A' \to A_L$ is a monomorphism: Since $\bar \eta_A$ is  split as map of graded $A$-modules, $\bar{\eta}_{A *}$ is injective and so is $\varphi = p_2$.

If we assume in addition that $A$ is a {cofibrant} dg algebra (see Section \ref{cofibrantdga}), then the algebra map 
$L\colon \D(A)(A,A)^* \to \D(A)(LA,LA)^*$ is not only induced by a zigzag of dg algebra maps, but by a \emph{morphism} $A\to A_L$.

\begin{cor}\label{Acofibrant}
Let $A$ be a cofibrant dg algebra. The algebra morphism 
$$\D(A)(A,A)^* \to \D(A)(LA,LA)^*,\quad f \mapsto L(f),$$ lifts to a dg algebra morphism $\psi \colon A \to A_L$.
\end{cor}

\begin{proof}
Since $A$ is cofibrant, the map $p_1\colon A' \to A$ in the pullback diagram
$$\xymatrix{
A' \ar[r]^-{p_2} \ar[d]|>{\object@{>>}}_-{p_1}^{\sim} &  A_L \ar[d]|>{\object@{>>}}_{\sim}^{\bar{\eta}_A^*}\\
A \ar[r]^-{\bar{\eta}_{A *}} & \HOM_A(A,LA) & }$$
splits: There is a morphism of dg algebras $s\colon A \to A'$ such that $p_1 \circ s = \id_A$.
We define $\psi$ to be the composition 
$$p_2 \circ  s \colon A \to A_L.$$
Then $H^*\psi$ identifies with the canonical map $L\colon \D(A)(A,A)^* \to \D(A)(LA,LA)^*$.
\end{proof}

Now our aim is to show that if $A$ is cofibrant, then we can identify the  functors $Q\colon \D(A) \to \D(A)/\C \simeq \D(A_L)$ and $-\otimes_A^{\bfL} A_L\colon \D(A) \to \D(A_L)$, where $A_L$ is a dg $(A,A_L)$-bimodule through the morphism $\psi\colon A \to A_L$.

\enlargethispage*{1cm}
\begin{lem}\label{RHOMmap}
There exists a natural transformation
$$
\lambda\colon \bfR \HOM_A(A,-)\to \bfR \HOM_A(LA,L-)
$$ 
in $\D(A)$ which commutes with the suspension functor. For every $M \in \D(A)$,  $\lambda_M$ induces the map $$\D(A)(A,M)\to\D(A)(LA,LM), \quad f \mapsto Lf,$$ in cohomology.
\end{lem}

\begin{proof} 
By Lemma \ref{eta*qis}, the adjunction unit $\eta_A\colon A \to LA$ induces  a natural  isomorphism \linebreak 
$\bfR \HOM_A(\eta_A,LM)$. Therefore we can define the morphism $\lambda_M$ to be the composition
$$\xymatrix{\bfR \HOM_A(A,M) \ar[rr]^-{\bfR \HOM(A,\eta_M)} & &
\bfR \HOM_A(A,LM)  \ar[rrr]^-{\bfR \HOM(\eta_A,LM)^{-1}}_-{\cong} & && \bfR \HOM_A(LA,LM),}$$
which obviously induces $L\colon \D(A)(A,M)^*\to \D(A)(LA,LM)$ in cohomology.
The naturality of $\lambda_M$ follows from the naturality of $\bfR \HOM(A,\eta_M)$ and $\bfR \HOM(\eta_A,LM)$.

The unit $\eta$ of the adjoint pair $(Q,R)$ commutes with the suspension functor $[1]$,  hence so does $\bfR \HOM(A,\eta_M)$. Since $\bfR \HOM(\eta_A,LM)$ commutes with $[1]$, we conclude that $\lambda \circ [1] \cong [1] \circ \lambda$.
\end{proof} 

Note that if $A$ is cofibrant, then $\lambda_A$ equals the dg algebra morphism $\psi\colon A \to A_L$ constructed in Corollary \ref{Acofibrant}. In addition, $\bfR \HOM_A(LA,LM)$ becomes an object in $\D(A)$ through the dg algebra morphism $\psi$.

\begin{prop} \label{functors isom}
Suppose that $A$ is a cofibrant dg algebra. Then the diagram
$$\xymatrix{\D(A)\ar[d]_-Q \ar[rr]^-{-\otimes_A^{\bfL} A_L} && \D(A_L)\\
\D(A)/\C \ar[rru]_-{\quad \quad \bfR \HOM_A(RQA,R-)}^{\simeq} && }$$
commutes up to natural isomorphism. 
\end{prop}
\begin{proof}
We show that the functors $\bfR \HOM_A(LA,L-)$ 
and 
$-\otimes^{\bfL}_A A_L$ 
are naturally isomorphic. A natural transformation
$$\tau\colon -\otimes^{\bfL}_A A_L \longrightarrow \bfR \HOM_A(LA,L-)$$ is given as composition of the three natural maps in the diagram
$$
\xymatrix{
M\otimes_A^{\bfL} A_L \ar@{..>}[rr]^-{\tau_M} \ar[d]_-{\can_M \otimes_A^{\bfL} A_L}^-\cong & &\bfR \HOM_A(LA,LM)\\
\bfR \HOM_A(A,M)\otimes_A^{\bfL} A_L \ar[rr]^-{\lambda_M \otimes_A^{\bfL} A_L} & & \bfR \HOM_A(LA,LM)\otimes_A^{\bfL} A_L \ar[u]_-{\nu_M}
}
$$
where $\can_M$ is the canonical identification and $\nu_M$ is defined by
\begin{eqnarray*}
\nu_M\colon \bfR \HOM_A(LA,LM)\otimes_A^{\bfL} \END_A(LA) & \to & \bfR \HOM_A(LA,LM).\\
f \otimes g & \mapsto & f \circ g
\end{eqnarray*}
Note that $\tau$ commutes with the suspension functor since this holds for $\lambda$ by Lemma~\ref{RHOMmap}, and obviously for $\can$ and $\nu$.
In order to prove that $\tau$ is an isomorphism, it suffices to show that
the full subcategory $\A=\{M\in\D(A)\,|\,\tau_M \text{ is an isomorphism}\}$ of $\D(A)$ is a localising subcategory containing $A$.

First we point out that $\A$ is closed under triangles by the Five-Lemma for triangulated categories, and that it is easy to check that $A \in \A$. 

Furthermore, $\A$ is closed under taking shifts because $\tau$ commutes with the suspension functor.
 
Finally, we show that $\A$ is closed under taking arbitrary direct sums. To that end, recall that the functor $\bfR \HOM_A(LA,R-)$ commutes with arbitrary direct sums (Proposition \ref{RHOM(RQA,R-)}) and hence, so does $\bfR \HOM_A(LA,L-)$. Since $-\otimes^{\bfL}_A A_L$ obviously commutes with arbitrary direct sums, the claim follows.
\end{proof}

\begin{rem}
Let $\T$ be an algebraic triangulated category in the sense of Keller. We refer to \cite[Sect.\ 6.5]{Kchicago} for a definition. If $\T$ is generated by a single compact object, then there is a dg algebra $A$ 
such that  $\T\simeq \D(A)$, see \cite[Sect.\ 6.5]{Kchicago}. Since we can choose the dg algebra $A$ to be cofibrant (see Lemma \ref{cofibrantreplacement}), every smashing localisation $L\colon \T \to \T$ is induced by a morphism of dg algebras $A \to A_L$.
\end{rem}

\subsection{The $\p$-localisation of a dg algebra}\label{Apdef}
Let $A$ be a dg algebra over a commutative ring $k$ and assume throughout this section that the cohomology algebra $H^*A$ is graded-commutative. We fix a graded prime ideal $\p $ of $H^*A$. By $\C_{\p}$ we denote the full subcategory of objects $M$ in $\D(A)$ such that $(H^*M)_{\p}=0$. The localisation $L_{\p}\colon \D(A) \to \D(A),$ given by the adjoint pair
$$\xymatrix{\ar @{--} \D(A)\ar@<-1ex>[r]_-Q & \D(A)/\C_{\p} \ar@<-1ex>[l]_-R,}$$
is smashing by Corollary \ref{HPSdia}. Now we apply the results of Section \ref{sec construction} to this special case. We define 
$$A_{\p} = A_{L_{\p}},$$
and we call $A_{\p}$ \emph{localisation of $A$ at a prime $\p$ in cohomology}.

From Lemma \ref{zigzag} and Proposition \ref{RHOM(RQA,R-)}, we infer that $\D(A)/\C_{\p} \simeq \D(A_{\p})$. For this special smashing localisation, we have

\begin{thm}\label{A'->A_p}
Let $A$ be a dg algebra over a commutative ring $k$ such that $H^*A$ is graded-commutative and let $\p$ be a graded prime ideal of $H^*A$. The dg algebra $A_{\p}$ has the property $H^*(A_{\p}) \cong (H^*A)_{\p}$. Moreover, with this identification of graded algebras, 
the canonical map
$$\can\colon H^*A \to (H^*A)_{\p}$$
is induced by a zigzag of dg algebra maps $$A \xleftarrow{\sim} A' \xto{\varphi} A_{\p}.$$
That is, we have a commutative diagram
$$\xymatrix{\ar @{--} H^*A' \ar[d]_{\cong} \ar[rrd]^-{H^*\varphi}&& \\
H^*A \ar[r]_-{\can}&  (H^*A)_{\p}  \ar[r]_-{\cong} & H^*(A_{\p})}$$
\end{thm}

\begin{proof}
Since $\D(A)(A,A)_{\p} \cong \D(A)(L_{\p}A,L_{\p}A)$ by Corollary \ref{HPSdia}, the dg algebra $A_{\p}$ satis\-fies $H^*(A_{\p}) \cong (H^*A)_{\p}.$
Theorem \ref{A'} shows that the zigzag $A \xleftarrow{\sim} A' \xto{\varphi} A_{\p}$
induces the map
$$L_{\p}\colon \D(A)(A,A)^* \to \D(A)(L_{\p}A,L_{\p}A)^*,\quad f \mapsto L_{\p}(f)$$
in cohomology. But we may identify the algebra maps $\can$ and $L_{\p}$ 
by Corollary \ref{HPSdia}. 
\end{proof}
The following result is an immediate consequence
of Corollary~\ref{Acofibrant} and Theorem~\ref{A'->A_p}. 
\begin{cor}\label{A->A_p}
Let $A$ be a cofibrant dg algebra such that $H^*A$ is graded-commutative and let $\p$ be a graded prime ideal of $H^*A$. Then the canonical algebra morphism 
$\can\colon H^*A \to (H^*A)_{\p}$ lifts to a dg algebra morphism 
$$\xymatrix@C-3.2pt{&&&&&& \psi\colon A \ar[r] & A_{\p}.&&&&&&\qed}
$$
\end{cor}
A class of cofibrant dg algebras with graded-commutative cohomology are the Sullivan algebras $(\Lambda V,d)$ introduced Section \ref{cofibrantdga}.

\begin{rem}
The smashing localisation $L_{\p}\colon \D(A)\to \D(A)$ can be interpreted as $\p$-localisation on the derived category. It satisfies a \emph{local-global principle}:

For $M \in \D(A)$ and a graded prime $\p$ of $H^*A$, we define
$$M_{\p}=\bfR \HOM_A(L_{\p}A,L_{\p}M) \in \D(A_{\p})$$
and call $M_{\p}$ \emph{localisation of $M$ at a graded prime ideal $\p$ of $H^*A$}. If $A$ is a cofibrant dg algebra, then we have  $M_{\p}=M \otimes_A^{\bfL} A_{\p}$ by Proposition \ref{functors isom}. Since $(H^*M)_{\p}\cong H^*(M_{\p})$,  the following conditions are equivalent for $M \in \D(A)$:
\begin{itemize}
\item[(1)] $M = 0$.
\item[(2)] $M_{\p} = 0$ for all graded prime ideals $\p$.
\item[(3)] $M_{\m} = 0$ for all graded maximal ideals $\m$.
\end{itemize}
\end{rem}
\subsubsection{A universal property of $A_{\p}$}\label{sec universal property}
Let $A$ be a dg algebra over a commutative ring $k$ such that $H^*A$ is graded-commutative and $\p \in  \grspec(H^*A)$.

The cohomology of the dg algebra $A_{\p}$ satisfies a universal property since $H^*(A_{\p})$ is isomorphic to the ring of fractions $S^{-1}(H^*A) = (H^*A)_{\p}$, where $S$ is the subset of homogeneous elements in $H^*A \setminus \p$. If $\beta\colon A \to B$ is a morphism of dg algebras such that $H^*\beta$ makes $S$ invertible, then $H^*\beta$ factors uniquely over the canonical morphism $\can\colon H^*A \to (H^*A)_{\p}$. 

Without loss of generality, we assume from now on that $A$ is cofibrant. Then $\can$ is induced by a morphism of dg algebras $\psi\colon A \to A_{\p}$ and the universal property yields a unique algebra morphism $g\colon H^*(A_{\p}) \to H^*B$ 
making the following diagram commute:
$$\xymatrix{
H^*A \ar[r]^-{H^*\beta} \ar[d]_-{H^*\psi} & H^*B\\
H^*(A_{\p})  \ar@{..>}[ru]_g & }$$

The dg algebra morphisms $\beta\colon A \to B$ and $\psi\colon A \to A_{\p}$
give rise to functors $$F_{\beta}\colon \D(A) \xto{-\otimes_A^{\bfL} B} \D(B) \quad \textrm{and} \quad F_{\psi}\colon \D(A) \xto{-\otimes_A^{\bfL} A_{\p}} \D(A_{\p}).$$
Now we prove a universal property on the level of derived categories.

\begin{prop}\label{univ prop}
There is a unique functor $G \colon \D(A_{\p}) \to \D(B)$ making the following diagram commute:
$$\xymatrix{
\D(A) \ar[r]^-{F_{\beta}} \ar[d]_-{F_{\psi}} & \D(B)\\
\D(A_{\p})  \ar@{..>}[ru]_G & 
}
$$
\end{prop}
\begin{proof}
We first note that by Proposition \ref{functors isom}, the functor $F_{\psi}$ is nothing but the quotient functor $Q\colon \D(A) \to \D(A)/C_{\p}$ composed with the equivalence $\D(A)/\C_{\p} \simeq \D(A_{\p})$. Thus we can use the universal property of $Q$ and only need to show that $F_{\beta}(\C_{\p}) = 0$.

In Proposition \ref{C_pcompgen} we have shown that
 $$\M=\{\cone(\sigma)\,\vert \, \sigma\colon A \to A[n] \in S\}$$ is a set of compact generators of $\C_{\p}$ and thus it suffices to check that $F_{\beta}$ vanishes on $\M$.

Any element of $\M$ fits into an exact triangle
$$
A \xto{x\cdot} A[n] \to \cone(x\cdot) \to A[1]
$$
in $\D(A)$, where $x\cdot$ denotes multiplication with an element $x \in A$  whose cohomology $H^*x$ belongs to $S$. 
Applying the functor $F_{\beta}$ to this triangle, we obtain a triangle in $\D(B)$ naturally isomorphic to
$$B \xto{\beta(x)\cdot} B[n] \to F_{\beta}(\cone(x\cdot)) \to B[1].$$
Since $H^*\beta(x)$ is invertible, we infer that $F_{\beta}(\cone(x\cdot))$ is contractible and consequently, the object
$F_{\beta}(\cone(x\cdot))$ is zero in $\D(B)$.
\end{proof}

Since $\C_{\p}$ is generated by compact elements, the quotient functor $Q\colon \D(A) \to \D(A)/\C_{\p}$ gives rise to a quotient functor $\D^{\per}(A)\to \D^{\per}(A)/\C_{\p}^{\per}$, where $\C_{\p}^{\per}= \C_{\p} \cap \D^{\per}(A)$. Furthermore,
this quotient functor identifies with the functor
$$\D^{\per}(A) \xto{-\otimes_A^{\bf L} A_{\p}} \D^{\per}(A_{\p}).$$
This proves
\begin{cor}\label{univ prop perf}
There is a unique functor $G \colon \D^{\per}(A_{\p}) \to \D^{\per}(B)$ which makes the following diagram commute:
$$\xymatrix@C-2.3pt{&&&&&
\D^{\per}(A) \ar[r]^-{F_{\beta}} \ar[d]_-{F_{\psi}} & \D^{\per}(B)&&&&&\\
&&&&&\D^{\per}(A_{\p})  \ar@{..>}[ru]_G &  &&&&& \,\qed \\
}$$
\end{cor}
\begin{Rem}
The discussion above raises the question whether the functor $G\colon \D(A_{\p})  \to \D(B)$
and with it the algebra map $g\colon H^*(A_{\p}) \to H^*B$
can be lifted to a zigzag of dg algebra morphisms.
Our construction in Section \ref{sec construction} does not apply since in general, we cannot expect that $G$ is  a smashing localisation. It remains to enlighten the relation of our construction with DG quotients, which have a universal property and 
were introduced by Drinfeld \cite{Dr}.

There is also a construction by To\"en \cite{T} (see also in \cite{K-dgcat}) which seems to be related: Let $\mathrm{dgcat}_k$ be the category of small dg categories over a commutative ring $k$. The localisation of $\mathrm{dgcat}_k$ with respect to the quasi-equivalences is denoted by $\mathrm{Hqe}$. 
If $\A$ is a small dg category and $S$ a set of morphisms in $H^0(\A)$, then a morphism $F\colon \A \to \B$ in $\mathrm{Hqe}$ is said to make $S$ invertible if the induced functor $H^0(\A) \to H^0(\B)$ takes each $s \in S$ to an isomorphism. 
To\"en constructs a morphism $Q\colon \A \to \A[S^{-1}]$ in $\mathrm{Hqe}$ which makes $S$ invertible. This morphism has a universal property: Each morphism in $\mathrm{Hqe}$ making $S$ invertible factors uniquely through $Q$. 

However, if $\A$ is a dg algebra, viewed as dg category with a single object, then the object $\A[S^{-1}]$ is in general not a dg algebra, but a dg category with more than one object.
\end{Rem}

\section{Realisability}\label{sec Realisability}
In this chapter we introduce the concept of realisability as considered in a paper of Benson, Krause and Schwede \cite{BKS}. They are concerned with deciding whether a graded module over the Tate cohomology ring $\hat H^*(G,k)$, where $G$ is a finite group and $k$ a field, is isomorphic to $\hat H^*(G,M)$ for some $kG$-module $M$.

More generally, they consider a triangulated category $\T$ admitting arbitrary direct sums
and a cohomological functor $H^*\colon \T \to \Modgr E$ into the category of graded modules over a graded ring $E$. The functor $H^*$ is required to preserve arbitrary direct sums and products.
Then realisability deals with deciding whether a graded $E$-module is isomorphic to a module in the image of this cohomological functor.

Benson, Krause and Schwede~\cite{BKS} have stated a local obstruction for realisability up to direct summands, and a criterion for realisability which is given by an infinite sequence of obstructions.


If $A$ is a dg algebra over a field $k$, then the functor in question is the cohomology functor $H^*\colon \D(A) \to \Modgr H^*A$.
In this setting, Benson, Krause and Schwede also prove the existence of a global obstruction for realisability up to direct summands.


In the first section we introduce the general setup of Benson, Krause and Schwede~\cite{BKS} and recall the construction of the local obstruction. 
After focusing on realisability in the setting of dg algebras in the second section, we  study the global obstruction and its basic properties in Section \ref{sec global obstruction}.

\subsection{A local obstruction for realisability}\label{realtria}
Let $\T$ be a triangulated category admitting arbitrary direct sums. We denote the suspension functor by $\Sigma$. Let $N$ be a compact object in $\T$ and $E = \T(N,N)^* = \coprod_{i \in \bbZ}\T(N,\Sigma^iN) $ the graded endomorphism ring of $N$ (see Example \ref{examgradendring}). 
If $X$ is a graded $E$-module, then we denote by $X[n]$ the $n$-fold shifted module.

If $M$ is an object in $\T$, then we obtain a graded $E$-module $\T(N,M)^*$
by composition of graded maps. On the other hand, given any graded $E$-module $X$, when is $X$ isomorphic to $\T(N,M)^*$ for some $M \in \T$?


We will mainly consider this question only up to direct summands.
Therefore, from now on, we use the following terminology for realisability\footnote{Our terminology is different from the one in \cite{BKS}. Benson, Krause and Schwede call an $E$-module realisable if $X \cong \T(N,M)^*$ for some $M \in \T$.}:
\begin{defn}
Let $X$ be a graded $E$-module. 
We call  $X$ \emph{realisable} if there exists an object $M \in \T$ such that $X$ is isomorphic to a direct summand of $\T(N,M)^*$. If $X \cong \T(N,M)^*$ for some $M \in \T$, then we call $X$ \emph{strictly realisable}.
\end{defn}

\begin{rem}
Note that the functor $\T(N,-)^*$ occurs in a very natural way: For every graded ring $\Lambda$ and every cohomological functor
$H^*\colon \T \to \Modgr \Lambda$
 which preserves arbitrary direct sums and products, there exists a compact object $C \in \T$
 such that $H^*$ is naturally isomorphic to $\T(C,-)^*$ \cite[Lemma 3.2]{Krcohom}.
\end{rem}


Benson, Krause and Schwede \cite{BKS} have constructed a \emph{local obstruction} for realisability:
\begin{thm}\cite[Thm.\ 3.7]{BKS}\label{BKSlocal}
Let $\T$ be a triangulated category with arbitrary direct sums, $N \in \T$ 
a compact object and $E = \T(N,N)^*$ the graded endomorphism algebra of $N$.
For each graded $E$-module $X$, there exists an element
$$\kappa(X) \in \Ext^{3,-1}_E(X,X)$$ determining the realisability:
$X$ is realisable if and only if $\kappa(X)$ is trivial.
\end{thm}
For the proof of
Theorem \ref{BKSlocal} and more details we refer to \cite{BKS}. However, for our purposes
we sketch the construction of the local obstruction $\kappa(X)$:

\begin{construction}\label{T-presentation}
A {\em $\T$-presentation} of a graded $E$-module $X$ consists of a
distinguished triangle
\begin{equation} \label{eq:T-presentation}
\Sigma^{-1}B \xrightarrow{\delta} R_1 \xrightarrow{\alpha} R_0
\xrightarrow{\pi} B
\end{equation}
in $\T$ together with an epimorphism of graded $E$-modules
$\ep \colon\T(N,R_0)^*\to X$ such that the sequence
\[ \T(N,R_1)^* \ \xrightarrow{\alpha_*} \ \T(N,R_0)^* \ \xrightarrow{\ep}
 X \to 0  \]
is exact. If the objects $R_0$ and $R_1$ are assumed to be $N$-free, that is, 
isomorphic to a direct sum of shifted copies of $N$, then  we refer to an {\em $N$-special}
$\T$-presentation.

Given an $N$-special $\T$-presentation
$$ (\Sigma^{-1}B \xrightarrow{\delta} R_1 \xrightarrow{\alpha} R_0
\xrightarrow{\pi} B,\ \ep\colon\T(N,R_0)^*\to X),$$
we obtain an exact sequence of graded $E$-modules
\begin{equation} \label{eq:Yoneda class}
 0 \rightarrow X[-1] \xrightarrow{\eta[-1]} \T(N,B)^*[-1]
\xrightarrow{\delta_*}
\T(N,R_1)^* \xrightarrow{\alpha_*} \T(N,R_0)^*
\xrightarrow{\ep} X \rightarrow 0 \  \end{equation}
by applying the functor $\T(N,-)^*$ to the triangle. The monomorphism $\eta\colon X\to\T(N,B)^*$ is determined by
$\eta\circ\ep=\pi_*$.

This sequence is called {\em associated extension of the $N$-special
$\T$-presentation.}
\end{construction}

In \cite[Prop.\ 3.4]{BKS} it is shown that there exists an $N$-special
$\T$-presentation for each graded $E$-module $X$, and that the Yoneda-class of the associated extension, denoted by
 $\kappa(X) \in \Ext^{3,-1}_E(X,X)$,  is independent of the
choice of the $N$-special $\T$-presentation. 

Since $\T(N,R_0)^*$ and $\T(N,R_1)^*$ are free, the extension $\kappa(X)$ is trivial if and only if the monomorphism
$$X[-1] \xrightarrow{\eta[-1]} \T(N,B)^*[-1]$$ is split. Thus $\kappa(X)$ determines, in fact, the realisability of $X$.

\begin{rem}\label{postnikov}
Let $\T$ be a triangulated category admitting direct sums and $N \in \T$ 
a compact object. Denote by $E = \T(N,N)^*$ the graded endomorphism algebra of $N$.
Benson, Krause and Schwede extend their theory by an infinite sequence of obstructions which decides whether a graded $E$-module $X$ is strictly realisable.

They show that if there exists an infinite sequence of obstructions 
$$\kappa_n(X) \in \Ext^{n,2-n}_{E}(X,X),\; n \ge 3,$$
where the class $\kappa_n(X)$ is defined provided that the previous one $\kappa_{n-1}(X)$ vanishes,  then it even holds $X \cong \T(N,M)^*$. In this sequence of obstructions all but the first one depend on choices. 
Only $\kappa_3(X)$ is uniquely determined and equals the local obstruction $\kappa(X)$.
 
In view of the need for an \emph{infinite} sequence of obstructions to decide whether \linebreak $X \cong \T(N,M)^*$ for some $M \in \T$, it is remarkable that the first obstruction of this sequence already tells whether $X$ is a direct summand of $\T(N,M)^*$.

In our results we will only consider this first obstruction.
\end{rem}

\subsection{Realisability and dg algebras}
Let $A$ be a dg algebra over some commu\-ta\-tive ring $k$. Remember that the functor $\D(A)(A,-)^*$ is naturally isomorphic to the cohomology functor  $H^*$ (Lemma \ref{D(A,-)^*=H^*-}). Hence a graded $H^*A$-module is realisable if and only if it is a direct summand of $H^*M$, where $M \in \D(A)$ is a dg $A$-module.
Moreover,  $A$ 
is a compact object in $\Ddg(A)$. Hence Theorem \ref{BKSlocal}  applies, and we can decide whether $X$ is a direct summand of $H^*M$ for some dg $A$-module $M \in \D(A)$.

\begin{exm}\label{example for real}
(1) Let $k$ be a Noetherian ring and $G$ a finite group. Remember that the group cohomology ring $H^*(G,k)$ is actually a graded endomorphism ring of an object in a triangulated category: We have $H^*(G,k)\cong \bfK(\Inj kG)(ik,ik)^*$, where $ik$ denotes an injective resolution of $k$. Moreover, $\Kinjg$ admits arbitrary direct sums since $\Inj kG$ does. 
Since $ik$ is compact in $\Kinjg$ by Proposition~\ref{Kinjcompactlygen}, the assumptions of Theorem~\ref{BKSlocal} are satisfied.

On the other hand, we know from Remark \ref{dgremark}(2) that $\bfK(\Inj kG)(ik,ik)^*$ is the cohomology of the endomorphism dg algebra $\END(ik)$ of the complex $ik$,
$$H^*\END(ik) \cong H^*(G,k),$$
and we can also consider realisability in the setting of dg algebras. Now note that we have a commutative square
 $$\xymatrix{ {\bf K}(\Inj kG) \ar@{>>}[d]_{\HOM(ik,-)} \ar[rrr]^-{{\bf K}(\Inj kG)(ik,-)^*}& & & \Modgr H^*(G,k)
\ar[d]^-{\cong} \\
\D(\END(ik)) \ar[rrr]^-{H^*}&&  & \Modgr H^* \END(ik) }$$
Since the exact functor $\HOM(ik,-)\colon \Kinjg \to \D(\END(ik))$ commutes with arbitrary direct sums and is dense (see for example  \cite[Prop.\ 3.1]{BKcompl}), a graded $H^*(G,k)$-module is  realisable by a complex $C \in \Kinjg$ if and only if it is  realisable by a dg $\END(ik)$-module $M$.

Observe that $H^*(G,k)$ is also isomorphic to the graded endomorphism ring \linebreak $\bfD(kG)(k,k)^*$. However, except in trivial cases, the stalk complex $k$ is not a compact object in $\bfD(kG)$  and so the assumptions for Theorem \ref{BKSlocal} are not fulfilled.

(2) Let $k$ be a field and $G$ be a finite group. The Tate cohomology ring $\hat{H}^*(G,k)$ is the graded endomorphism ring $\uHom_{kG}(k,k)^*$ and $k$ is compact in $\uMod kG$. 
With the equivalence $Z^0\colon \bfK_{ac}(\Inj kG) \xto{\simeq} \uMod kG$ (Lemma \ref{kac=stmod}), we can write $\hat{H}^*(G,k)$ also as graded endomorphism ring $\bfK_{ac}(\Inj kG)(tk,tk)^*$ and  conclude that 
$$H^*\END(tk) \cong \hat{H}^*(G,k).$$
Similarly as in $(1)$, we have a commutative square
 $$\xymatrix{ \bfK_{ac}(\Inj kG) \ar@{>>}[d]_{\HOM(tk,-)} \ar[rrr]^-{\bfK_{ac}(\Inj kG)(tk,-)^*}& & & \Modgr \hat{H}^*(G,k)
\ar[d]^-{\cong} \\
\D(\END(tk)) \ar[rrr]^-{H^*}& & & \Modgr H^* \END(tk) }$$
and $\HOM(tk,-)$ is exact, dense and preserves arbitrary direct sums. Hence a 
graded $\hat{H}^*(G,k)$-module is  realisable by an acyclic complex (or equivalently, by an object of $\uMod kG$) if and only if it is  realisable by an object of $\D(\END(tk))$.
\end{exm}
Note that if $k$ is a field of characteristic $p>0$ and $G$ a $p$-group, then the functors $\HOM(ik,-)$ and $\HOM(tk,-)$ are even equivalences \cite{BKcompl}.

Considering  realisability in the setting of dg algebras has a striking advantage, as the following section indicates.
\subsection{A global obstruction for realisability}\label{sec global obstruction}
Let $A$ be a dg algebra over a field $k$. We have seen in Chapter \ref{sectionAinf} that 
$H^*A$ is an $A_{\infty}$-algebra whose secondary multiplication determines a Hochschild class $\mu_A \in \HH^{3,-1}(H^*A)$, called the canonical class. Now the importance of this result for the realisability theory becomes evident:

\begin{thm}\cite[Thm.\ 6.2]{BKS}\label{BKSglobal}
Let $X$ be a graded $H^*A$-module. The realisability obstruction
$$\kappa(X) \in \Ext^{3,-1}_{H^*A}(X,X)$$ is given by the cup product pairing
$$\id_X \cup \mu_A.$$
In particular, if the class $\mu_A \in \HH^{3,-1}(H^*A)$ is trivial, then
all graded $H^*A$-modules are realisable.
\end{thm}
Because of the last property the class $\mu_A$ is referred to
as \emph{global obstruction}. 
\begin{Rem}\label{converse does not hold}
The converse of the last statement in Theorem \ref{BKSglobal} is not true in general: Benson, Krause and Schwede provide an example of a dg algebra $A$ with the property  that all $H^*A$-modules are realisable, but with \emph{non-trivial} canonical class $\mu_A$ \cite[Exm.\ 5.15]{BKS}.
\end{Rem}
\section{Realisability and $\p$-localisation}\label{sec Realisability and localisation}
In this chapter we study the relation between realisability of modules over graded-commutative cohomology rings and $\p$-localisation.

A first motivation for our use of $\p$-localisation is the problem when a module over the graded-commutative ring $H^*(G,k)$ is isomorphic to  $H^*(G,M)$ for some $kG$-module $M$. 
This is discussed in Section \ref{sec motivation}.

More generally,  we consider dg algebras  with graded-commutative cohomology rings  in Section \ref{realislocal}. In our main result of this chapter we prove that the classical local-global principle of Commutative Algebra applies for realisability.

\subsection{A motivation for $\p$-localisation}\label{sec motivation}
Let $G$ be a finite group and $k$ a field such that $\Char(k)$ divides the order of $G$. 
Let $X$ be a graded module over the group cohomology ring $H^*(G,k)$. With Theorem \ref{BKSlocal} we can determine whether $X$ is  realisable by a \emph{complex} in the homotopy category $\Kinjg$. However, we rather want to know when $X$ 
is realisable by a \emph{module}, i.e.\ $X$ is a direct summand of $H^*(G,M)$ with $M \in \Mod kG$, or even $X \cong H^*(G,M)$.

The category $\Mod kG$ is embedded in $\Kinjg$, but with Theorem  \ref{BKSlocal}  we cannot decide whether an arbitrary  realisable $H^*(G,k)$-module can actually be realised by a mo\-dule. Now we show that one can say more about realisability of $\p$-local modules, where $\p$ is a non-maximal prime.

Throughout this section we denote by $H^*(G)$ the group cohomology ring of $G$. Remember that $H^*(G)$ is  graded-commutative (Section \ref{sec Group cohomology rings}).

Let $\p$ be a graded prime ideal of $H^*(G)$ and  denote by $\C_{\p}$ the kernel of the cohomological functor
 $$\Hom_{\bfK}(ik,-)^*_{\p}=\bfK(\Inj kG)(ik,-)^* \otimes_{H^*(G)} H^*(G)_{\p}\colon \Kinjg \rightarrow \Modgr H^*(G)_{\p}.$$
From Theorem \ref{Henning} we obtain a smashing localisation 
 $$\xymatrix{\ar @{--}\Kinjg \ar@<-1ex>[r]_-{Q_{\p}} & \Kinjg/\C_{\p} \ar@<-1ex>[l]_-{R_{\p}}.}$$
 
On the other hand, we have a smashing localisation 
$$\xymatrix{\ar @{--}\Kinjg \ar@<-1ex>[r]_-{\hat{Q}} & \uMod kG \ar@<-1ex>[l]-_{\hat{R}},}$$
where $\hat{Q}(C) = Z^0( C \otimes_k tk)$ and $\hat{R}(M)=tM$, see Chapter \ref{recoll}.

\begin{lem}\label{barQ}
If $\p$ is non-maximal, then $Q_{\p}$ factors over $\hat{Q}$, i.e.\ there is a functor
$$\bar{Q}\colon \uMod kG \to \Kinjg/\C_{\p}$$ such that $Q_{\p} = \bar{Q} \circ \hat{Q}$.
Moreover, $\bar{Q}$ has a right adjoint $\bar{R}$ satisfying $\bar{R} \bar{Q} \cong \id$ and
it holds $R_{\p} = \hat{R} \bar{R}$.
\end{lem}

\begin{proof}
By Proposition \ref{recollhopf}, we have a localisation sequence
$${\bf D}(\Mod kG) \xto{- \otimes_k pk} \Kinjg
\xto{\hat{Q}= Z^0( - \otimes_k tk)} \uMod kG.$$ 
Consequently, in order to obtain the functor $\bar Q$, it suffices to show that the composition $Q_{\p} \circ (- \otimes_k pk)$ is zero (see Lemma \ref{Verdier}(3)). 

Now observe that
$\Hom_{\bfK}(ik,kG)^* \cong \Hom_{kG}(H^0(ik),kG) \cong k$. Since
$\p$ is non-maximal, it holds  $k_{\p}= 0$ and hence $kG$ is contained in $\C_{\p}$.
We  infer that the composition $Q_{\p} \circ (- \otimes_k
pk)$ vanishes on $kG$, and by d\'evissage on all objects of $\bfD(\Mod kG)$.

The functor $\bar{Q}\colon \uMod kG \to
\Kinjg_{\p}$ we have obtained that way commutes with arbitrary direct sums since this holds for
$\hat{Q}$ and $Q_{\p}$, and because $\hat{Q}$ is dense. Furthermore, the
category $\uMod kG$ is compactly generated by the finite dimensional
modules (see \cite{Rickard}), thus $\bar{Q}$ has a right adjoint $\bar R$ by Proposition \ref{Brownprop}. Obviously it holds $R_{\p} = \hat{R} \bar{R}$
and we conclude that $\bar{R} \bar{Q} = \id$.
\end{proof}

Remember that a $H^*(G)$-module $X$ is $\p$-local if $X \cong X_{\p}$ as $H^*(G)$-modules.
\begin{prop}\label{p non-maximal}
Let $\p$ be a non-maximal prime and assume that $X \in \Modgr \Hgk$ is a $\p$-local module. 
If $X$ is realisable, then it can be realised by a Tate resolution $tM$ of some $kG$-module $M$. Furthermore, $X_{\p}$ is isomorphic to a direct summand of $H^*(G,M)_{\p}$.
\end{prop}

\begin{proof}
From Theorem \ref{Henning} and Remark \ref{alsoadjointscommute} we obtain
diagrams
$$\xymatrix@C-12pt{\ar @{--} \Kinjg \ar[d]_-{Q_{\p}} \ar[rrr]^-{\Hom(ik,-)^*}& && \Modgr H^*(G)
\ar[d]^{- \otimes H^*(G)_{\p}} && \Kinjg  \ar[rrr]^-{\Hom(ik,-)^*}& && \Modgr H^*(G)\\
\Kinjg/\C_{\p} \ar[rrr]^-{\Hom(Q_{\p}(ik),-)^*}&& &\Modgr H^*(G)_{\p} && \Kinjg/\C_{\p} \ar[u]_-{R_{\p}} \ar[rrr]^-{\Hom(Q_{\p}(ik),-)^*}& && \ar[u]^{\inc}
\Modgr H^*(G)_{\p}}$$ which commute up to isomorphism. Let $X$ be a graded $\p$-local module over $H^*(G)$. If $X$ is realisable by some complex $C \in \Kinjg$, then
$X_{\p}$ is  realisable by an object $Q_{\p}(C) \in \Kinjg / \C_{\p}$.
Hence $X \cong \inc(X_{\p})$ is  realisable by $R_{\p} Q_{\p}(C)$. We infer from Lemma~\ref{barQ} that $X$ can be realised by $\hat{R}(M)$, where $M$ denotes $\bar{R}Q_{\p} (C) \in \uMod kG$. But the functor $\hat{R}$ assigns to $M$ its Tate resolution $tM$.

For the second claim, we show that $\Hom_{\bfK}(ik,tM)^*_{\p} \cong H^*(G,M)_{\p}$.
Applying the cohomological functor $\Hom_{\bfK}(ik,-)_{\p}$ to the canonical  triangle $$pM \to iM \to tM \to pM[1]$$ yields an exact sequence
$$\Hom_{\bfK}(ik,pM[-1])_{\p}^* \to  \Hom_{\bfK}(ik,iM)_{\p}^* \to \Hom_{\bfK}(ik,tM)_{\p}^* \to
\Hom_{\bfK}(ik,pM)_{\p}^*$$ of graded $H^*(G)$-modules.
In the proof of Lemma \ref{barQ} we have shown that $pM$ and $pM[-1]$ lie in the kernel of $\Hom_{\bfK}(ik,-)_{\p}$. Thus $H^*(G,M)_{\p}$ and $\Hom_{\bfK}(ik,tM)^*_{\p}$ are, in fact, isomorphic.
\end{proof}
Observe that the proof also shows that for a \emph{strictly} realisable, graded $H^*(G,k)$-module $X$ which is moreover  $\p$-local, there exists a module $M \in \Mod kG$ such that $X \cong \Hom_{\bfK}(ik,tM)$ and $X_{\p}\cong H^*(G,M)_{\p}$.

\begin{Rem}
If the module $H^*(G,M)$ in Proposition \ref{p non-maximal} is $\p$-local, then $X$ is realisable by the $kG$-module $M$. However, in general we cannot expect $H^*(G,M)$ to be $\p$-local.
\end{Rem}

\subsection{Realisability is a local property}\label{realislocal}
In this section we show that the classical local-global principle  applies for realisability. Let $A$ be a differential graded algebra over a commutative ring $k$, and assume that the cohomology ring $H^{*}A$ is graded-commutative. Under some finiteness assumptions we prove that a graded $H^*A$-module is realisable if and only if $X_{\p}$ is realisable for all graded prime ideals $\p$ of $H^*A$.

\begin{lem}\label{X real, dann X_p real}
Let $A$ be a dg algebra with graded-commutative cohomology ring $H^*A$ and fix a graded prime ideal $\p$ of $H^*A$.
If a graded $H^*A$-module $X$ is (strictly) realisable, then $X_{\p}$ is (strictly) realisable. 
\end{lem} 
\begin{proof}
We use the commutative diagrams
$$\xymatrix@C+4.5pt{\ar @{--} \Ddg(A) \ar[d]_Q \ar[rr]^-{\Ddg(A)(A,-)^*}& & \Modgr H^*A
\ar[d]^{- \otimes_{H^*A} (H^*A)_{\p}} && \Ddg(A)  \ar[rr]^-{\Ddg(A)(A,-)^*}& & \Modgr H^*A\\
\Ddg(A)/\C_{\p} \ar[rr]^-{\Ddg(A)/\C_{\p}(QA,-)^*}& & \Modgr H^*A_{\p} && \Ddg(A)/\C_{\p} \ar[u]_R \ar[rr]^-{\Ddg(A)/\C_{\p}(QA,-)^*}& & \ar[u]^{\inc}
\Modgr H^*A_{\p}}$$
from Corollary \ref{HPSdia}. If $X$ is realisable, then the $H^*A_{\p}$-module $X_{\p}$ is realisable by an object of $\Ddg(A)/\C_{\p}$. But then the $H^*A$-module $X_{\p}$ is realisable by an object of $\Ddg(A)$, by the right hand diagram.
\end{proof}

Now our aim is to show that  $X_{\p}$ being realisable for all graded primes $\p$ of $H^*A$ implies that $X$ is realisable. On that purpose, we study the behaviour of the local obstruction $\kappa(X) \in \Ext^{3,-1}_{H^*A}(X,X)$ under $\p$-localisation. The following lemma is stated in a more generally setting, but applies to this situation.

\begin{lem}\label{diaclass}
Let $\T$, $\U$ be triangulated categories with arbitrary direct sums and let $N \in \T$ be a compact object. Assume that
$$F\colon \T \to \U$$ is an exact functor such that $FN$ is compact in $\U$,
and that
$$\widetilde{F}\colon \Modgr \T(N,N)^* \to \Modgr \U(FN,FN)^*$$
is an exact functor such that the diagram below
commutes up to natural isomorphism.
$$\xymatrix{\ar @{--} \T \ar[d]_-F \ar[rr]^-{\T(N,-)^*}& & \Modgr \T(N,N)^* \ar[d]^-{\widetilde{F}}\\
\U \ar[rr]^-{\U(FN,-)^*}& & \Modgr \U(FN,FN)^*}$$
Then we have for every graded $\T(N,N)^*$-module $X$
 $$\widetilde{F}(\kappa(X)) = \kappa(\widetilde{F}X) \text{ in }
\Ext^{3,-1}_{\U(FN,FN)^*}(\widetilde{F}X,\widetilde{F}X).$$ 
\end{lem}

\begin{proof}
Let $(\Sigma^{-1}B \xrightarrow{\delta} R_1 \xrightarrow{\alpha} R_0
\xrightarrow{\pi} B,\ \ep\colon\T(N,R_0)^*\to X)$ be an $N$-special
$\T$-{pre\-sen\-ta\-tion} with associated extension $\kappa(X) \in \Ext^{3,-1}_{\T(N,N)^*}(X,X)$.
Since $F$ is exact and preserves direct sums, we obtain a triangle
$\Sigma^{-1}FB \xrightarrow{F\delta} FR_1 \xrightarrow{F\alpha} FR_0
\xrightarrow{F\pi} FB$ in $\U$, where $FR_0$ and $FR_1$ are $FN$-free.
The exactness of $\tilde F$ and the commutativity of the diagram
yield an epimorphism $\zeta\colon \U(FN,FR_0)^* \to \widetilde{F}X$.
Consequently
\[ (\Sigma^{-1}FB \xrightarrow{F\delta} FR_1 \xrightarrow{F\alpha} FR_0
\xrightarrow{F\pi} FB,\  \zeta \colon \U(FN,FR_0)^* \to \widetilde{F}X)\]
is an $FN$-special $\U$-presentation and the commutativity of the diagram
{\small
$$\xymatrix@C-6pt{ 0 \ar[r] & (\widetilde{F}X)[-1] \ar[r] \ar[d]^{\cong} & \U(FN,\Sigma^{-1}FB)^* \ar[r]^-{(F\delta)_*}
\ar[d]^{\cong} & \U(FN,FR_1)^* \ar[r]^-{(F\alpha)_*}
\ar[d]^{\cong} & \U(FN,FR_0)^* \ar[r]^-{\zeta} \ar[d]^{\cong} & \widetilde{F}X \ar[r] \ar@{=}[d]  & 0\\
0 \ar[r] & \widetilde{F}(X[-1]) \ar[r] & \widetilde{F}(\T(N,\Sigma^{-1}B)^*) \ar[r]^-{\widetilde{F}(\delta_*)} & \widetilde{F}(\T(N,R_1)^*)
 \ar[r]^-{\widetilde{F}(\alpha_*)}&
\widetilde{F}(\T(N,R_0)^*) \ar[r]& \widetilde{F}X \ar[r] & 0}$$}
shows that $\kappa(\widetilde{F}X) = \widetilde{F}(\kappa(X))$ in
$\Ext^{3,-1}_{\U(FN,FN)^*}(\widetilde{F}X,\widetilde{F}X)$.
 \end{proof}

The next lemma is well-known for strictly commutative Noetherian rings (see for example \cite{BIV}), and it is easy to check that it also holds true for graded-commutative coherent rings. Remember that a graded ring $R$ is called \emph{coherent} if finitely
  generated graded submodules of finitely presented graded $R$-modules are assumed to be finitely presented.

\begin{lem}\label{extloc}
Let $R$ be a graded-commutative and coherent ring, and let $\p$ be a graded prime ideal of $R$.
Let $M, N$ be graded $R$-modules and assume that $M$ is finitely presented. Then there is a natural isomorphism 
\begin{eqnarray*}
\Ext^{i,*}_R(M,N)_{\p} \cong \Ext^{i,*}_{R_{\p}}(M_{\p},N_{\p})
\end{eqnarray*}
for all $i \ge 0$. \qed
\end{lem}

\begin{thm}[Local-global principle]\label{loc-glo-local}
Let $A$ be a dg algebra such that $H^*A$ is graded-commutative and coherent. The following conditions are equivalent for a finitely presented graded $H^*A$-module $X$:
\begin{itemize}
\item[(1)] $X$ is realisable.
\item[(2)] $X_{\p}$ is realisable for all graded prime ideals  $\p$ of $H^*A$.
\item[(3)] $X_{\m}$ is realisable for all graded maximal ideals $\m$ of $H^*A$.
\end{itemize}
\end{thm}

\begin{proof}
By Lemma \ref{X real, dann X_p real}, it suffices to show that $X$ is realisable if  $X_{\m}$ is realisable by an object in $\Ddg(A)/\C_{\m}$ for all graded maximal ideals $\m$ of $H^*A$.

The $H^*A$-module $X$ is realisable if and only if the
class $\kappa(X) \in \Ext^{3,-1}_{H^*A}(X,X)$ is
trivial (Theorem \ref{BKSlocal}). By Proposition \ref{local global principle}, the latter holds true if and
only if the fraction $\frac{\kappa(X)}{1}$ in $\Ext^{3,-1}_{H^*A}(X,X)_{\m}$
equals zero for all graded maximal ideals $\m$ of $H^*A$. But since $X$ is finitely presented and
$H^*A$ is assumed to be graded-commutative and coherent, we may apply the
natural isomorphism 
$$(\Ext^{3,*}_{H^*A}(X,X))_{\m} \cong  \Ext^{3,*}_{(H^*A)_{\m}}(X_{\m},X_{\m})$$
from Lemma \ref{extloc}, which maps the fraction $\frac{\kappa(X)}{1}$ to the extension
$\kappa(X) \otimes_{H^*A} (H^*A)_{\m}$. The object $QA \in
\Ddg(A)/\C_{\m}$ is compact (Lemma \ref{QAcompact}), hence we can
apply Lemma \ref{diaclass} to the commutative diagram 
$$\xymatrix@C+4pt{\ar @{--} \Ddg(A) \ar[d]_Q \ar[rr]^-{\Ddg(A)(A,-)^*}& & \Modgr H^*A
\ar[d]^{- \otimes_{H^*A} (H^*A)_{\m}} \\
\Ddg(A)/\C_{\m} \ar[rr]^-{\Ddg(A)/\C_{\m}(QA,-)^*}& & \Modgr H^*A_{\m}}$$
and obtain
$$\kappa(X)\otimes_{H^*A} (H^*A)_{\m} = \kappa(X \otimes_{H^*A} (H^*A)_{\m}).$$
Altogether we infer that $X$ is realisable if and only if the class $\kappa(X_{\m})$ is trivial for all $\m$, or equivalently, if $X_{\m}$ is realisable by an object
in $\Ddg(A)/\C_{\m}$ for all graded maximal ideals $\m$ of $H^*A$.
\end{proof}

A graded $H^*A$-module is realisable if and only if all its direct summands are realisable. 
Since $\p$-localisation commutes with direct sums, this shows
\begin{cor}\label{also sums of fp modules}
Let $A$ be a dg algebra such that $H^*A$ is graded-commutative and coherent.
Let $X \in \Modgr H^*A$ be an arbitrary direct sum of finitely presented graded $H^*A$-modules.  Then the following are equivalent:
\begin{itemize}
\item[(1)] $X$ is realisable.
\item[(2)] $X_{\p}$ is realisable for all graded prime ideals  $\p$ of $H^*A$.
\item[(3)] $X_{\m}$ is realisable for all graded maximal ideals $\m$ of $H^*A$.\qed
\end{itemize}
\end{cor}

Theorem \ref{loc-glo-local} and Corollary \ref{also sums of fp modules} apply in particular for realisability over the group cohomology ring $H^*(G,k)$, where $G$ is a finite group and $k$ is a noetherian ring (see Example \ref{example for real}). 
The finiteness of the group is not necessary as long as $H^*(G,k)$ is still coherent and $kG$ is noetherian. The latter is to ensure that the category  $\bfK(\Inj kG)$ is closed under taking  arbitrary direct sums. 

We also remark that in all results in this section, the realisability setting $$\D(A) \xto{H^*} \Modgr H^*A$$ can  be replaced by the more general setting $$\T \xto{\T(N,-)^*} \Modgr \T(N,N)^*,$$ where $\T$ is a triangulated category which admits arbitrary direct sums and is generated by compact objects, and $N \in \T$ is a compact object such that $\T(N,N)^*$ is graded-commutative and coherent.

\begin{Rem}
(1) One might want to have this local-global principle for \emph{arbitrary} graded $H^*A$-modules. 
It is well-known that every graded module is a direct limit of finitely presented graded modules, 
but it is open whether a realisable finitely presented module can be written as direct limit of \emph{realisable} finitely presented modules.
It is not known either whether an arbitrary direct limit of realisable modules is realisable. 

(2) It would be nice to have  a local-global principle also for \emph{strict} realisability (see Remark \ref{postnikov}). 
Let $X$ be a graded $H^*A$-module, where $A$ is a dg algebra with graded-commutative cohomology ring. If $X$ is strictly realisable, then so is $X_{\p}$ (Lemma \ref{X real, dann X_p real}). If $X$ is finitely presented and $H^*A$ coherent, 
 does $X_{\p}$ being strictly realisable for all primes $\p$ imply that $X$ is strictly realisable?

The infinite sequence of obstructions deciding on strict realisability
$$\kappa_n(X) \in \Ext^{n,2-n}_{H^*A}(X,X),\; n \ge 3,$$
where  $\kappa_n(X)$ is defined provided that the previous one $\kappa_{n-1}(X)$ vanishes, arises from an $A$-exact $\infty$-Postnikov-System which in cohomology gives rise to a map having cokernel~$X$ (see \cite[App.\ A]{BKS} for details). 
More precisely, an $A$-exact $l$-Postnikov system gives rise to the obstructions $\kappa_n(X), 3 \le n \le l$, and can be extended to an $A$-exact  $(l+1)$-Postnikov system provided that the class $\kappa_l(X)$ is trivial. If an $A$-exact $\infty$-Postnikov system exists, then $X$ is strictly realisable \cite[Prop.\ A.19]{BKS}.

Now one might want to prove iteratively $\kappa_l(X_{\p}) \cong \kappa_l(X)_{\p}$ for all $l \ge 3$ and use the same methods as in the proof of Theorem \ref{loc-glo-local}.
The problem is that all but the first obstruction are \emph{not uniquely determined} and consequently, the compatibility of the obstructions for the realisability of $X$ and the realisability of $X_{\p}$ cannot be expected in general.

\end{Rem}

\section{Localising the global obstruction}\label{sec Localising mu}
Let $A$ be a dg algebra over a field $k$ and assume that $H^*A$ is graded-commutative.
We have shown in Section \ref{realislocal} that
realisability is a local property. The local-global principle we have shown
applies for finitely presented modules but does not yield information on global realisability. In this chapter we develop a local-global principle for global realisability. 

Applying our results from Chapter \ref{seclift}, we can state a global obstruction for the $\p$-local $H^*A$-modules: Let $A_{\p}$ the localisation of $A$ at a prime $\p$ in cohomology, defined in Section~\ref{Apdef}. From Corollary \ref{HPSdia}, Proposition \ref{RHOM(RQA,R-)} and Theorem \ref{A'->A_p} we conclude that the diagram
$$\xymatrix{\ar @{--} \Ddg(A) \ar[d] \ar[rr]^-{\Ddg(A)(A,-)^*}& & \Modgr H^*A
\ar[d]^{- \otimes_{H^*A} (H^*A)_{\p}}\\
\Ddg(A_{\p}) \ar[rr]^-{\Ddg(A_{\p})(A_{\p},-)^*}& & \Modgr H^*(A_{\p}) }$$
commutes up to isomorphism. Since in particular $H^*(A_{\p}) \cong (H^*A)_{\p}$, we infer that the canonical class $\mu_{A_{\p}} \in \HH^{3,-1}(H^*A_{\p})$ is a global obstruction for the $\p$-local modules. 

At first sight, it is not clear whether the Hochschild classes $\mu_A \in \HH^{3,-1}(H^*A)$ and  $\mu_{A_{\p}} \in \HH^{3,-1}(H^*A_{\p})$ are associated in some way.
In order to relate them,  we show in Section \ref{flatepi} the  existence of a map of Hochschild cohomology rings 
$$\Gamma\colon \HH^{*,*}(H^*A) \longrightarrow \HH^{*,*}(H^*A_{\p})$$
which has the property $\Gamma(\mu_A)=\mu_{A_{\p}}$.

After discussing localisation of Hochschild cohomology groups of graded-commutative algebras in Section \ref{sec loc-glob global}, we are ready to prove a local-global principle for global reali\-sa\-bility.

\subsection{A map of Hochschild cohomology rings}\label{flatepi}
In general, a morphism of graded algebras $\varphi\colon R \to T$ does not induce a homomorphism $\HH^{*,*}(R) \to \HH^{*,*}(T)$. We show in this section that such a map does exist whenever
$T_{R}$ and $_{R} T$ are flat and  $\varphi\colon R \to T$ is an epimorphism in the category of rings, i.e.\ for any ring $T'$ and morphisms $\alpha,\beta\colon T \to T'$ with $\alpha \varphi = \beta \varphi$, it follows $\alpha = \beta$. Such a map can be characterised in the following way:

\begin{lem}\cite[Ch.\ XI, Prop.\ 1.2]{St}
The following conditions are equivalent for a morphism of graded rings $\varphi\colon R \to T$:
\begin{itemize}
\item[(1)] $\varphi$ is an epimorphism in the category of rings.
\item[(2)] The map $R \otimes_T R \to R, \; r \otimes r' \mapsto rr',$ is an isomorphism.
\item[(3)] The restriction functor $\Modgr R \to \Modgr T$ is full.
\end{itemize}
\end{lem}

We call a map of graded algebras $\varphi\colon R \to T$ a \emph{flat epimorphism} if $\varphi$ is an epimorphism in the category of rings and furthermore, the modules $T_{R}$ and $_{R} T$ are flat.

\begin{exm}\label{locmorflat}
If $R$ is a graded-commutative ring and $S \subseteq R$ a multiplicative subset of homogeneous elements, then $R \to S^{-1}R$ is a flat epimorphism: $S^{-1}R$ is flat as both left and right $R$-module, and it is an epimorphism of rings 
 by the universal property of the ring of fractions.
\end{exm}

By $\bbB(\Lambda)$ we denote the graded Bar resolution of a graded algebra $\Lambda$ as introduced in Chapter \ref{secBar}.
\begin{lem}\label{eta}
Let $\varphi\colon R \to T$ be a map of graded algebras which is a flat epimorphism of rings. Then the complex $T \otimes_R \bbB(R) \otimes_R T$ is a $T^e$-projective resolution of $T$ and a chain map
$\eta\colon T \otimes_R \bbB(R) \otimes_R T \to \bbB(T)$ is given by
\begin{eqnarray*}
\eta_n \colon    T \otimes_R \bbB(R)_n \otimes_R T & \longrightarrow & \bbB(T)_n\\
   (t,r_0,\cdots,r_{n+1},t') &  \longmapsto &
 (t\varphi(r_0),\varphi(r_1),\cdots,\varphi(r_n),\varphi(r_{n+1})t').\\
\end{eqnarray*}
\end{lem}
\begin{proof}
We first remark that $T \otimes_R \bbB(R) \otimes_R T$ is exact because $T \otimes_R -$ and $-\otimes_R T$ are exact. Since $_{R} T_T$ and $_{T} T_R$ are projective as $T$-modules and $\bbB(R)_n$ is a projective $R^e$-module for all $n \ge 0$, we get indeed a $T^e$-projective resolution of $T \otimes_R R \otimes_R T$. But since
$\varphi\colon R \to T$ is an epimorphism of rings, the map
$$T \otimes_R R \otimes_R T \to T,\quad (t,r,t') \mapsto t \varphi(r) t',$$ is an
isomorphism of $T^e$-modules.
\end{proof}
As a consequence we obtain
\begin{prop}\label{hochmap}
The maps
\begin{eqnarray*}
\Hom_{R^e}^l(\bbB(R)_{n},R) &  \longrightarrow &
\Hom_{T^e}^l(T \otimes_R \bbB(R)_{n} \otimes_R T,T),\\
\zeta & \longmapsto & T \otimes_R \zeta \otimes_R T
\end{eqnarray*}
where $l \in \bbZ$ and $n \ge 0$, induce a homomorphism of bigraded algebras
\begin{eqnarray*}
{\Gamma} \colon \HH^{*,*}(R) & \longrightarrow & \HH^{*,*}(T).
\end{eqnarray*}
\end{prop}

\begin{proof}
$\Gamma$ is obviously a morphism of graded vector spaces, and it is easy to check that it commutes with Yoneda multiplication.
\end{proof}

\begin{thm}\label{GammatutsAB}
Let $A, B$ dg algebras over a field $k$ and suppose that $\psi\colon A \to B$ is a morphism of dg algebras inducing a flat epimorphism
$\psi^* \colon H^*A \to H^*B$ in cohomology. Then the map
$$\Gamma \colon \HH^{*,*}(H^*A)  \longrightarrow  \HH^{*,*}(H^*B)$$  satisfies $\Gamma(\mu_A)=\mu_B$.
\end{thm}

\begin{proof}
We choose representing cocycles  $m_A \in \Hom_{(H^*A)^e}^{-1}(H^*A^{\otimes 5},H^*A)$ of $\mu_A$ 
and  \linebreak $m_B \in \Hom_{(H^*B)^e}^{-1}(H^*B^{\otimes 5},H^*B)$ of $\mu_B$ 
and show that 
the diagram
$$\xymatrix{\ar @{--} H^*B \otimes_{H^*A} H^*A^{\otimes 5} \otimes_{H^*A} H^*B \ar[d]_{\eta_3}
\ar[rrr]^-{H^*B \otimes m_A \otimes H^*B} & & & H^*B \otimes_{H^*A} H^*A \otimes_{H^*A} H^*B \ar[d]^{\eta_0}_{\cong}\\
H^*B^{\otimes 5} \ar[rrr]^-{m_B} & & & H^*B}$$ commutes up to
coboundaries. 

An element
$(s,x_0,x_1,x_2,x_3,x_4,t) \in H^*B \otimes H^*A^{\otimes 5} \otimes
H^*B$ is sent by $m_B \circ \eta_3$ to
\begin{equation}\label{links}
(-1)^{|s|} s m_B(\psi^*(x_0),\psi^*(x_1),\psi^*(x_2),\psi^*(x_3),\psi^*(x_4))t.
\end{equation}
On the other hand, under $\eta_0 \circ (H^*B \otimes m_A \otimes H^*B)$ our element maps to
\begin{equation}\label{rechts}
(-1)^{|s|} s \psi^*(m_A(x_0,x_1,x_2,x_3,x_4))t.
\end{equation}
By Lemma \ref{aboutm3}, we have a $(2,-1)$-Hochschild cocycle  
$u \in \Hom_{(H^*A)^e}^{-1}(H^*A^{\otimes 4}, H^*B)$ such that
$$m_B \circ (\psi^*)^{\otimes 5} - \psi^*  \circ m_A = u \circ d_3,$$
where $d_3\colon H^*A^{\otimes 5} \to H^*A^{\otimes 4}$ is the differential of the Bar resolution
of $H^*A$.
Hence the difference of (\ref{rechts}) and (\ref{links}) equals
$$(H^*B \otimes (u \circ d_3) \otimes H^*B)(s,x_0,x_1,x_2,x_3,x_4,t).$$

But now we are done since $H^*B \otimes (u \circ d_3) \otimes H^*B$ is a coboundary
in the complex $$\Hom_{(H^*B)^e}^{-1}(H^*B \otimes_{H^*A} \bbB(H^*A) \otimes_{H^*A} H^*B,H^*B)$$
which computes the Hochschild cohomology $\HH^{*,-1}(H^*B)$.
\end{proof}

\subsection{Local-global principle for the global obstruction}\label{sec loc-glob global}
Throughout this section let $A$ be a dg algebra over a field $k$ and suppose that $H^*A$ is graded-commutative. 
Fix a graded prime ideal $\p$ of $H^*A$. In Theorem \ref{A'->A_p} we have shown the existence of a dg algebra $A'$ quasi-isomorphic to $A$ and a zigzag of dg algebra maps 
$$A \xleftarrow{\sim} A' \xto{\varphi} A_{\p}$$
which induces the canonical map $\can\colon H^*A \to (H^*A)_{\p}$ in cohomology.  Now we use this result to obtain more information about the global obstruction for the $\p$-local modules.

\begin{prop}\label{GammatutsH*A_p}
The map
$$\Gamma \colon \HH^{*,*}(H^*A')  \longrightarrow \HH^{*,*}(H^*A_{\p})$$ 
satisfies $\Gamma(\mu_{A'})=\mu_{A_{\p}}$. Moreover, the composition
$$\xymatrix{\HH^{*,*}(H^*A) \ar[r]^-{\cong} & \HH^{*,*}(H^*A') \ar[r]_-{\Gamma} & \HH^{*,*}(H^*A_{\p})}$$
maps $\mu_A$ to $\mu_{A_{\p}}.$
\end{prop}
\begin{proof}
The dg algebra morphism  $\varphi\colon A' \to A_{\p}$ induces  in cohomology the flat epimorphism
of rings $H^*A \to (H^*A)_{\p}$. 
The first claim then
follows from Theorem \ref{GammatutsAB}. For the second just note that by Lemma \ref{aboutm3}, the
isomorphism $\HH^{*,*}(H^*A) \xto{\cong} \HH^{*,*}(H^*A')$ induced by the \qis\ $A' \to A$ maps $\mu_A$ to $\mu_{A'}$.
\end{proof}
Proposition \ref{GammatutsH*A_p} implies that  if $\mu_A$ is trivial, then so is $\mu_{A_{\p}}$  for all graded primes $\p$ of $H^*A$. Observe that this was not clear before: if all $H^*A$-modules are realisable, then so are in particular all ${\p}$-local modules. But this does not imply that $\mu_{A_{\p}}$ is trivial, see Remark~\ref{converse does not hold}.

Under an additional assumption for $H^*A$ we now show that if $\mu_{A_{\p}}$ is trivial for all prime ideals $\p$ of $H^*A$, then $\mu_A$ is trivial. 
\begin{lem}\label{bimodulextloc}
Let $R$ be a graded-commutative algebra over a commutative ring $k$.  Assume that $M$ 
is a graded $R^e$-module which admits a resolution of finitely generated projective graded  $R^e$-modules. Let $N$ be any graded $R^e$-module. Then for all $i \ge 0$,
there is a natural isomorphism
$$\Ext^{i,*}_{R^e}(M,N) \otimes_{R^e} (R_{\p})^e \cong
\Ext^{i,*}_{(R_{\p})^e}(M \otimes_{R^e} (R_{\p})^e,N \otimes_{R^e} (R_{\p})^e).$$
\end{lem}
\begin{proof}
Note first that $R^e$ is graded-commutative and consequently, $\Ext^{i,*}_{R^e}(M,N)$ is indeed a graded $R^e$-module.

The claim is immediately checked if $M$ is a finitely generated graded free $R^e$-module and $i=0$. 

Now assume that $M$ admits a resolution of finitely generated projective graded  $R^e$-modules. Then $M$ is in particular finitely presented and there exists an exact sequence
$$F_1 \xto{f} F_0 \to M \to 0,$$
where $F_0$ and $F_1$ are finitely generated graded free $R^e$-modules.
The commutative dia\-gram
{\small $$\xymatrix@C+40pt{\Hom_{R^e}^*(F_0,N) \otimes_{R^e} (R_{\p})^e  \ar[r]_{f^*  \otimes_{R^e} (R_{\p})^e} \ar[d]^{\cong} &
\Hom_{R^e}^*(F_1,N) \otimes_{R^e} (R_{\p})^e  \ar[d]^{\cong}\\
\Hom_{(R_{\p})^e}^*(F_0 \otimes_{R^e} (R_{\p})^e,N \otimes_{R^e} (R_{\p})^e ) \ar[r]_{(f\otimes_{R^e} (R_{\p})^e)^* } &
 \Hom_{(R_{\p})^e}^*(F_1 \otimes_{R^e} (R_{\p})^e,N \otimes_{R^e} (R_{\p})^e )   }$$}
and the exactness of $- \otimes_{R^e} (R_{\p})^e \cong R_{\p} \otimes_R - \otimes_R R_{\p}$ give the isomorphism
$$\Hom_{R^e}^*(M,N) \otimes_{R^e} (R_{\p})^e  \cong \Hom_{(R_{\p})^e}^*(M\otimes_{R^e} (R_{\p})^e,N\otimes_{R^e} (R_{\p})^e).$$

The claim for the Ext-groups follows from the condition that $M$ admits a resolution of finitely generated projective $R^e$-modules.
\end{proof}

\begin{rem}
The assumptions of Lemma \ref{bimodulextloc} are satisfied if $R^e$ is Noetherian and $M$ a finitely generated $R^e$-module. Note that it is not enough to assume that $R$ is Noetherian. In general, we cannot expect a tensor product of two Noetherian algebras to be Noetherian: If $F$ is a perfect field of characteristic $p >0$ and $k$ is an imperfect subfield of $F$, then the tensor product $F \otimes_{k} F$ is not Noetherian, see \cite[Sect.\ 1]{MM}.
\end{rem}

Before we prove the local-global principle, we establish a nice relation between the Hochschild groups of $R$ and $R_{\p}$.
\begin{prop}\label{Hochschildloc}
Let $R$ be a graded-commutative algebra over a field $k$. Suppose that the enveloping algebra $R^e$ is Noetherian, and let $\p$ be a graded prime ideal of $R$. For all $n \ge 0$, we have 
\begin{equation*}\label{HH(R_p)}
\HH^{n,*}(R_{\p}) \cong \HH^{n,*}(R)_{\p}
\end{equation*}
as graded $R$-modules.
In particular, a Hochschild group $\HH^{n,*}(R)$ is trivial if and only if
$\HH^{n,*}(R_{\p})$ is trivial for all graded prime ideals $\p$ of $R$.
\end{prop}
\begin{proof}
We first point out that
$$\HH^{n,*}(R_{\p}) \cong R_{\p} \otimes_R \HH^{n,*}(R) \otimes_R R_{\p}$$
by Lemma \ref{bimodulextloc}. In order to show that $R_{\p} \otimes_R \HH^{n,*}(R) \otimes_R R_{\p}$
is, in fact, just localisation of $\HH^{n,*}(R)$ at $\p$, we apply our results from Chapter \ref{hochring}. The Hochschild cohomology ring $\HH^{*,*}(R)$ is bigraded-commutative by Theorem \ref{bigraded-commutative} and  since $R$ is graded-commutative, Remark \ref{graded centre} implies that
$$\HH^{0,*}(R) = R.$$ 
Hence we have a well-defined $R$-linear map
\begin{eqnarray*}
\nu\colon R_{\p} \otimes_R \HH^{n,*}(R) \otimes_R R_{\p} & \longrightarrow &  \HH^{n,*}(R)_{\p}, \\
\frac{r}{s} \otimes \zeta \otimes \frac{r'}{s'} & \longmapsto & (-1)^{|r||\zeta|} \zeta \cdot \frac{rr'}{ss'}
\end{eqnarray*}
and it is easy to check that $\nu$ is bijective by stating the obvious inverse map.
\end{proof}
Note that the denominators $s, s'$ do not need sign adjustment because we have chosen them to have even degree (see Section \ref{grcom localisation}).

Now we are ready to prove 
\begin{thm}[Local-global principle]\label{loc-glo-global}
Let $A$ be a differential graded algebra over a field $k$ such that $H^*A$ is graded
commutative. Assume that the enveloping algebra $(H^*A)^e$ is Noetherian. Then the following conditions are equivalent:
\begin{itemize}
\item[(1)] $\mu_A \in \HH^{3,-1}(H^*A)$ is trivial.
\item[(2)] $\mu_{A_{\p}} \in \HH^{3,-1}(H^*A_{\p})$ is trivial for all graded prime ideals $\p$ of $H^*A$.
\item[(3)] $\mu_{A_{\m}} \in \HH^{3,-1}(H^*A_{\m})$ is trivial for all graded maximal ideals $\m$ of $H^*A$.
\end{itemize}
In particular, all graded $H^*A$-modules are realisable if the Hochschild class $\mu_{A_{\p}}$ is trivial for all graded prime ideals $\p$ of $H^*A$.
\end{thm}

\begin{proof}
Fix a graded prime $\p$ of $H^*A$. We prove that under the isomorphism
$$\HH^{3,*}(H^*A_{\p})\cong \HH^{3,*}(H^*A)_{\p}$$
of Proposition \ref{Hochschildloc}, the class $\mu_{A_{\p}}$ is mapped to the fraction $\frac{\mu_A}{1}$. This shows the claim.

Let  $A \xleftarrow{\sim} A' \xto{\varphi} A_{\p}$ be a zigzag of dg algebra maps which induces the canonical map $H^*A \to (H^*A)_{\p}$ in cohomology.
Corollary \ref{GammatutsH*A_p} implies that in the Hochschild group  $\HH^{3,-1}(H^*A_{\p})$, we have
$$\mu_{A_{\p}}= [{H^*A_{\p}}
\otimes_{H^*A'}\; m_{A'}\, \otimes_{H^*A'} {H^*A_{\p}}],$$
where $m_{A'}$ is a representing cocycle for $\mu_{A'} \in \HH^{3,-1}(H^*A')$.
We conclude from Lemma~\ref{bimodulextloc} and Proposition \ref{aboutm3}  that under the composition of isomorphisms 
\begin{eqnarray*}
\HH^{3,*}(H^*A_{\p}) & \cong & \HH^{3,*}(H^*A_{\p} \otimes_{H^*A'}  H^*A'\otimes_{H^*A'} H^*A_{\p})\\
& \cong & H^*A_{\p} \otimes_{H^*A'}  \HH^{3,*}(H^*A') \otimes_{H^*A'} H^*A_{\p}\\
& \cong & H^*A_{\p} \otimes_{H^*A}  \HH^{3,*}(H^*A) \otimes_{H^*A} H^*A_{\p},
\end{eqnarray*}
the Hochschild class $\mu_{A_{\p}}$ maps to $\frac{1}{1} \otimes  \mu_{A} \otimes \frac{1}{1}$.
But now we can apply the isomorphism
$$\nu\colon H^*A_{\p} \otimes_{H^*A} \HH^{3,*}(H^*A) \otimes_{H^*A} H^*A_{\p}  \longrightarrow   \HH^{3,*}(H^*A)_{\p}$$
from the proof of Proposition \ref{Hochschildloc}. Since 
$\nu(\frac{1}{1} \otimes  \mu_{A} \otimes \frac{1}{1})= \frac{\mu_A}{1},$ we have proved the claim.
\end{proof}

\begin{rem}
The reader might have noticed  that the dg algebra $A_{\p}$ was not shown to be uniquely determined up to \qis. The universal property we have proved in Section \ref{sec universal property} only holds on the level of derived  categories. Hence for another dg algebra $A'_{\p}$ satisfying $H^*(A'_{\p}) \cong (H^*A)_{\p}$, we obtain a canonical class $\mu_{A'_{\p}}$ which could behave differently.

However, what we have actually shown in Theorem \ref{loc-glo-global} is that for \emph{every} dg algebra $A'_{\p}$ admitting a zigzag of dg algebra morphisms $A \xleftarrow{\sim} A'' \xto{\varphi} A'_{\p}$ which induces $\can\colon H^*A \to (H^*A)_{\p}$ in cohomology, the canonical class $\mu_{A'_{\p}}$ is the image of $\mu_A$ under the map
$$\HH^{3,-1}(H^*A) \longrightarrow \HH^{3,-1}(H^*A)_{\p},\quad \zeta \longmapsto \frac{\zeta}{1}.$$ So the choice of the dg algebra inducing $(H^*A)_{\p}$ in cohomology is not relevant, as long as it admits such a zigzag.
\end{rem}

\begin{Rem}
One might want to have Proposition \ref{Hochschildloc} and with it  Theorem \ref{loc-glo-global} under weaker assumptions.
In Proposition \ref{Hochschildloc} we have assumed that $R^e$ is Noetherian to ensure that $R$ admits a resolution of finitely generated $R^e$-projective modules. 
We do not know whether there is a way to avoid the latter condition.
\end{Rem}

\section{Comparing realisability over group and Tate cohomology}\label{sec Comparing}
Now we focus on realisability in group representation theory and compare realisability over the group cohomology ring and the Tate cohomology ring.

The group cohomology ring $H^*(G,k)$ has better properties than the Tate cohomology ring $\hat H^*(G,k)$ which, for instance, is not Noetherian in general. However, when it comes to the source categories of realisability, the stable module category $\uMod kG$ is more ``handsome'' than the homotopy category $\Kinjg$. This is the reason why we are interested in studying the relation of realisability over  group and Tate cohomology. 

The triangulated categories $\Kinjg$ and $\uMod kG$ are related by a smashing localisation
$$\xymatrix{\ar @{--}\bfK(\Inj kG) \ar@<-1ex>[r]_-Q & \uMod kG \ar@<-1ex>[l]_-R}$$
(see Proposition \ref{recollhopf}) and we are now concerned with finding a relation between realisability and this localisation of triangulated categories. 

Remember that both $H^*(G,k)$ and $\hat H^*(G,k)$ are the cohomology of a dg algebra (Example \ref{example for real}) and thus, they admit an $A_{\infty}$-algebra structure yielding global obstructions which we now denote by 
$\mu_G \in \HH^{3,-1}(H^*(G,k))$ and $\hat \mu_G \in \HH^{3,-1}(\hat H^*(G,k))$.
\smallskip

In the first section we study realisability of fixed modules. Then we focus on global realisability. The canonical class $\hat \mu_G$ has been computed for some groups $G$ by Benson, Krause and Schwede \cite{BKS}, and by Langer \cite{L}. 
We consider the same groups and compute the global obstructions for the group cohomology rings. We will see in Section \ref{secexamples} that in all but one case, the Hochschild classes $\mu_G$ and $\hat \mu_G $ turn out to behave surprisingly similar. 
As a first explication for this similarity we  show in Section \ref{secmorext} that the algebra morphism $H^*(G,k) \to \hat H^*(G,k)$ is induced by a zigzag of dg algebra morphisms. 
The main result of this chapter is stated in the last section and gives a complete explanation for the relation between $\mu_G$ and $\hat \mu_G$ we observed before in examples.

We like to thank Dave Benson for discussions that helped to improve this chapter and in particular, for a contribution to the proof of Theorem \ref{murelation}.

 \subsection{Local realisability}\label{sec Local realisability}
Let $k$ be a field of characteristic $p >0$ and $G$ be a finite group such that $p$ divides the order of $G$. In this section we discuss ways to construct a realisable $H^*(G,k)$-module from a realisable module over $H^*(G,k)$, and vice versa.

The following proposition will play an important role in the next section.
\begin{prop}\label{rank1real}
Assume that the $p$-rank of $G$ equals one. 
\begin{itemize}
\item[(1)] If $X \in \Modgr H^*(G,k)$ is realisable, then $X \otimes_{H^*(G,k)} \hat{H}^*(G,k)$ is a realisable $\hat{H}^*(G,k)$-module.
\item[(2)] A graded $\hat H^*(G,k)$-module $Y$ is realisable if and only if its restriction to $H^*(G,k)$ is a realisable $H^*(G,k)$-module.
\end{itemize}
\end{prop}
\begin{proof}
By Lemma \ref{periodic-loc} and Theorem \ref{periodic-rank}, there exists a multiplicative subset $S \subseteq H^*(G,k)$ such that $\hat H^*(G,k) = S^{-1} H^*(G,k)$. Hence we may apply Theorem \ref{Henning} and Remark  \ref{alsoadjointscommute} and obtain commutative diagrams

$$\xymatrix@C-2.7pt{\ar @{--} \bfK(\Inj kG) \ar[d]_Q \ar[rr]^-{\Hom_{\bfK}(ik,-)^*}& & \Modgr H^*(G,k)
\ar[d]^{- \otimes_{H^*(G,k)} \hat H^*(G,k)} &&  \bfK(\Inj kG) \ar[rr]^-{\Hom_{\bfK}(ik,-)^*}& & \Modgr H^*(G,k)\\
\uMod kG \ar[rr]^-{\uHom_{kG}(k,-)^*}& & \Modgr \hat H^*(G,k) && \uMod kG \ar[u]_R \ar[rr]^-{\uHom_{kG}(k,-)^*}& & \ar[u]^{\res}
\Modgr \hat H^*(G,k)}$$
which yield the claim.
\end{proof}

If the $p$-rank of $G$ is at least two, then we cannot expect to obtain realisable modules from realisable modules by induction or restriction as above. The reason for this is that in general, the induction functor $- \otimes_{H^*(G,k)} \hat H^*(G,k)$ is not given by localisation with respect to a multiplicatively closed subset.

However, there is another possibility to obtain a realisable $H^*(G,k)$-module from a realisable  $\hat H^*(G,k)$-module. This construction also works in the general case.
\begin{lem}\label{truncation}
Let $X \in \Modgr \hat{H}^*(G,k)$ realisable. Then its truncation to non-negative degrees $X^{\geqslant 0}$ is a realisable $H^*(G,k)$-module.
\end{lem}
\begin{proof}
Let $X$ be a direct summand of $\hat{H}^*(G,M)$, where $M$ is a $kG$-module. Without loss of generality, we may assume that $M$ does not have any projective direct summands. Then $\hat{H}^0(G,M) \cong H^0(G,M)$ and consequently, $\hat{H}^*(G,M)^{\geqslant 0} \cong H^*(G,M)$. It follows that $X^{\geqslant 0}$ is a direct summand of $H^*(G,M)$.
\end{proof}

\subsection{Examples for the global obstruction}\label{secexamples}
Let $k$ be a field of characteristic $p>0$. We study the group and Tate cohomology rings of cyclic $p$-groups and the Quaternion group, and focus on the global obstructions $\mu_G$ and $\hat \mu_G$. 

\begin{thm}\cite[Thm.\ 7.1]{BKS}\label{tate3}
Let $G$ be a cyclic group of order $p^n$, where $n \ge 1$. If $p^n=2$, then
the Tate cohomology ring is a Laurent polynomial ring $k[X,X^{-1}]$
on a 1-dimensional class $X$. If $p^n \ge 3$, then the Tate
cohomology ring is a truncated Laurent polynomial ring in two
variables,
$$\hat{H}^*(G,k) = k[X,Y,Y^{-1}]/(X^2),$$
with $\deg(X)=1$ and $\deg(Y)=2$. The secondary multiplication $m_3$
of the $A_{\infty}$-algebra $\hat{H}^*(G,k)$ and thus $\hat \mu_G \in \HH^{3,-1}(\hat H^*(G,k))$
is trivial except when $p =3$ and $n=1$. In this case, the secondary multiplication is given by
$$m_3(XY^i,XY^j,XY^l) = Y^{i+j+l+1}, \quad i,j,l \in \bbZ,$$
and vanishes on all other tensor products of monomials. Furthermore,
its Hochschild class $\hat{\mu}_{\bbZ/(3)}$ is non-trivial.
\end{thm}

Note that strictly speaking, one can make choices to obtain $m_3$ in the shape as stated above. With other choices of $f_1$ and $f_2$ in Construction \ref{m3construction} we could obtain a different map. However, the Hochschild class $\hat \mu_G$ of $m_3$ is independent of all choices.


Now we consider the group cohomology ring $H^*(G,k)$ of a cyclic group $G$ as above. It identifies with the subring of non-negative degrees of $\hat H^*(G,k)$. 
The global obstruction $\mu_G \in \HH^{3,-1}(H^*(G,k))$  turns out to behave very similarly:
\begin{prop}\label{ext3}
Let $G$ be a cyclic group of order $p^n$, where $n \ge 1$. If $p^n=2$, then
the group cohomology ring is a polynomial ring $k[X]$, where $X$ has degree 1.
If $p^n \ge 3$, then $H^*(G,k)$  is a truncated
polynomial ring in two variables,
$$H^*(G,k) = k[X,Y]/(X^2),$$
with $\deg(X)=1$ and $\deg(Y)=2$. The secondary multiplication $m_3$ 
of the $A_{\infty}$-algebra $H^*(G,k)$ and thus $\mu_G \in \HH^{3,-1}(H^*(G,k))$ is trivial except when $p =3$ and $n=1$. In this case, it satisfies
$$m_3(XY^i,XY^j,XY^l) = Y^{i+j+l+1}, \quad i,j,l \ge 0,$$
and vanishes on all other tensor products of monomials. Its Hochschild class $\mu_{\bbZ_3}$ is non-trivial.
\end{prop}
\begin{proof}
Since the characteristic of $k$ is positive, we may identify $kG$ with the truncated polynomial ring $K[T]/(T^r)$, where $r=p^n$. An injective resolution of $k$ which is moreover $2$-periodic is given by

$$\xymatrix@C-12pt{ ik\colon  & \cdots \ar[rr] & & 0 \ar[rr]
 & & 0 \ar[rr]   & & I_0 \ar[rr]^{T}  && I_1 \ar[rr]^{-T^{r-1}}  & & I_2 \ar[rr]^-{T}  & & I_3 \ar[rr]^-{-T^{r-1}} && \cdots}$$

Let $x\colon ik \to \Sigma ik$ be the degree-one chain map
$$\xymatrix@C-12pt{ ik \ar[d]^{x} & \cdots \ar[rr] & & 0 \ar[rr]
\ar[d] & & 0 \ar[rr] \ar[d]  & & I_0 \ar[rr]^{T} \ar[d]^1 && I_1 \ar[rr]^-{-T^{r-1}} \ar[d]^-{T^{r-2}} & & I_2 \ar[rr]^-{T} \ar[d]^1 & & I_3 \ar[d]^-{T^{r-2}} \ar[rr]^-{-T^{r-1}}&&\cdots\\
\Sigma ik  & \cdots \ar[rr] && 0 \ar[rr] && I_0
 \ar[rr]_-{-T} &&
I_1 \ar[rr]_-{T^{r-1}} && I_2 \ar[rr]_-{-T} & &I_3 \ar[rr]_-{T^{r-1}}& & I_4 \ar[rr]_-{-T}  &&\cdots}$$
and $y\colon ik \to \Sigma^2 ik$ the degree-two chain map
$$\xymatrix@C-12pt{ ik \ar[d]^{y} & \cdots \ar[rr]  && 0 \ar[rr]
\ar[d] & &0 \ar[rr] \ar[d]  && I_0 \ar[rr]^-{T} \ar[d]^1 && I_1 \ar[rr]^-{-T^{r-1}} \ar[d]^1 && I_2 \ar[rr]^-{T} \ar[d]^1 && \cdots\\
\Sigma^2 ik  & \cdots \ar[rr] && I_0
 \ar[rr]_-T &&
I_1 \ar[rr]_-{-T^{r-1}} && I_2 \ar[rr]_-T && I_3 \ar[rr]_-{-T^{r-1}}& & I_4 \ar[rr]_-T && \cdots}$$

One easily checks that $xy=yx$. If $r\geq3$, then $x^2$ is nullhomotopic by the homotopy $q$, given by multiplication with $T^{r-3}$ in odd degrees, and the zero map in even degrees. If $r=2$, then obviously $x^2=y$. We infer that $x$ and $y$ are cycles of $\END(ik)$ representing the classes $X \in H^1(G,k)$ and $Y \in H^2(G,k)$, respectively.

In order to compute the secondary multiplication, we define a cycle selection homomorphism 
$$f_1\colon H^*\END(ik) \to \END(ik)$$
 on the $k$-basis $\{X^{\epsilon}Y^i\,|\,\epsilon \in \{0,1\}, i \ge 0\}$ of $H^*(G,k)$, given by
$$f_1(X^{\epsilon}Y^i)=x^{\epsilon} y^i.$$
If $r=2$, then this map is multiplicative and it follows that $f_2$ and with it $m_3$ can be chosen to be trivial.
If $r\ge 3$, then $f_1$ is multiplicative except on two odd dimensional classes. The product
$f_1(XY^i)f_1(XY^j)=x^2 y^{i+j}$ is only nullhomotopic, by the homotopy $q y^{i+j}$.

We define $$f_2\colon H^*\END(ik) \otimes H^*\END(ik) \to \END(ik)$$ to be trivial except on two odd dimensional classes: In this case, we set
$$f_2(XY^i,XY^j)=qy^{i+j}.$$

Since $m_3$ maps $(A,B,C)$ to the cohomology class of the expression
\begin{equation}\label{ABC}
(-1)^{|A|}f_1(A)f_2(B,C)-f_2(A,B)f_1(C)-f_2(AB,C)+f_2(A,BC) 
\end{equation}
(see \eqref{m_3}), we infer that $m_3$ vanishes on all tensor product of monomials with at least one monomial having even degree. Using the fact that the homotopy $q$ commutes with $y$, one checks that the tensor product of three odd degree monomials $(XY^i,XY^j,XY^l)$ is mapped under \eqref{ABC} to
$$(qx+xq)y^{i+j+l}.$$

The chain map $qx+xq\colon ik \to \Sigma^2 ik$ is given by multiplication with $T^{r-3}$ in each degree. Now if $r>3$, then $qx+xq$ is nullhomotopic via the homotopy given by the zero map in even degrees, and by multiplication with $T^{r-4}$ in odd degrees. Thus $m_3=0$. But if $r=3$, then $qx+xq=y$ and we conclude 
$$m_3(XY^i,XY^j,XY^l)=Y^{i+j+l+1}.$$

It remains to show that also the Hochschild class $\mu_{\bbZ_3}$ is non-trivial. But this is is a consequence of the following remark.
\end{proof}

\begin{rem}
Using Massey Products, Benson, Krause and Schwede  proved that $\hat H^*(\bbZ_3,k)/(X)$ is a non-realisable $\hat H^*(\bbZ_3,k)$-module \cite[Exm.\  7.6]{BKS}. From Proposition~\ref{rank1real} we conclude that $\hat H^*(\bbZ_3,k)/(X)$, viewed as $H^*(\bbZ_3,k)$-module, is not realisable either. 
\end{rem}

\begin{thm}\cite[Satz 2.10]{L}\label{Tatequaternion}
Denote by $Q_8$ be the quaternion group and let $k$ be a field of characteristic two. Then the Tate cohomology ring $\hat{H}^*(Q_8,k)$ is given by
$$k[X,Y,S^{\pm 1}]/(X^2+Y^2 = XY, X^3 = Y^3 = 0, X^2Y = XY^2),$$
with $\vert X \vert = \vert Y \vert = 1, \vert S \vert =4$. 
The canonical class $\hat{\mu}_{Q_8}$ is non-trivial.
\end{thm}
Furthermore, Langer shows  the existence of a non-realisable module \cite[Lemma 2.23]{L}:
 Write $\hat H$ for $\hat H ^*(Q_8,k)$. The cokernel of the map
$$g\colon \hat H[-1] \oplus \hat H[-1] \xto{\smatrix{Y & X+Y\\ X & Y}} \hat H \oplus \hat H$$
is not a realisable module. 

Since $r_2(Q_8)=1$ (see Theorem \ref{when rank is one}), it follows from Proposition \ref{rank1real} that $\Coker g$, viewed as module over $H^*(Q_8,k)$, is not realisable. We obtain
\begin{cor}\label{groupquaternion}
Denote by $Q_8$ be the quaternion group and let $k$ be a field of characteristic two. Then $$H^*(Q_8,k) = k[X,Y,S]/(X^2+Y^2 = XY, X^3 = Y^3 = 0, X^2Y = XY^2),$$ with $\vert X \vert = \vert Y \vert = 1, \vert S \vert =4$.  The canonical class $\mu_{Q_8}$ is non-trivial. \qed\end{cor}


One might wonder whether it always holds true that $\mu_G$ is trivial if and only if $\hat{\mu}_G$ is trivial. But in general, this is not the case. For any finite abelian  $2$-group $G$, the class $\mu_G$ is trivial by Proposition \ref{ext3} and Remark \ref{m3Kunneth}. However, the situation for Tate cohomology is entirely different when it comes to the Klein four group:
\begin{thm}\cite[Exm.\ 7.7]{BKS} \cite[Thm.\ 3.1]{L}\label{Langer}
Let $G = \bbZ_2 \times \bbZ_2$ and $k$ be a field with $\Char (k) = 2$. Then $\hat{\mu}_{\bbZ_2}$ is non-trivial. For any other finite abelian $2$-group, the canonical class is trivial.
\end{thm}

\subsubsection{Reduction to Sylow subgroups}\label{secsylow}
Let $k$ be a field of characteristic $p > 0$ and $G$ be a finite group such that $p$ divides the order of $G$. Let $P$ a $p$-Sylow subgroup of $G$. Benson, Krause and Schwede \cite{BKS} have shown that the canonical class $\hat \mu_G$ is already determined by $\hat \mu_P$. In order to compare $\mu_G$ and $\hat \mu_G$ in more cases, we now briefly   explain why the same holds true for $\mu_G$ and $\mu_P$.


Let $P$ be a subgroup of $G$. If $M, N$ are $kG$-modules, we define the \emph{transfer} or \emph{trace map} 
$$\Tr \colon \Hom_{kP}(M,N) \to \Hom_{kG}(M,N)$$
as follows: 
$$\Tr_{P,G}(\Phi)(m) = \sum_{i \in I} g_i \Phi(g_i^{-1}m),$$
where $\{g_i \;|\; i \in I\}$ is any choice of left coset representatives of $P$ in $G$. See \cite[Ch.\ 3.6]{Bbook1} for details.
The transfer $\Tr$ induces a well-defined map
$$\Tr\colon \hat{H}^n(P,k) \to \hat{H}^n(G,k)$$
for $n \in \bbZ$, independent on the choice of the resolution \cite[Lemma 3.6.16]{Bbook1}. 

By considering a $kG$-Tate resolution of the trivial module as $kP$-Tate resolution of $k$ regarded as $kP$-module, we obtain for $n \in \bbZ$ a restriction map
$$\res \colon  \hat{H}^n(G,k) \to \hat{H}^n(P,k).$$

\begin{lem}\cite[Ch.\ 3.6]{Bbook1}\label{res}
The composition $\Tr  \circ \res\colon \hat{H}^n(G,k) \to \hat{H}^n(G,k)$ is given by multiplication with $[G:P]$.
\end{lem}


Note that the restriction map induces morphisms of graded algebras $\res\colon H^*(G,k) \to H^*(P,k)$ and
$\res\colon \hat{H}^*(G,k) \to \hat{H}^*(P,k)$. Moreover, if $P$ is a $p$-Sylow subgroup of $G$, then these algebra morphisms are split monomorphisms by Lemma \ref{res}.

\begin{thm} \label{TateSylowreduction}\cite[Thm.\ 8.3]{BKS} 
Let $k$ be a field of characteristic $p>0$, $G$ a finite
group and $P$ a $p$-Sylow subgroup of $G$. Then the canonical class
$\hat{\mu}_G\in\HH^{3,-1}(\hat H^*(G,k))$ is determined by the canonical class
$\hat{\mu}_P\in\HH^{3,-1}(\hat H^*(P,k))$ of the Sylow subgroup, and is given by the formula
\[ \hat{\mu}_G \ = 
\ \frac{\Tr \circ \hat{\mu}_P \circ \res^{\otimes 3}}{[G:P]} \ . \] 
\end{thm}

The key point of the proof is that the restriction map in Tate cohomology is induced by a morphism of dg algebras. 
Since similarly, $\res\colon H^*(G,k) \to H^*(P,k)$ is induced by the inclusion dg algebras
$$\END_{kG}(ik) \to \END_{kP}(ik)$$
given by viewing the $kG$-injective resolution $ik$ as injective resolution over $kP$, we obtain 

\begin{lem}\label{GroupSylowreduction}
Let $k$ be a field of characteristic $p>0$, $G$ a finite
group and $P$ a $p$-Sylow subgroup of $G$. Then for the canonical classes
${\mu}_G\in\HH^{3,-1}( H^*(G,k))$ and \linebreak ${\mu}_P\in\HH^{3,-1}( H^*(P,k))$, it holds
\begin{align*}
\quad &\quad &\quad &\quad &\quad &\quad & &\hfill&\hfill&\hfill&\hfill&&&\hfill&\hfill&\hfill \hfill \hfill \hfill\hfill\hfill\ \hfill\hfill\hfill&&&&
{\mu}_G  = \frac{\Tr \circ {\mu}_P \circ \res^{\otimes 3}}{[G:P]}\; .
&\hfill&\hfill &\hfill&\hfill&\hfill&&&&&&\quad&\hfill&\hfill&\hfill&\hfill \hfill \hfill \hfill\hfill\hfill\quad &\quad &\quad &\qed
\end{align*}
\end{lem}

Benson, Krause and Schwede concluded in particular that if $G$ is a group whose $p$-Sylow group is cyclic with order different from $3$, then $\hat \mu_G$ is trivial \cite[Cor.\ 8.4]{BKS}. We infer from Lemma \ref{GroupSylowreduction} and Proposition \ref{ext3} that the same holds true for $\mu_G$.
\begin{cor}
Let $k$ be a field of characteristic $p>0$ and $G$ be a group whose $p$-Sylow subgroup is cyclic of order $p^n, n \geq 1$. Suppose that $n \geq 2$ if $p=3$. Then both the Hochschild classes $\mu_G$ and $\hat{\mu}_G$ are trivial.
\end{cor}

In order to investigate the case $P = \bbZ/(3)$, one needs to distinguish whether $G$ is $3$-nilpotent or not:
\begin{defn}
Let $G$ be a finite group and let $P$ be a $p$-Sylow subgroup. $G$ is called \emph{$p$-nilpotent} if there exists a normal subgroup $U \trianglelefteq G$ such that the composite
$$P \xto{\mathrm{incl}} G \xto{\mathrm{proj}} G/U$$
is an isomorphism.
\end{defn}

\begin{prop}\cite[Prop.\ 8.5]{BKS}
Let $k$ be a field of characteristic $p>0$ and $G$ be a finite group with $p$-Sylow subgroup $P$. The following conditions are equivalent:
\begin{itemize}
\item[(1)] $\res\colon \hat{H}^*(G,k) \to \hat{H}^*(P,k)$ is an isomorphism.
\item[(2)] $\res\colon H^*(G,k) \to H^*(P,k)$  is an isomorphism.
\item[(3)] $\res\colon H^1(G,k) \to H^1(P,k)$ is an isomorphism.
\item[(4)] The group $G$ is $p$-nilpotent.
\end{itemize}
\end{prop}

Since both the restriction maps $\hat{H}^*(G,k) \to \hat{H}^*(P,k)$ and $H^*(G,k) \to H^*(P,k)$ are induced by a morphism of dg algebras, we obtain as a consequence of Proposition~\ref{aboutm3}
\begin{cor}
Let $k$ be a field of characteristic $3$ and $G$ a finite, $3$-nilpotent group whose $3$-Sylow group is cyclic of order $3$. Then both the Hochschild classes $\mu_G$ and $\hat{\mu}_G$ are non-trivial.
\end{cor}
For $\hat \mu_G$, the result above is stated in \cite[Sect.\ 8]{BKS}.

\begin{thm}\cite[Thm.\ 8.6]{BKS}
Let $k$ be a field of characteristic $3$ and $G$ a finite group whose $3$-Sylow group is cyclic of order $3$. Assume that $G$ is not $3$-nilpotent. Then 
$$\hat{H}^*(G,k) = k[V,W,W^{-1}]/(V^2),$$
where $V$ is of degree $3$ and $W$ of degree $4$. The canonical class 
$\hat \mu_G$ is represented by the $(3,-1)$-cocycle $m$ given by
$$m(VW^i,VW^j,VW^l)=W^{i+j+l+2}, \quad i,j,l \in \bbZ,$$
and vanishes on all other tensor products of monomials in $V$ and $W$. The Hochschild class $\hat{\mu}_G$ is non-trivial.
\end{thm}

In \cite{BKS}, it is further shown that $\hat{H}^*(G,k)/(V)$ is not realisable. We conclude from Proposition \ref{rank1real} that $V$ viewed as module over $H^*(G,k)$ is not realisable  either. In particular, the canonical class $\mu_G$ must be non-trivial. 
Using Lemma \ref{GroupSylowreduction}, we can compute a representing cocycle for the canonical class, using the same methods as in \cite[Thm.\ 8.6]{BKS}.
\begin{prop}\label{group3Sylow}
Let $k$ be a field of characteristic $3$ and $G$ a finite group whose $3$-Sylow group is cyclic of order $3$. Assume that $G$ is not $3$-nilpotent. Then 
$$H^*(G,k) = k[V,W]/(V^2)$$
where $V$ is of degree $3$ and $W$ of degree $4$.
A representing $(3,-1)$-cocycle $m$  for the canonical class $\mu_G$ is given by
$$m(VW^i,VW^j,VW^l)=W^{i+j+l+2}, \quad i,j,l \geq 0,$$
and vanishes on all other tensor products of monomials in $V$ and $W$. Its Hochschild class $\hat{\mu}_G$ is non-trivial.\qed
\end{prop}

\subsection{Lifting $H^*(G,k) \to \hat{H}^*(G,k)$ to a morphism of dg algebras}\label{secmorext}
Let $G$ be a finite group and $k$ a field.
In the examples we have considered in the last section, the class $\mu_G$ always arises as restriction of $\hat \mu_G$ to non-negative degrees. Now we give a explanation for this fact: We show that the canonical inclusion
$$\iota\colon H^*(G,k) \to \hat{H}^*(G,k)$$
is induced by a zigzag of dg algebra morphisms. The construction is analogous to the one in Section \ref{seclift}. However, for the convenience of the reader we sketch it briefly.
Let $\eta\colon \id \to RQ$ be the unit and $\varepsilon\colon QR
\to \id$ the counit of the adjunction
$$\xymatrix{\ar @{--}\bfK(\Inj kG) \ar@<-1ex>[r]_Q & \bfK_{ac}(\Inj kG) \ar@<-1ex>[l]_R,}$$
where $Q = - \otimes tk$ and $R$ is the inclusion. This is a smashing localisation by Proposition~\ref{recollhopf}.

The map $$\eta_{ik} \colon  ik \to RQ(ik)$$ is up to natural isomorphism just the canonical
inclusion $ik \to tk$, see Proposition \ref{recollhopf}. Therefore the map
$$\bfK(\Inj kG)(ik,ik)^* \to \bfK(\Inj kG)(RQ(ik),RQ(ik))^*,\quad f \mapsto RQ(f)$$
is up to isomorphism the canonical inclusion 
$$\iota \colon H^*(G,k) \to \hat{H}^*(G,k).$$

We have $\bfK(\Inj kG)(ik,RQ(ik)) \cong H^0(\HOM(ik,RQ(ik))$. For any representing cocycle $\bar{\eta}_{ik} \in
Z^0(\HOM(ik,RQ(ik))$ of $\eta_{ik}$ we obtain as in Lemma \ref{eta*qis} a \qis
$$\bar{\eta}_{ik}^{\,*} \colon \END(RQ(ik)) \to \HOM(ik,RQ(ik)), \quad f \mapsto f \circ \bar{\eta}_{ik}.$$ 
We choose for $\bar{\eta}_{ik}$ the inclusion of complexes $ik \to tk$ which is a degree-wise split monomorphism of complexes. Thus $\bar{\eta}_{ik}^{\,*}$ is surjective.

\begin{thm}\label{exttatemorphism}
There exists a dg algebra $\END(ik)'$ quasi-isomorphic to $\END(ik)$ and a 
zigzag of dg algebra morphisms
$$\xymatrix{\END(ik) & \ar[l]_-{\rho}^-{\sim} \END(ik)' \ar[r]^-{\varphi} & \END(tk)}$$
inducing the canonical inclusion $\iota \colon H^*(G,k) \to \hat{H}^*(G,k)$ in cohomology. That is, we have diagrams
$$\xymatrix{\ar @{--} \END(ik)' \ar[d]^{\sim}_{\rho} \ar[rd]^-{\varphi} && H^*(G,k) \ar[d]^{\cong}_{H^*\rho} \ar[rd]^-{H^*\varphi}& \\
\END(ik) & \END(tk)   & H^*(G,k) \ar[r]^-{\iota}&  \hat H^*(G,k)}$$
where right hand diagram is commutative and identifies with the cohomology of the left hand diagram.
\end{thm}

\begin{proof}
We form the pullback diagram
$$\xymatrix{
\END(ik)' \ar[r]^-{p_2} \ar[d]_{p_1} &  \END(RQ(ik)) \ar[d]|>{\object@{>>}}_{\sim}^{\bar{\eta}_{ik}^{\,*}} \\
\END(ik) \ar[r]^-{\bar{\eta}_{ik \,*}} & \HOM(ik,RQ(ik)) & }$$
Since $\bar{\eta}_{ik}^{\,*}$ is a surjective quasi-isomorphism, it follows from Lemma \ref{pullback} that $\END(ik)'$ is a dg algebra quasi-isomorphic to $\END(ik)$. 

In cohomology, we obtain a commutative diagram
$$
\xymatrix@C+7pt{
H^*\END(ik)' \ar[r]^-{H^*(p_2)} \ar[d]_{H^*(p_1)}^{\cong} & H^* \END(RQ(ik)) \ar[d]_{\cong}^{H^* (\bar{\eta}_{ik}^{\,*})}\\
H^* \END(ik) \ar[r]^-{H^*(\bar{\eta}_{ik\,*})}
& H^*(\HOM(ik,RQ(ik)) & }
$$
where the composition $$H^*(\bar{\eta}_{ik}^{\,*})^{-1} H^*(\bar{\eta}_{ik\,*})\colon \bfK(\Inj kG)(ik,ik)^* \to \bfK(\Inj kG)(RQ(ik),RQ(ik))^*$$ is given by applying the functor $RQ$ to a map $f \in  \bfK(\Inj kG)(ik,ik)^*$.  But this is up to isomorphism just the inclusion $\iota\colon H^*(G,k) \to \hat{H}^*(G,k)$.
\end{proof}
Observe that the map $\varphi\colon \END(ik)' \to \END(tk)$ is a monomorphism which moreover induces a monomorphism in cohomology.

Now we show that the relation between $\mu_G$ and $\hat \mu_G$ which we have observed in the examples is true in general.
\begin{prop}\label{classexttate}
In the Hochschild group $\HH^{3,-1}(H^*(G,k),\hat H^*(G,k))$, it holds
$$\iota \circ \mu_G = \hat\mu_G \circ \iota^{\otimes 3},$$
where $\iota\colon H^*(G,k) \to \hat{H}^*(G,k)$ is the canonical inclusion.
\end{prop}
\begin{proof}
We use the zigzag
$$\xymatrix{\END(ik) & \ar[l]_-{\rho}^-{\sim} \END(ik)' \ar[r]^-{\varphi} & \END(tk)}$$
of Theorem \ref{exttatemorphism}. By Proposition \ref{aboutm3}, we have

\begin{equation}\label{rho*}
H^*\rho \circ \mu_{\END(ik)'}=\mu_{\END(ik)}\circ (H^*\rho)^{\otimes 3} \textrm{\quad in \quad} \HH^{3,-1}(H^*\END(ik)',H^*\END(ik))
\end{equation}
and 
\begin{equation}\label{psi*}
H^*\varphi \circ \mu_{\END(ik)'}=\mu_{\END(tk)}\circ (H^*\varphi)^{\otimes 3}\textrm{\quad in \quad} \HH^{3,-1}(H^*\END(ik)',H^*\END(tk)).
\end{equation}
Hence in the Hochschild group $\HH^{3,-1}(H^*(G,k),\hat H^*(G,k))$, it holds
\begin{equation*}
\iota \circ \mu_{\END(ik)} = \mu_{\END(tk)} \circ \iota^{\otimes 3}.
\qedhere
\end{equation*}
\end{proof}

\begin{rem}
Even more is true: Since both $\rho$ and $\varphi$ induce monomorphisms in cohomology, there exist choices in defining the secondary multiplications  of $H^*\END(ik)$ and $H^*\END(tk)$ to obtain the equations \eqref{rho*} and \eqref{psi*} 
even on the level of $k$-linear maps, see Proposition \ref{H*mono}. With these choices, we have
$$\iota \circ m_3^{H^*(G,k)} = m_3^{\hat{H}^*(G,k)} \circ \iota^{\otimes 3} \textrm{\quad in \quad} \Hom_k^{-1}(H^*(G,k)^{\otimes 3}, \hat H^*(G,k)).$$
The secondary multiplication $m_3^{H^*(\bbZ_3,k)}$ we computed in Proposition \ref{ext3} and $m_3^{\hat H^*(\bbZ_3,k)}$ of Theorem \ref{tate3} satisfy this equation. However,  we cannot just restrict a \emph{fixed} $m_3^{\hat H^*(G,k)}$ to non-negative degrees to obtain $m_3^{H^*(G,k)}$.
\end{rem}

\subsection{Relating the global obstructions of $H^*(G,k)$ and $\hat{H}^*(G,k)$} \label{sec relation}
In all but one example 
we have considered in the first section, the Hochschild classes $\mu_G$ and $\hat \mu_G$ were either both trivial or both non-trivial. Now we are ready to give a general explanation for this fact. The main result of this chapter is
\begin{thm}\label{murelation}
Let $G$ be finite group, $k$ a field of characteristic $p>0$ and assume that $p$ divides the order of $G$.
If the Hochschild class $\hat{\mu}_G \in \HH^{3,-1}(\hat H^*(G,k))$ is trivial, then so is the Hochschild class $\mu_G \in \HH^{3,-1}(H^*(G,k))$. 
If the $p$-rank of $G$ equals one, then $\hat{\mu}_G$ is trivial if and only if $\mu_G$ is trivial. 
\end{thm}

\begin{proof}
Let $m\colon H^*(G,k)^{\otimes 3} \to H^*(G,k)$ be any $(3,-1)$-cocycle representing $\mu_G$.
In the Hochschild group $\HH^{3,-1}(H^*(G,k),\hat H^*(G,k))$, we have the equation
$$\iota \circ \mu_G = \hat\mu_G \circ \iota^{\otimes 3}$$
by Proposition \ref{classexttate}. Hence if the Hochschild class  $\hat{\mu}_G$ is trivial, then the $k$-linear map $\iota \circ m$ is a $(3,-1)$-Hochschild coboundary, say 
$$\iota \circ m = \delta g,$$
where $g \in \Hom_k^{-1}(H^*(G,k)^{\otimes 2}, \hat{H}^*(G,k))$.

If $g(1,1)=0$, then we actually have that  $g \in \Hom_k^{-1}(H^*(G,k)^{\otimes 2}, H^*(G,k))$  and 
$$m = \delta g.$$ 
Thus $\mu_G$ is trivial in this case.

If $g(1,1)\neq 0$, then we choose any map $u \in \Hom_k^{-1}(H^*(G,k),\hat{H}^*(G,k))$ satisfying $u(1)\neq 0$. Since $\hat{H}^{-1}(G,k)$ is one-dimensional by Tate duality (see Proposition \ref{Tateduality}), there exists an element $\kappa \in k$ such that $g(1,1)=\kappa \cdot u(1).$
Setting
$$g'=g-\kappa \cdot \delta u,$$
we obtain
$$\iota \circ m = \delta g',$$
and since $g'(1,1)=0$, we conclude that $\mu_G$ is trivial.

If the $p$-rank of $G$ equals one, then the inclusion $\iota\colon H^*(G,k) \to \hat{H}^*(G,k)$ is a flat epimorphism of rings by Lemma \ref{periodic-loc}. 
Moreover, it is induced by a zigzag of dg algebra morphisms by Theorem \ref{exttatemorphism}.
Thus we can apply Theorem  \ref{GammatutsAB} and conclude that the map
$\Gamma \colon \HH^{*,*}(H^*(G,k))  \to  \HH^{*,*}(\hat{H}^*(G,k))$ satisfies $\Gamma(\mu_G)=\hat{\mu}_G$. In particular, $\mu_G$ being trivial implies that $\hat{\mu}_G$ is trivial.
\end{proof}
Note that the second statement of the theorem is not true if $r_p(G)\ge 2$: If $G$ is the Klein four group, then $\mu_G$ is trivial but $\hat \mu_G$ is not, see Section \ref{secexamples}.

\begin{Rem}
In order to prove that $\mu_G$ is non-trivial in Proposition \ref{ext3}, Corollary \ref{groupquaternion} and Proposition \ref{group3Sylow}, we have shown the existence of a non-realisable module over $H^*(G,k)$. But in general, one cannot expect to have a non-realisable module whenever the global obstruction is non-trivial: 
 Benson, Krause and Schwede provide an example of a dg algebra $A$ such that the canonical class $\mu_A\in \HH^{3,-1}(H^*A)$ is non-trivial, but such that  
 \emph{all} $H^*A$-modules are realisable  \cite[Prop.\ 5.16]{BKS}.

However, in all known examples of non-trivial global obstructions for group or Tate cohomology, the existence of a non-realisable module could always be shown. Benson, Krause and Schwede use Massey products to show the existence of a non-realisable module over $\hat H^*(\bbZ_3,k)$ and $\hat H^*(\bbZ_2\times \bbZ_2,k)$ \cite[Exm.\ 6.7, Exm.\ 7.7]{BKS}.  Langer has used Matrix Massey products~\cite[Lemma 2.23]{L}  to construct a non-realisable module over $\hat H^*(Q_8,k)$. 

It is an open question whether for group and Tate cohomology, one can expect to have a non-realisable module whenever the global obstruction is non-trivial.
\end{Rem}


\end{document}